\documentclass[english]{amsart}

\title{The \texorpdfstring{$R$}{R}-matrix formalism for quantized enveloping algebras}
\author[S. Gautam]{Sachin Gautam}
\address{Department of Mathematics, The Ohio State University}
\email{gautam.42@osu.edu}

\author[M. Rupert]{Matthew Rupert}
\address{Department of Mathematics and Statistics, Texas A\&M University Corpus Christi}
\email{matthew.rupert@tamucc.edu}

\author[C. Wendlandt]{Curtis Wendlandt}
\address{Department of Mathematics and Statistics, University of Saskatchewan}
\email{wendlandt@math.usask.ca}


\usepackage{amsmath,amssymb} 
\usepackage{mathrsfs} 
\usepackage[shortalphabetic]{amsrefs} 
\usepackage[usenames,dvipsnames]{xcolor} 
\usepackage{bbm} 
\usepackage[scr=boondoxo]{mathalfa} 
\usepackage{dsfont} 
\usepackage[
	colorlinks,
	plainpages,
	citecolor=red!50!black,
    filecolor=Darkgreen,
    linkcolor=blue!40!black,
    urlcolor=cyan!50!black!90]{hyperref} 
\usepackage{enumitem}
\usepackage{version} 

\usepackage{babel}  
\usepackage{csquotes}
\MakeOuterQuote{"}
\usepackage{tikz-cd} 
\usepackage{accents} 
\usepackage{stmaryrd} 

\newtheorem{theorem}{Theorem}[section]
\newtheorem{proposition}[theorem]{Proposition}
\newtheorem{corollary}[theorem]{Corollary}
\newtheorem{lemma}[theorem]{Lemma}

\newtheorem{theoremintro}{Theorem}

\theoremstyle{definition}
\newtheorem{definition}[theorem]{Definition}
\newtheorem{remark}[theorem]{Remark}

\newcommand{\wh}{\widehat}

\newcommand{\id}{\mathrm{id}}

\newcommand{\End}{\mathrm{End}}

\newcommand{\Hom}{\mathrm{Hom}}

\newcommand{\Ker}{\mathrm{Ker}}

\newcommand{\iso}{\xrightarrow{\,\smash{\raisebox{-0.5ex}{\ensuremath{\scriptstyle\sim}}}\,}}

\newcommand{\into}{\hookrightarrow}
\newcommand{\onto}{\twoheadrightarrow}

\makeatletter  
\newcommand{\sbullet}{%
  \hbox{\fontfamily{lmr}\fontsize{.4\dimexpr(\f@size pt)}{0}\selectfont\textbullet}}

\makeatother

\newcommand{\mfa}{\mathfrak{a}}

\newcommand{\mfb}{\mathfrak{b}}

\newcommand{\mfg}{\mathfrak{g}}
\newcommand{\mfh}{\mathfrak{h}}

\newcommand{\mfn}{\mathfrak{n}}

\newcommand{\mfp}{\mathfrak{p}}

\newcommand{\mfq}{\mathfrak{q}}

\newcommand{\mfz}{\mathfrak{z}}
\newcommand{\mfZ}{\mathfrak{Z}}

\newcommand{\mfgl}{\mathfrak{g}\mathfrak{l}}
\newcommand{\mfsl}{\mathfrak{s}\mathfrak{l}}

\newcommand{\mcA}{\mathcal{A}}

\newcommand{\mcI}{\mathcal{I}}
\newcommand{\mcJ}{\mathcal{J}}

\newcommand{\mcR}{\mathcal{R}}

\newcommand{\mcV}{\mathcal{V}}
\newcommand{\mcW}{\mathcal{W}}

\newcommand{\mbI}{\mathbf{I}}

\newcommand{\mbL}{\mathbf{L}}

\newcommand{\mbT}{\mathbf{T}}

\newcommand{\mbbL}{\mathbb{L}}

\newcommand{\C}{\mathbb{C}}

\newcommand{\Z}{\mathbb{Z}}
\newcommand{\N}{\mathbb{N}}

\newcommand{\mrL}{\mathrm{L}}
\newcommand{\mrR}{\mathrm{R}}
\newcommand{\mrT}{\mathrm{T}}

\newcommand{\msA}{\mathsf{A}}

\newcommand{\msd}{\mathsf{d}}

\newcommand{\msF}{\mathsf{F}}

\newcommand{\msG}{\mathsf{G}}

\newcommand{\msH}{\mathsf{H}}

\newcommand{\msL}{\mathsf{L}}

\newcommand{\msQ}{\mathsf{Q}}

\newcommand{\msR}{\mathsf{R}}

\newcommand{\msS}{\mathsf{S}}

\newcommand{\msT}{\mathsf{T}}

\newcommand{\msX}{\mathsf{X}}

\newcommand{\msz}{\mathsf{z}}
\newcommand{\msZ}{\mathsf{Z}}

\newcommand{\eps}{\epsilon}
\newcommand{\veps}{\varepsilon}

\allowdisplaybreaks
\numberwithin{equation}{section}

\usepackage{calligra}
\DeclareMathAlphabet{\mathcalligra}{T1}{calligra}{m}{n}
\DeclareFontShape{T1}{calligra}{m}{n}{<->s*[2.2]callig15}{}
\newcommand{\scriptr}{\mathcalligra{r}\,}

\definecolor{darkgreen}{RGB}{28,133,26}
\definecolor{darkblue}{rgb}{0.2, 0.2, 0.6}
\newcommand{\Uhg}{U_{\hbar}\mfg}
\newcommand{\Uh}[1]{U_\hbar{#1}}
\newcommand{\UR}[1]{\mathrm{U}_{\mathrm{R}}({#1})}
\newcommand{\sbinom}[2]{\begin{bmatrix}{#1}\\{#2}\end{bmatrix}}
\newcommand{\Casimir}{\mathfrak{c}}
\newcommand{\Root}{\mathsf{\Phi}}
\usepackage[T1]{fontenc}
\usepackage{graphicx}
\usepackage{upgreek}
\def\chk#1{#1^{\smash{\scalebox{.7}[1.4]{\rotatebox{90}{\textnormal{\guilsinglleft}}}}}}
\begin{document}

\begin{abstract}
Let $U_\hbar\mathfrak{g}$ denote the Drinfeld--Jimbo quantum group associated to a simple Lie algebra $\mathfrak{g}$. We apply a modification of the $R$-matrix construction for quantum groups to the evaluation of the universal $R$-matrix of $U_\hbar\mathfrak{g}$ on any of its finite-dimensional representations. This produces a quantized enveloping algebra $\mathrm{U}_\mathrm{R}(\mathfrak{g})$ whose definition is given in terms of two generating matrices satisfying variants of the well-known $RLL$ relations. We prove that $\mathrm{U}_\mathrm{R}(\mathfrak{g})$ is isomorphic to the tensor product of the quantum double of the Borel subalgebra $U_\hbar\mathfrak{b}\subset U_\hbar\mathfrak{g}$ and a quantized polynomial algebra encoded by the space of $\mathfrak{g}$-invariants associated to the semiclassical limit $V$ of the underlying finite-dimensional representation. Using this description, we characterize $U_\hbar\mathfrak{g}$ and the quantum double of $U_\hbar\mathfrak{b}$ as Hopf quotients of $\mathrm{U}_\mathrm{R}(\mathfrak{g})$ and as fixed-point subalgebras with respect to certain natural automorphisms. As an additional corollary, we deduce that $\mathrm{U}_\mathrm{R}(\mathfrak{g})$ is quasitriangular precisely when $V$ is multiplicity free.
\end{abstract}

\maketitle

{\setlength{\parskip}{0pt}
\setcounter{tocdepth}{1} 
\tableofcontents
}

\section{Introduction} \label{sec:I}


Let $\mfg$ be a simple Lie algebra over the complex numbers, and let $\Uh{\mfg}$ be its standard quantization over $\C[\![\hbar]\!]$, as first defined in generality by Drinfeld \cites{DR,DrQG} and Jimbo \cites{Jimbo85}; see Section \ref{ssec:Uhg}. An important feature of $\Uh{\mfg}$ is that it is quasitriangular, and thus possesses a universal $R$-matrix which can be evaluated on the tensor square of any of its finite-dimensional representations $\mcV$ (see Section \ref{ssec:Uhg-fdreps}) to produce a solution $\mrR$ of the quantum Yang--Baxter equation associated to $\mcV$. In this article, we address the problem of rebuilding $\Uhg$ uniformly from only this data. 

This problem has a storied history closely tied to the remarkable foundational paper \cite{FRT} of Faddeev, Reshetikhin and Takhtajan, and the work of Faddeev and his contemporaries on the quantum inverse scattering method that preceded it; see \cites{FST79, FT79, Faddeev80, KuSk-82} in addition to \cite{Majid-book}*{\S4.4} or \cite{CPBook}*{\S7.5}, for example. To elaborate, the article \cite{FRT} laid the foundation for a general theory of quantum groups centered around the two matrix equations 
\begin{equation*}
\mrR L_1 L_2= L_2 L_1 \mrR \quad \text{ and }\quad \mrR_{12}\mrR_{13}\mrR_{23}=\mrR_{23}\mrR_{13}\mrR_{12},
\end{equation*}
which are the algebraic relations at the heart of the quantum inverse scattering method. The latter equation is the aforementioned quantum Yang--Baxter equation, while the former is often taken to be a defining relation for an algebra whose generators are encoded as the coefficients of the matrix $L$. We refer the reader to \cite{FRT} and \cite{DrQG}*{\S10-11}, for instance, for the general meaning of these equations; our use of them shall be made transparent in Sections \ref{ssec:I-results-a}, \ref{ssec:Notation} and \ref{ssec:URg-def}. 

The techniques developed in \cite{FRT}, which are now collectively referred to as the $R$-\textit{matrix} (or \textit{FRT}) \textit{formalism} for quantum groups, output several examples of concrete, well-known quantum algebras and have played an instrumental role in the development of the representation theory of quantum groups over the last four decades; we refer the reader to the textbooks \cite{Majid-book}*{\S4}, \cite{CPBook}*{\S7}, \cite{KS-book}*{Parts II--III} in addition to \cite{MoBook} for an extensive overview. One of the key points stressed in \cite{FRT} is that this formalism is compatible with Drinfeld's theory of quantized enveloping algebras, in that it can be applied to recover the quantum groups $\Uhg$ introduced by Drinfeld and Jimbo. To illustrate this concretely, a solution to the problem of rebuilding $\Uhg$ from $\mrR$ and $\mcV$ (as in the first paragraph of this introduction) was outlined in the special case where $\mfg$ is of classical type and $\mcV$ is taken to be the vector representation of $\Uhg$; see Theorems 12 and 18 of \cite{FRT}, in addition to Theorem 33 of \cite{KS-book} and Section \ref{sec:Ex} below. Similar results have since been established in other isolated cases; for instance, in \cite{Sasaki}*{Thm.~11} this was extended to the case where $\mfg$ is of type $\mathsf{G}_2$ and $\mcV$ is a quantization of its seven dimensional fundamental representation.

In this article, we further intertwine the $R$-matrix techniques originating from the quantum inverse scattering method with the theory of quantized enveloping algebras in order to provide a uniform solution to the above reconstruction problem for any $\mfg$ and non-trivial finite-dimensional representation $\mcV$ of $\Uhg$.

\subsection{The quantum algebra $\UR{\mfg}$}\label{ssec:I-results-a} 

We now turn towards providing a more detailed outline of our results. Let $\mcV$ be a finite-dimensional representation  of $\Uhg$ (see Sections \ref{ssec:Uhg-fdreps}) with the property that its semiclassical limit $V=\mcV/\hbar\mcV$ is faithful as a representation of the Lie algebra $\mfg$ over $\C$. 

In Section \ref{ssec:URg-def}, we apply the approach of Faddeev, Reshetikhin and Tahktajan  to introduce a $\C[\![\hbar]\!]$-algebra $\UR{\mfg}$ associated to the evaluation $\mrR$ of the universal $R$-matrix of $\Uhg$ on the tensor square of $\mcV$. The defining relations of $\UR{\mfg}$ are given in terms of two elements $\mrT^+$ and $\mrT^-$ of $\End(V)\otimes_\C \UR{\mfg}$ and are such that the auxiliary matrices $\mrL^\pm=I+\hbar \mrT^\pm$ satisfy the familiar identities
\begin{equation*}
\mrR \mrL_1^\pm \mrL_2^\pm = \mrL_2^\pm\mrL_1^\pm\mrR \quad \text{ and }\quad \mrR \mrL_1^+ \mrL_2^- = \mrL_2^-\mrL_1^+\mrR,
\end{equation*}
in addition to two natural relations encoded by the weight space decomposition of the $\mfg$-module $\mfgl(V)$; see Definition \ref{D:URg}. The algebra $\UR{\mfg}$ plays a central role in this article. As is to be expected from the literature, it has several remarkable properties. For instance, it comes with a large collection of natural automorphisms and admits a Hopf algebra structure which is remarkably simple to describe in terms of $\mrT^\pm$ or $\mrL^\pm$; see Proposition \ref{P:aut} and Theorem \ref{T:URg-main}.

Let us now explain how $\UR{\mfg}$ relates to the quantized enveloping algebra $\Uhg$.
Let $\mfgl(V)^\mfg\subset \mfgl(V)$ be the Lie subalgebra of $\mfg$-invariants, which we view as a Lie bialgebra equipped with trivial Lie cobracket. Let $\mfz_V^+=(\mfgl(V)^\mfg)^\ast$ denote its Lie bialgebra dual, which is commutative as a Lie algebra, but in general has a non-trivial coalgebra structure. In Section \ref{ssec:quant-z}, we follow a standard procedure to introduce a quantization $\msS_\hbar(\mfz_V^+)$ of $\mfz_V^+$, which coincides with the trivial deformation $\msS(\mfz_V^+)[\![\hbar]\!]$ of the symmetric algebra on $\mfz_V^+$ as an algebra. We then prove the following theorem (see Theorem \ref{T:URg-main}), where $\msS_\hbar(\mfz_V^-)$ is the co-opposite Hopf algebra to $\msS_\hbar(\mfz_V^+)$. 
\begin{theoremintro}\label{T:intro}
There is an isomorphism of topological Hopf algebras 
\begin{equation*}
\Upsilon:\UR{\mfg}\iso D(\Uh{\mfb})\otimes \msS_\hbar(\mfz_V^+)\otimes \msS_\hbar(\mfz_V^-),
\end{equation*}
where $D(\Uh{\mfb})$ is the quantum double of the Borel subalgebra $\Uh{\mfb}\subset \Uhg$.
\end{theoremintro}
Here the symbol $\otimes$ denotes the topological tensor product of $\C[\![\hbar]\!]$-modules (see Section \ref{ssec:h-adic}) -- 
we refer the reader to Theorem \ref{T:URg-main} for the precise statement of the theorem, including the definition of $\Upsilon$.
%
%
 Its proof makes crucial use of Theorem \ref{T:g-main}, established earlier in Section \ref{sec:g-r}, which shows that the Lie bialgebra $\mfg_\scriptr$ quantized by $\UR{\mfg}$ decomposes as $\mfg_\scriptr \cong D(\mfb)\oplus \mfz_V^+\oplus \mfz_V^-$, where $D(\mfb)$ is the Drinfeld double of the Borel subalgebra $\mfb$ of $\mfg$.
\subsection{Recovering $\Uhg$}\label{ssec:I-results-b}
Theorem \ref{T:intro} has several consequences and, in particular,  yields two uniform solutions to the reconstruction problem for $\Uhg$ stated at the beginning of this introduction.
 The first solution describes $\Uhg$ as a Hopf algebra quotient of $\UR{\mfg}$ by the ideal generated by a family of distinguished central elements. Namely, if $\mrL_0^\pm$ denotes the zero weight component of $\mrL^\pm$ then $\mrL_0^+\, \mrL_0^-$ has coefficients in the center of $\UR{\mfg}$ and decomposes as a product of central factors
\begin{gather*}
\mrL_0^+\,\mrL_0^-=q^\Theta \cdot (I+\hbar \Upsigma^+) \cdot (I+\hbar \Upsigma^-)^{-1}.
\end{gather*}
The element $q^\Theta=\exp(\frac{\hbar}{2}\Theta)$ is defined in Lemma \ref{L:Theta}, while the matrices  $\Upsigma^+$ and $\Upsigma^-$ lay in  $\mfgl(V)^\mfg \otimes_\C \msZ(\UR{\mfg})$ and  appear earlier in Theorem \ref{T:URg->DUhb} -- their coefficients are computed explicitly in Corollary \ref{C:Z-compute} in the case where $V$ is multiplicity free; see also Remarks \ref{R:weight-basis} and 
\ref{R:zero-weight}.
It is shown in Theorem \ref{T:URg->Uhg} that $\Uhg$ is isomorphic to the quotient of $\UR{\mfg}$ by the (Hopf) ideal generated by the coefficients of $\Theta$, $\Upsigma^+$ and $\Upsigma^-$.

The second solution, also given in Theorem \ref{T:URg->Uhg}, identifies $\Uhg$ with the subalgebra  of $\UR{\mfg}$ consisting of all elements fixed by a natural family  of automorphisms indexed by $\mfh\times \mathrm{GL}_I(V)^\mfg\times \mathrm{GL}_I(V)^\mfg$, where $\mfh$ is the Cartan subalgebra of $\mfg$ and $\mathrm{GL}_I(V)^\mfg$ is the group of $\mfg$-invariant elements in $I+\hbar \mfgl(V)[\![\hbar]\!]$. Namely, for each triple $(h,C^+,C^-)$ belonging to this direct product, there is an automorphism of $\UR{\mfg}$ uniquely determined by 
\begin{equation*}
\mrL^\pm \mapsto q^{-\pi(h)/2} \mrL^\pm q^{-\pi(h)/2} \cdot C^\pm,
\end{equation*}
where $q=e^{\hbar/2}$ and $\pi:U(\mfg)\to \End(V)$ defines the action of $\mfg$ on $V$. The Drinfeld--Jimbo algebra $\Uhg$ is then isomorphic to the subalgebra of $\UR{\mfg}$ consisting of all elements stable under each of these automorphisms.

Both of the characterizations of $\Uhg$ provided by Theorem \ref{T:URg->Uhg} are deduced, in part, using similar characterizations for the quantum double $D(\Uh{\mfb})$ obtained in Theorem \ref{T:URg->DUhb} which follow more directly from Theorem \ref{T:intro}.

\subsection{Quasitriangularity}\label{ssec:I-results-c} 

As another consequence of Theorem \ref{T:intro}, we prove in Corollary \ref{C:URb} that there is an isomorphism of topological Hopf algebras 
\begin{equation*}
\UR{\mfb}\cong \Uh{\mfb}\otimes \msS_\hbar(\mfz_V^+),
\end{equation*}
where $\UR{\mfb}$ is the Hopf subalgebra of  $\UR{\mfg}$ topologically generated by the coefficients of $\mrT^+$. 
 In all of the examples which have appeared in the literature, $\UR{\mfg}$ is itself quasitriangular and can be recovered as the quantum double of $\UR{\mfb}$. Given these examples, and the nature of the definition of $\UR{\mfg}$, it is reasonable to postulate that $\UR{\mfg}$ is quasitriangular for any $\mrR$ and $\mcV$ and, in addition, can be recovered as the quantum double of $\UR{\mfb}$. However, Theorem \ref{T:intro} implies that both these assertions are false in general. 

A precise characterization of the quasitriangularity of $\UR{\mfg}$ is given in Theorem \ref{T:Quasi}: $\UR{\mfg}$ is quasitriangular if and only if the $\mfg$-module $V$ is multiplicity free or, equivalently, the space of invariants $\mfgl(V)^\mfg$ is abelian. In this case, $\UR{\mfg}$ is indeed isomorphic to the quantum double of $\UR{\mfb}$. 
\subsection{Remarks} A number of the results outlined above draw inspiration from historical developments in the literature on quantum groups of \textit{affine} type and, in particular, the structure theory of (extended) \textit{Yangians}; see \cites{AMR,DR,MNO,Ol,WRTT} in addition to the monograph \cite{MoBook}. Indeed, it was already known to Drinfeld in the 1980's that  the Yangian $Y_\hbar(\mfg)$ of an arbitrary simple Lie algebra $\mfg$ could be rebuilt from any of its finite-dimensional irreducible representations $\mcV$ by following a procedure not so different from that outlined in Sections \ref{ssec:I-results-a} and \ref{ssec:I-results-b}; see \cite{DR}*{Thm.~6}. Namely, $Y_\hbar(\mfg)$ can be realized as the quotient of the so-called \textit{extended Yangian} associated to $\mcV$ (the affine, degenerate analogue of $\UR{\mfg}$) by the ideal generated by the coefficients of a distinguished central series. This description was generalized to arbitrary finite-dimensional $\mcV$ in the third author's paper \cite{WRTT}, where the Yangian analogues of Theorems \ref{T:URg-main} and \ref{T:URg->Uhg} were also obtained in general: see Theorems 7.3 and 7.11 therein. This further extended earlier results obtained in the case where $\mcV$ is the vector representation of the Yangian of a classical Lie algebra; see \cites{AMR,MNO} and \cite{MoBook}*{\S1}. Our proof of Theorem \ref{T:intro} (\textit{i.e.,} Theorem \ref{T:URg-main}) is partly based on the techniques used to prove the main results of \cite{WRTT}.

For the closely related quantum affine algebras $U_\hbar(\hspace{.5pt}\wh{\mfg}\hspace{.5pt})$, some of the results summarized in the previous paragraph have been explored in general \cite{RS90}, though almost all attention has focused on the special cases alluded to above; see \cites{DF93,MiHa-L98,FiTs19, GoMo10, JLM20, JLM20-b, JLM20-c,MRS}, for example, in addition to the forthcoming paper \cite{GRWqLoop}. In fact, the main results of the present paper are  a prerequisite to developing the $R$-matrix formalism for quantum loop/affine algebras in general. 

Finally, it is worth emphasizing that the current article does not consider non-trivial specializations of the quantized enveloping algebras $\UR{\mfg}$ or $\Uhg$. Rather, our results are written so as to be compatible with the general theory of Lie bialgebra quantizations over $\C[\![\hbar]\!]$. This is an important point that does not play a significant role in the Yangian picture summarized above. That being said, it would be interesting to further develop the results of this paper so as to incorporate rational and integral forms of $\UR{\mfg}$ and $\Uh{\mfg}$. 

\subsection{Outline}
The article is, roughly speaking, divided into two parts. The first part, which consists of Sections \ref{sec:g} and \ref{sec:g-r}, is focused on establishing the classical variant of Theorem \ref{T:intro} (see Theorem \ref{T:g-main}). We begin in Section \ref{sec:g} by recalling some basic facts about the standard Lie bialgebra structure on $\mfg$ and its relation to the Drinfeld double of the Borel subalgebra $\mfb\subset \mfg$. The classical counterpart $\mfg_\scriptr$ of $\UR{\mfg}$ is then defined in Definition \ref{Def:gr} and it is proven in Theorem \ref{T:g-main} that it admits the Lie bialgebra decomposition $\mfg_\scriptr\cong D(\mfb)\oplus \mfz_V^+\oplus \mfz_V^-$.

The second part of the article is devoted to establishing the results outlined in detail in Sections \ref{ssec:I-results-a}--\ref{ssec:I-results-c} above. We begin in Section \ref{sec:Uhg} by surveying some of the key properties of $\Uhg$ and the quantum double $D(\Uh{\mfb})$ which will play a role throughout the remainder of the paper. This includes a brief overview of the general definitions from the theory of quantized enveloping algebras in Section \ref{ssec:h-adic}. In Section \ref{sec:URg}, we introduce the $R$-matrix algebra $\UR{\mfg}$, establish some of its basic properties (see Lemma \ref{L:URg-hom} and Proposition \ref{P:aut}) and then prove Theorem \ref{T:intro} and the main results outlined in Section \ref{ssec:I-results-b}; see Theorems \ref{T:URg-main}, \ref{T:URg->DUhb} and \ref{T:URg->Uhg}. The section concludes with a study of the quantum Borel subalgebras $\UR{\mfb}=\UR{\mfb^+}$ and $\UR{\mfb^-}$ of $\UR{\mfg}$ in Section \ref{ssec:URb}; see Corollaries \ref{C:URb} and \ref{C:URb->Ub}. In Section \ref{sec:Quasi}, we obtain the characterization of the quasitriangularity of $\UR{\mfg}$ sketched in Section \ref{ssec:I-results-c} and compute the central elements alluded to in Section \ref{ssec:I-results-b} explicitly in the case where $V$ is multiplicity free; see Theorem \ref{T:Quasi}, Corollary \ref{C:Z-compute}, and Remarks \ref{R:weight-basis} and \ref{R:zero-weight}. We conclude our paper in Section \ref{sec:Ex} by outlining how our results specialize to well-known constructions in the case where $\mcV$ is the vector representation of $\Uh{\mfsl_n}$. 

\subsection{Acknowledgments}
We are extremely grateful to Pavel Etingof for suggesting that some of the ideas from \cite{WRTT} could be applied to further develop the $R$-matrix formalism for the standard quantizations of semisimple Lie algebras. These preliminary remarks have motivated a large part of this article. We would also like to thank Andrea Appel and Jean Auger for their help in translating the article's abstract and title for the published version. SG was supported through the Simons foundation collaboration grant 526947. MR was supported by the Pacific Institute for the Mathematical Sciences (PIMS) Postdoctoral Fellowship program.
CW gratefully acknowledges the support of the Natural Sciences and Engineering Research Council of Canada (NSERC), provided via the Discovery Grants Program (Grant RGPIN-2022-03298 and DGECR-2022-00440).

\section{Recollections on \texorpdfstring{$\mfg$}{g}} \label{sec:g}

\subsection{The Lie algebra $\mfg$}\label{ssec:g}
Let $\mfg$ be a simple Lie algebra over the complex numbers and fix an invariant, non-degenerate, symmetric bilinear form $(\,,\,)$ on $\mfg$.   Let $\mfh\subset \mfg$ be a Cartan subalgebra with $\{\alpha_i\}_{i\in \mbI}\subset \mfh^\ast$ a basis of simple roots, and let  $\Root^+\subset \mfh^\ast$ be the associated set of positive roots of $\mfg$. The full set of roots $\Root^+\cup (-\Root^+)$ will be denoted by $\Root$.   We normalize $(\,,\,)$, if necessary, so that the square length of every short root is $2$. In addition, we let $\msQ_+=\bigoplus_{i\in\mbI}\Z_{\geq 0}\alpha_i$ denote the positive cone in the root lattice $\msQ=\Z \Root^+\subset \mfh^\ast$ of $\mfg$.   Let $(a_{ij})_{i,j\in \mbI}$ be the Cartan matrix of $\mfg$, so that $d_ia_{ij}=(\alpha_i,\alpha_j)$, where $\{d_i\}_{i\in \mbI}$ are the symmetrizing integers of $\mfg$, defined by $d_i=\frac{(\alpha_i,\alpha_i)}{2}$ for all $i\in \mbI$.

We shall write $\Omega_\mfh\in \msS^2(\mfh)\subset \mfh\otimes \mfh$ for the canonical element associated to $(\,,\,)|_{\mfh\otimes \mfh}$. If $\nu:\mfh\iso \mfh^\ast$ is the isomorphism determined by $(\,,\,)|_{\mfh\otimes \mfh}$, then we may define $\{h_i\}_{i\in \mbI}\subset \mfh$ by $h_i=\nu^{-1}(\alpha_i)$, and we have
\begin{gather*}
\Omega_\mfh=\sum_{i\in \mbI} \varpi_i^\vee \otimes h_i,
\end{gather*}
where $\{\varpi_i^\vee\}_{i\in \mbI}\subset \mfh$ are the fundamental coweights of $\mfg$, determined by $\alpha_j(\varpi_i^\vee)=\delta_{ij}$ for all $i,j\in \mbI$. Finally, for each root $\alpha\in \Root$, we let $\mfg_\alpha\subset \mfg$ denote the corresponding root space, so that 
\begin{equation*}
\mfg_\alpha=\{x\in \mfg: [h,x]=\alpha(h)x \quad \forall\; h\in \mfh\}
\end{equation*}
and $\mfg=\mfh\oplus \bigoplus_{\alpha\in \Root}\mfg_\alpha$. 

\subsection{Lie bialgebra structure}\label{ssec:D(b)}

Recall that a Lie bialgebra structure on a Lie algebra $\mfa$ is given by the additional data of a linear map
$\delta:\mfa\to \mfa\wedge \mfa$ satisfying the cocycle and co-Jacobi identities 
\begin{gather*}
\delta([x,y])=[\delta(x),\Delta(y)]+[\Delta(x),\delta(y)] \quad \forall \; x,y\in \mfa, \\ 
(\id +(1\,2\,3)+(1\,3\,2))\circ (\delta\otimes \id) \circ \delta=0, 
\end{gather*}
respectively, where $\Delta(x)=x\otimes 1 + 1\otimes x$ for all $x\in \mfa$. 

If $\mfa=(\mfa,[\,,\,]_\mfa, \delta_\mfa)$ is a finite-dimensional Lie bialgebra, then so is its dual $\mfa^\ast$, with bracket $[\,,\,]_{\mfa^\ast}$ and cobracket $\delta_{\mfa^\ast}$ given by the transposes $\delta_\mfa^t:\mfa^\ast \wedge \mfa^\ast \to \mfa^\ast$ and $[\,,\,]_{\mfa}^t:\mfa^\ast \to \mfa^\ast\wedge \mfa^\ast$, respectively.
A Lie bialgebra   $(\mfa,\delta)$ is said to be \textit{quasitriangular} if there is $r\in \mfa\otimes \mfa$ satisfying 
\begin{gather*}
\delta(x)=[x\otimes 1 + 1\otimes x,r]=[\Delta(x),r] \quad \forall \; x\in \mfa, \\
[r_{12},r_{13}]+[r_{12},r_{23}]+[r_{13},r_{23}]=0,\\
r+r_{21}\in (\mfa\otimes \mfa)^\mfa.
\end{gather*}
That is, $r$ is a coboundary structure of $(\mfa,\delta)$ satisfying the classical Yang--Baxter equation (it is an $r$-matrix),  whose symmetric part is $\mfa$-invariant. Here we are using the standard notation where $r_{ij}$ is the element of $U(\mfa)^{\otimes 3}$ (with $U(\mfa)$ the universal enveloping algebra of $\mfa$) whose $i$-th and $j$-th tensor factors are the first and second tensor factors of $r$, respectively, and the remaining factor is the unit $1\in U(\mfa)$. For instance, one has $r_{12}=r\otimes 1$ and $r_{23}=1\otimes r$. 

 The \textit{Drinfeld double} $D(\mfb)$ of a finite-dimensional Lie bialgebra $\mfb$ is a quasitriangular Lie bialgebra containing $(\mfb,\delta_\mfb)$ and $(\mfb^\ast, -\delta_{\mfb^\ast})$ as Lie sub-bialgebras, which can be realized on the space $\mfb\oplus \mfb^\ast$. It has cobracket $\delta$ given by $\delta=\delta_\mfb \oplus (-\delta_{\mfb^\ast})$, and Lie bracket determined by the cross relations
\begin{equation*}
[x,f]=\mathrm{ad}^\ast(x)(f)+f\otimes \id_\mfb \circ \delta(x) \quad \forall\; x\in \mfb\; \text{ and }\; f\in \mfb^\ast.
\end{equation*}
The associated $r$-matrix is given by the canonical element $r\in \mfb\otimes \mfb^\ast\subset D(\mfb)^{\otimes 2}$. Moreover, the natural pairing  $\mfb\otimes \mfb^\ast\to \C$ extends to a symmetric, non-degenerate, invariant bilinear form $\langle\,,\,\rangle$ on $D(\mfb)$ for which $\mfb$ and $\mfb^\ast$ are isotropic subspaces. Hence, the triple $(D(\mfb),\mfb,\mfb^\ast)$ forms a \textit{Manin triple}; that is, a triple $(\mfa,\mfp,\mfq)$ consisting of a Lie algebra $\mfa$ with a symmetric, non-degenerate invariant bilinear form, and two isotropic Lie subalgebras $\mfp$ and $\mfq$ for which $\mfa=\mfp\oplus \mfq$ as vector spaces. Every finite-dimensional Manin triple can be realized via the Drinfeld double construction, as above; see \cite{ES}*{\S4}, for instance.

The simple Lie algebra $\mfg$ admits a quasitriangular Lie bialgebra structure with $r$-matrix given by the element
\begin{equation}\label{half-Cas}
\scriptr=\sum_{\alpha\in \Root^+}x_\alpha^+\otimes x_\alpha^- + \frac{1}{2}\Omega_\mfh,
\end{equation}
where, for each $\alpha\in \Root^+$,  $x_\alpha^\pm\in \mfg_{\pm \alpha}$ are fixed root vectors satisfying $(x_\alpha^+,x_\alpha^-)=1$. In particular, $\scriptr+\scriptr_{21}$ is the Casimir tensor $\Omega\in (\mfg\otimes\mfg)^\mfg$.
This quasitriangular structure naturally arises by realizing $\mfg$ as a quotient of the Drinfeld double of its  positive Borel subalgebra. We now briefly review these facts, following \cite{ES}*{\S4.4}.

Consider the external direct sum of Lie algebras 
$\mfg\oplus \mfh$, and let
\begin{equation*}
\zeta:\mfh \into \mfg\oplus \mfh
\end{equation*}
be the natural inclusion; we shall refer to its image as the central copy of $\mfh$ in $\mfg\oplus \mfh$.
We equip $\mfg\oplus \mfh$ with the invariant, non-degenerate, symmetric bilinear form $\langle\,,\rangle$ defined by 
\begin{equation*}
\langle x_1+\zeta(h_1),x_2+\zeta(h_2)\rangle:=(x_1,x_2)-(h_1,h_2) \quad \forall\; x_i\in \mfg, \, h_i\in \mfh.
\end{equation*}
Next, we define $\mfn^\pm=\bigoplus_{\alpha\in\Root^+}\mfg_{\pm \alpha}\subset \mfg$ and set 
\begin{equation*}
\mfh^{\pm}=\mathrm{span}_\C\{h\pm \zeta(h)\,:\,h\in \mfh\}\subset \mfg\oplus \mfh.
\end{equation*}
Then the subspace $\mfh^\pm\oplus \mfn^\pm\subset \mfg\oplus\mfh$ is a Lie subalgebra canonically isomorphic to the Borel subalgebra $\mfb^\pm=\mfh\oplus \mfn^\pm$ of $\mfg$, and we have the vector space decomposition 
\begin{equation*}\label{decomp}
\mfg\oplus \mfh = (\mfh^+\oplus \mfn^+)\oplus (\mfh^-\oplus \mfn^-) \cong \mfb^+\oplus \mfb^-.
\end{equation*}
In what follows, we shall identify $\mfb^\pm$ with $\mfh^\pm \oplus \mfn^\pm$, viewed as a Lie subalgebra of $\mfg\oplus \mfh$. In addition, we set $h^\pm=h\pm \zeta(h)$ for each $h\in \mfh$. 
\begin{proposition}\label{P:b-dbl}
The triple $(\mfg\oplus\mfh, \mfb^+, \mfb^-)$ is a Manin triple. Consequently, $\mfb^+$ is a Lie bialgebra, and the canonical element 
\begin{equation*}
\scriptr^D=\sum_{\alpha\in\Root^+}x_\alpha^+\otimes x_\alpha^- + \frac{1}{2}\sum_{i\in \mbI} (\varpi_i^\vee)^+\otimes h_i^- \in \mfb^+\otimes \mfb^-
\end{equation*}
associated to  $\langle\,,\,\rangle$ defines a quasitriangular Lie bialgebra structure on $\mfg\oplus \mfh$ such that $D(\mfb^+)\cong \mfg\oplus \mfh$. 
\end{proposition}
Since the central copy of $\mfh$ in $D(\mfb^+)=\mfg\oplus \mfh$ is killed by the cobracket $\delta$, it defines a Lie bialgebra ideal in $D(\mfb^+)$, and hence the canonical Lie algebra homomorphism 
\begin{equation*}
\uppsi: D(\mfb^+)=\mfg\oplus \mfh \onto \mfg
\end{equation*} 
defines a quasitriangular Lie bialgebra structure on $\mfg$ with $r$-matrix $\scriptr=(\uppsi\otimes \uppsi)(\scriptr^D)$ given by \eqref{half-Cas}. The resulting Lie bialgebra structure is often called the \textit{standard structure} on $\mfg$. 

\subsection{Chevalley involution}\label{ssec:Db-Chev} 
Let $\omega\in \mathrm{Aut}(\mfg)$ denote the Chevalley involution of $\mfg$, uniquely determined by 
\begin{equation*}
\omega|_{\mfh}=-\id_\mfh \quad \text{ and }\quad \omega(x_i^\pm)=-x_i^\mp \quad \forall\; i\in \mbI,
\end{equation*}
 where $x_i^\pm:=x_{\alpha_i}^\pm\in \mfg_{\pm \alpha_i}$, with $x_\alpha^\pm$ as in \eqref{half-Cas}. Then $\dot\omega:=\omega\oplus \id_\mfh$ is a Lie algebra involution of $\mfg\oplus \mfh \cong D(\mfb^+)$ satisfying $\dot\omega(x_i^\pm)=-x_i^\mp$ and $\dot\omega(h^\pm)=-h^\mp$ for all $i\in \mbI$ and $h\in \mfh$. In particular, one has $(\dot\omega \otimes \dot\omega) (\scriptr^D)=\scriptr^D_{21}$. It follows that $\dot\omega$ is a Lie coalgebra anti-involution of $D(\mfb^+)$:
\begin{equation*}
(\dot\omega \otimes \dot\omega)\circ \delta = -\delta \circ \dot \omega. 
\end{equation*}
 We shall henceforth simply write $\omega$ for $\dot\omega$; as the latter restricts to the Chevalley involution of $\mfg$, this will not cause any ambiguity.
 
\section{The classical construction}\label{sec:g-r}

\subsection{Notation}\label{ssec:Notation}

Throughout Section \ref{sec:g-r}, we take $V$ to be a fixed finite-dimensional $\mfg$-module with associated  algebra homomorphism 
\begin{equation*}
\pi:U(\mfg)\to \End(V).
\end{equation*}
We further assume that $V$ has a non-trivial irreducible summand or, equivalently, that $V$ is a faithful representation of $\mfg$. 

In addition, we shall make use of the following notation, frequently employed in the literature (see, for example, \cite{MoBook}*{\S1.5}).
Let $\mathscr{U}$ be a unital, associative algebra over the complex numbers, and suppose that $n$ and $a$ are positive integers with $a\leq n$.  Let $\imath_{\mathscr{U},n}^{(a)}:\mathscr{U}\to \mathscr{U}^{\otimes n}$ denote the algebra homomorphism
\begin{equation*}
\imath_{\mathscr{U},n}^{(a)}(x)=1_\mathscr{U}^{\otimes (a-1)}\otimes x \otimes 1_\mathscr{U}^{\otimes (n-a)} \quad \forall \quad x\in \mathscr{U}.
\end{equation*}
Given an element $F\in \End(V)\otimes \mathscr{U}$, we then set 
\begin{equation*}
F_{i[j]}=(\imath^{(i)}_{\mathscr{A},n}\otimes \imath^{(j)}_{\mathscr{U},m})(F)\in \End(V)^{\otimes n} \otimes \mathscr{U}^{\otimes m},
\end{equation*}
where we have abbreviated $\mathscr{A}=\End(V)$, $m$ is any auxiliary positive integer, and $i$ and $j$ satisfy  $1\leq i\leq n$ and $1\leq j\leq m$. In the case where $n=1$, we will write $F_{[j]}$ for $F_{1[j]}$, and in the case where $m=1$ we shall write $F_i$ for $F_{i[1]}$. The positive integers $n$ and $m$ will always be clear from context. We shall sometimes apply the above notation when $F\in \End(V)\otimes \mfa$ for some complex Lie algebra $\mfa$. In this case, it is implicitly understood that $\mathscr{U}=U(\mfa)$, the universal enveloping algebra of $\mfa$.
\subsection{The $\mfg$-module $\mfgl(V)$}\label{ssec:gl(V)}

We view the Lie algebra $\mfgl(V)=\End(V)$ as a $\mfg$-module with action given by restricting the adjoint action of $\mfgl(V)$ on itself to $\mfg\cong \pi(\mfg)$. This coincides with the standard action of $\mfg$ on $\End (V)$ under the identification $\End(V)\cong V^\ast\otimes V$. The $\mfg$-module $\mfgl(V)$ then admits a weight space decomposition 
\begin{equation*}
\mfgl(V)=\bigoplus_{\mu\in \mfh^\ast}\mfgl(V)_\mu,
\end{equation*}
where the weight space $\mfgl(V)_\mu$ is given explicitly by
\begin{align*}
\mfgl(V)_\mu&=\{X\in \mfgl(V): [\pi(h),X]=\mu(h)X\quad\forall\; h\in \mfh\}=\bigoplus_{\lambda\in \mfh^\ast}\Hom(V_\lambda,V_{\lambda+\mu}).
\end{align*}
Given an arbitrary $\C$-vector space $W$ and an element $X\in \End (V) \otimes W$, we define 
\begin{equation*}
X_\mu:=(\mathbf{1}_\mu\otimes \mathrm{id})X \in \mfgl(V)_\mu \otimes W \quad \forall \; \mu\in \mfh^\ast, 
\end{equation*}
 where $\mathbf{1}_\mu:\mfgl(V)\to \mfgl(V)_\mu$ is the natural projection onto the $\mu$-weight space of the $\mfg$-module $\mfgl(V)$.

\subsection{The Lie algebra $\mathfrak{g}_\scriptr$}\label{ssec:g-matrix-def}

We now introduce the classical structure at the heart of the current section: the $r$-matrix Lie algebra $\mfg_\scriptr$. Let $\{v_i\}_{i\in \mcI}\subset V$ be any fixed basis of $V$, and let $\{E_{ij}\}_{i,j\in \mcI}\subset \End(V)$ denote the associated elementary matrices, defined by $E_{ij}v_k=\delta_{jk}v_i$ for all $i,j,k\in \mcI$. 
\begin{definition}\label{Def:gr}
Let $\mfg_\scriptr$ be the Lie algebra generated by $\{l_{ij}^\pm\}_{i,j\in \mcI}$, subject to 
\begin{gather}
L^\pm_{\lambda}=0 \quad \forall\; \lambda\in \dot{\mathsf{Q}}_{\mp}, \label{gr:triangle}\\
[L_1^\pm,L_2^\pm]=-[\scriptr_\pi,L_1^\pm+L_2^\pm], \label{gr:LpmLpm}\\
[L_1^+,L_2^-]=-[\scriptr_\pi,L_1^++L_2^-]\label{gr:L+L-}, 
\end{gather}
where $\scriptr_\pi=(\pi\otimes \pi)\scriptr$ and $L^\pm$ is the generating matrix 
\begin{equation*}
L^\pm=\sum_{i,j\in \mcI} E_{ij}\otimes l_{ij}^\pm \in \End(V)\otimes \mfg_\scriptr.
\end{equation*}
\end{definition}
Here we have set  $\dot{\msQ}_{\pm}=\msQ_\pm \setminus\{0\}$, where $\msQ_-=-\msQ_+$. The relation \eqref{gr:triangle}, which we call the \textit{triangularity relation} for $L^\pm$,  holds in $\End(V)\otimes \mfg_\scriptr$, while \eqref{gr:LpmLpm} and \eqref{gr:L+L-} are relations in the space $\End(V)^{\otimes 2}\otimes \mfg_\scriptr$ and we have written $\scriptr_\pi$ for $(\scriptr_\pi)_{12}$. 
\begin{remark}\label{R:g_r-depends}
The definition of $\mfg_\scriptr$ of course depends on the underlying representation $V$ and, more precisely, the evaluation $\scriptr_\pi\in \End(V)^{\otimes 2}$ of $\scriptr$ on $V\otimes V$. However, since $\pi$ is fixed throughout this section, the notation $\mfg_\scriptr$ will not cause any ambiguity. 
\end{remark}
\begin{remark}
We will see below that Equation \eqref{gr:triangle} indeed implies that
$L^+$ and $L^-$ are upper and lower triangular, respectively, in the sense that 
\begin{equation*}
L^\pm \in (\mfgl(V)_0\otimes \mfg_{\scriptr})\oplus \bigoplus_{\alpha\in \Root^+}(\mfgl(V)_{\pm \alpha}\otimes \mfg_{\scriptr}).
\end{equation*}
In fact, we will establish a stronger assertion in Lemma \ref{L:Z-wtgr}.
\end{remark}
%
%
\subsection{Central elements and triangularity}\label{ssec:g-central}

Following \cite{WRTT}, we consider the $\mfg$-module decomposition 
\begin{equation*}
\mfgl(V)=\pi(\mfg)\oplus W=\mathrm{ad}(\mfg)\oplus (\mfgl(V)^\mfg\oplus M),
\end{equation*}
where $W=\mfgl(V)^\mfg\oplus M$ with $M$ a $\mfg$-submodule of $\mfgl(V)$ complimentary to $\pi(\mfg)\oplus \mfgl(V)^\mfg$, and $\mathrm{ad}(\mfg)$ is the adjoint representation of $\mfg$ realized on the space $\pi(\mfg)\cong \mfg$. Note that the submodule of $\mfg$-invariants $\mfgl(V)^\mfg$ is precisely the intertwiner space $\End_\mfg(V)$. 
The generating matrices $L^\pm$ therefore decompose uniquely as  
\begin{equation*}
L^\pm=\msL^\pm + K^\pm \quad \text{ with }\quad \msL^\pm\in\pi(\mfg)\otimes \mfg_\scriptr,\quad K^\pm\in W\otimes \mfg_\scriptr. 
\end{equation*}

Let $\Casimir\in U(\mfg)$ denote the standard quadratic Casimir element, obtained by applying the  multiplication map $U(\mfg)\otimes U(\mfg)\to U(\mfg)$ to $\Omega=\scriptr+\scriptr_{21}$.  We shall write $\Casimir(L^\pm)$ for $(\Casimir\otimes \mathrm{id})(L^\pm)$, where the action of $\mfg$ (and thus $U(\mfg)$) on $\mfgl(V)$ is as in Section \ref{ssec:gl(V)}. Let $\kappa$ denote the eigenvalue of $\Casimir$ on the adjoint representation of $\mfg$, so that 
\begin{equation*}
\Casimir(x)=\kappa\cdot x \quad \forall \; x\in \pi(\mfg).
\end{equation*}
The following lemma provides an analogue of \cite{WRTT}*{Lem.~4.2}. 
\begin{lemma}\label{L:Z-wtgr}
\leavevmode 
\begin{enumerate}[font=\upshape]
\item\label{wtgr:Z1} The coefficients of $K^\pm$ belong to the center of $\mfg_\scriptr$. 
\item\label{wtgr:Z2} We have $\Casimir(K^\pm)=0$. In particular, 
$K^\pm \in \mfgl(V)^\mfg\otimes \mfg_\scriptr$ and 
\begin{equation*}
L^\pm \in (\pi(\mfb^\pm)\oplus\mfgl(V)^\mfg)\otimes \mfg_\scriptr. 
\end{equation*}
\item\label{wtgr:Z3} We have $L^\pm-\frac{1}{\kappa}\Casimir(L^\pm)=K^\pm$. 
\end{enumerate}
\end{lemma}
\begin{proof} 
We will prove the lemma by projecting the defining relations \eqref{gr:LpmLpm} and \eqref{gr:L+L-} onto the summands of $\End(V)^{\otimes 2}\otimes \mfg_\scriptr$ with respect to the decomposition 
\begin{equation*}
\End(V)^{\otimes 2}\otimes \mfg_\scriptr=(\pi(\mfg)^{\otimes 2}\otimes \mfg_\scriptr) \oplus (\pi(\mfg)\otimes W \otimes \mfg_\scriptr) \oplus (W\otimes \pi(\mfg) \otimes \mfg_\scriptr) \oplus (W^{\otimes 2}\otimes \mfg_\scriptr). 
\end{equation*}
Substituting $L^\pm=\msL^\pm+K^\pm$ into the relation \eqref{gr:LpmLpm}, and projecting onto each of the four above summands yields the four relations 
\begin{equation}
\begin{gathered}\label{components}
[\msL_1^\pm,\msL_2^\pm]=-[\scriptr_\pi,\msL_1^\pm+\msL_2^\pm],\\
[K_1^\pm,K_2^\pm]=0, \; [K_1^\pm,\msL_2^\pm]=-[\scriptr_\pi,K_1^\pm], \; [\msL_1^\pm,K_2^\pm]=-[\scriptr_\pi,K_2^\pm].
\end{gathered}
\end{equation}
Applying the permutation operator $(1\,2)$ to $[K_1^\pm,\msL_2^\pm]=-[\scriptr_\pi,K_1^\pm]$ yields 
\[
[\msL_1^\pm,K_2^\pm]=[\scriptr_\pi^{21},K_2^\pm],
\]
and thus we must have $[\scriptr_\pi^{21},K_2^\pm]=-[\scriptr_\pi,K_2^\pm]$. Equivalently, 
\begin{equation*}
[\Omega_\pi,K_2^\pm]=0=[\Omega_\pi,K_1^\pm],
\end{equation*}
where we recall that $\Omega=\scriptr+\scriptr_{21}\in \mfg\otimes \mfg$ is the Casimir tensor and $\Omega_\pi=(\pi\otimes \pi)\Omega$. 
One concludes from the above relation exactly as in the proof of \cite{WRTT}*{Lem.~4.2} that $\Casimir(K^\pm)=0$, and consequently that $K^\pm \in \mfgl(V)^\mfg\otimes \mfg_\scriptr$ and
\begin{equation*}
L^\pm=\msL^\pm+K^\pm \in (\pi(\mfg)\oplus\mfgl(V)^\mfg)\otimes \mfg_\scriptr. 
\end{equation*} 
The stronger assertion that $L^\pm\in (\pi(\mfb^\pm)\oplus\mfgl(V)^\mfg)\otimes \mfg_\scriptr$ is an application of the triangularity relation \eqref{gr:triangle}, the decomposition $\pi(\mfg)=\pi(\mfn^\mp)\oplus \pi(\mfb^\pm)$, and that $\mfgl(V)^\mfg$ is contained in the zero weight space $\mfgl(V)_0$ of $\mfgl(V)$. 
This proves Part \eqref{wtgr:Z2} of the Lemma, from which Part \eqref{wtgr:Z3} follows immediately.

We are thus left to establish Part \eqref{wtgr:Z1} of the lemma, which asserts that the coefficients of $K^\pm$ belong to the center of $\mfg_\scriptr$. We first observe that, since $K^\pm\in \mfgl(V)^\mfg\otimes\mfg_\scriptr$, we have  $[\scriptr_\pi,K^\pm_1]=0=[\scriptr_\pi,K^\pm_2]$. Hence, \eqref{components} implies that 
\begin{equation*}
[K_1^\pm,\msL_2^\pm]=0=[\msL_1^\pm,K_2^\pm]. 
\end{equation*} 
We are thus left to show that  $[L_1^\pm,K_2^\mp]=0$.
Using again that $[\scriptr_\pi,K^\pm_i]=0$, we deduce that the relation
\eqref{gr:L+L-} is equivalent to 
\begin{equation*}
[\msL_1^+,\msL_2^-]+[\msL_1^+,K_2^-]+[K_1^+,\msL_2^-]+[K_1^+,K_2^-]=-[\scriptr_\pi,\msL_1^++\msL_2^-]. 
\end{equation*}
Projecting onto the four summands of $\End (V)^{\otimes 2}\otimes \mfg_\scriptr$ therefore yields the relations 
\begin{gather*}
[\msL_1^+,\msL_2^-]=-[\scriptr_\pi,\msL_1^++\msL_2^-],\\
[K_1^+,K_2^-]=0,\; [\msL_1^+,K_2^-]=0, \; [K_1^+,\msL_2^-]=0. 
\end{gather*}
The second line above yields the desired result. \qedhere
\end{proof}
\begin{remark}\label{R:gL-sub}
In the process of proving the above lemma we have shown that $\msL^\pm$ satisfy the identities
\begin{gather*}
\msL^\pm_{\lambda}=0 \quad \forall\; \lambda\in \dot{\mathsf{Q}}_{\mp},\\
[\msL_1^\pm,\msL_2^\pm]=-[\scriptr_\pi,\msL_1^\pm+\msL_2^\pm],\quad [\msL_1^+,\msL_2^-]=-[\scriptr_\pi,\msL_1^++\msL_2^-],\\
\kappa\msL^\pm=\Casimir(\msL^\pm).
\end{gather*}
It will be established in Corollary \ref{C:gr->D(b)} that these are in fact defining relations for the Lie algebra $D(\mfb^+)$.
\end{remark}
%
%

\subsection{The Lie bialgebra dual to $\mfgl(V)^\mfg$}\label{ssec:g-inv-dual}

Consider the Lie subalgebra $\mfgl(V)^\mfg\subset \mfgl(V)$. We may view it as a Lie bialgebra, equipped with trivial Lie cobracket $\delta\equiv 0$. Let $\mfz_V^+=(\mfgl(V)^\mfg)^\ast$ denote its Lie bialgebra dual; it is a commutative Lie algebra, but when $V$ is not a direct sum of distinct irreducible representations it has a non-trivial cobracket $\delta$. Indeed, if  $\mfZ^+\in \mfgl(V)^\mfg\otimes \mfz_V^+$ is the canonical element, then 
\begin{equation*}
\delta(\mfZ^+)=[\mfZ_{[1]}^+,\mfZ_{[2]}^+]\in \mfgl(V)^\mfg\otimes (\mfz_V^+\wedge \mfz_V^+)
\end{equation*}
where $\delta$ is applied to the second tensor factor of $\mfZ^+$. We let $\mfz_V^-=(\mfz_V^+,-\delta)$ denote the opposite Lie bialgebra to $\mfz_V^+$. Let $\mfZ^-\in \mfgl(V)^\mfg\otimes \mfz_V^-$ be the canonical element, which is nothing but $\mfZ^+$, viewed as an element of $\mfgl(V)^\mfg\otimes \mfz_V^-$. Note that $\mfz_V^+$ and $\mfz_V^-$ are isomorphic as Lie bialgebras, with an isomorphism $\mfz_V^+\iso \mfz_V^-$ given by $\mfZ^+\mapsto -\mfZ^-$.  
%

\subsection{From $\mathfrak{g}_\scriptr$ to the double of $\mfb^+$}\label{ssec:g-main}
Let us now turn to relating $\mfg_\scriptr$ to the Lie bialgebra double $D(\mfb^+)$ of the Borel subalgebra $\mfb^+$ (see Section \ref{ssec:D(b)}). We extend $\pi$ to a representation of $D(\mfb^+)\cong \mfg\oplus \mfh$ by letting the central copy of $\mfh$ act trivially (equivalently, we pull back $\pi$ via the projection $\uppsi:\mfg\oplus \mfh\onto \mfg$). Recall from Proposition \ref{P:b-dbl} that 
\begin{equation*}
\scriptr^D=\sum_{\alpha\in \Root^+}x_\alpha^+\otimes x_\alpha^- + \frac{1}{2}\sum_{i\in \mbI}(\varpi_i^\vee)^+\otimes h_i^- 
\end{equation*}
defines a quasitriangular structure on $D(\mfb^+)$. In particular, it is an $r$-matrix:
\begin{equation}\label{CYBE-rD}
[\scriptr^D_{12},\scriptr^D_{13}]+[\scriptr^D_{12},\scriptr^D_{23}]+[\scriptr^D_{13},\scriptr^D_{23}]=0.
\end{equation}
In addition, by definition of $\pi$,  $\scriptr^D$ satisfies $(\pi\otimes\pi)\scriptr^D=\scriptr_\pi=(\pi\otimes\pi)\scriptr$.
We now introduce $\mbbL^\pm \in \End(V)\otimes D(\mfb^+)$ by 
\begin{equation*}
\mathbb{L}^+=(\pi\otimes \omega)\scriptr^{D} \quad \text{ and }\quad  \mbbL^-=-(\pi\otimes \omega)\scriptr^{D}_{21},
\end{equation*}
where we recall that $\omega$ is the Chevalley involution introduced in Section \ref{ssec:Db-Chev}. We then have the following result, which provides the first main theorem of this article. 
\begin{theorem}\label{T:g-main}
The assignment $L^\pm \mapsto \mathbb{L}^\pm \pm {\mathfrak{Z}}^\pm$ 
uniquely extends to an isomorphism of Lie algebras 
\begin{equation*}
\Phi_\scriptr:\mfg_\scriptr\iso D(\mfb^+)\oplus \mfz_V^+\oplus \mfz_V^-.
\end{equation*}
Moreover, the unique Lie cobracket $\delta_\scriptr:\mfg_\scriptr\to \mfg_\scriptr\wedge \mfg_\scriptr$ satisfying $\Phi_{\scriptr}^{\otimes 2} \circ \delta_\scriptr=\delta\circ \Phi_\scriptr$ is given by
$\delta_\scriptr(L^\pm)=[L_{[1]}^\pm, L_{[2]}^\pm]$.
\end{theorem}
\begin{proof}
First, observe that $\mbbL^\pm\in \pi(\mfb^\pm)\otimes \mfb^\pm$ and $\mfZ^\pm\in \mfgl(V)^\mfg\otimes \mfz_V^\pm$.
As $\mfgl(V)^\mfg$ and $\pi(\mfh)$ are both subsets of the zero weight space $\mfgl(V)_0$ and $\pi(\mfg_\alpha)\subset \mfgl(V)_{\alpha}$ for each root $\alpha$ of $\mfg$, it follows that $\mathbb{L}^\pm \pm {\mathfrak{Z}}^\pm$ satisfy  the triangularity relations: 
\begin{equation*}
(\mathbb{L}^\pm \pm {\mathfrak{Z}}^\pm)_\lambda=\mathbb{L}^\pm_\lambda \pm {\mathfrak{Z}}^\pm_\lambda =0 \quad \forall \,\lambda \in \dot{\msQ}_\mp. 
\end{equation*}
Let us now argue that the stated assignment also preserves the defining relations \eqref{gr:LpmLpm} and \eqref{gr:L+L-} of $\mfg_\scriptr$. Since the coefficients of $\mfZ^\pm$ are central elements in 
$D(\mfb^+)\oplus \mfz_V^+\oplus \mfz_V^-$, we have 
\begin{equation*}
\left[\mbbL_1^\pm \pm \mfZ^\pm_1, \mbbL_2^\pm \pm \mfZ^\pm_2\right]=[\mbbL_1^\pm,\mbbL_2^\pm],
\quad 
\left[\mbbL_1^+ + \mfZ^+_1, \mbbL_2^- - \mfZ^-_2\right]=[\mbbL_1^+,\mbbL_2^-].
\end{equation*}
Moreover, since $\scriptr_\pi\in \pi(\mfg)\otimes \pi(\mfg)$ and $\mfZ^\pm\in \mfgl(V)^\mfg\otimes \mfz_V^\pm$, we have 
\begin{equation*}
[\scriptr_\pi, \mfZ_1^\pm]=0=[\scriptr_\pi, \mfZ_2^\pm].
\end{equation*}
It therefore suffices to show that $\mbbL^\pm$ satisfy the defining relations \eqref{gr:LpmLpm} and \eqref{gr:L+L-} of $\mfg_\scriptr$. This is a consequence of the classical Yang--Baxter equation \eqref{CYBE-rD}. Indeed, after rewriting this relation as 
\begin{equation*}
[\scriptr_{13}^D,\scriptr_{23}^D]=-[\scriptr_{12}^D,\scriptr_{13}^D+\scriptr_{23}^D],
\end{equation*}
we obtain the following, after applying $\pi\otimes \pi\otimes \omega$:
\begin{equation*}
[\mbbL_1^+,\mbbL_2^+]=-[\scriptr_\pi,\mbbL_1^++\mbbL_2^+].
\end{equation*}
Applying the permutation operator $(1\,2)\circ (1\,3)$ to \eqref{CYBE-rD}, rearranging, and then applying $\pi\otimes \pi \otimes \omega$ gives instead
\begin{equation*}
[\mbbL_1^-,\mbbL_2^-]=-[\scriptr_\pi,\mbbL_1^-+\mbbL_2^-]. 
\end{equation*}
Finally, after acting by $(2\,3)$ on \eqref{CYBE-rD} and then applying $\pi\otimes \pi \otimes \omega$, we obtain
\begin{equation*}
[\mbbL_1^+,\mbbL_2^-]=-[\scriptr_\pi,\mbbL_1^++\mbbL_2^-].
\end{equation*}
This completes the proof that the assignment $L^\pm \mapsto \mbbL^\pm \pm \mfZ^\pm$ extends to a Lie algebra homomorphism $\Phi_\scriptr$, as in the statement of the theorem. As $\mbbL^\pm\in \pi(\mfg)\otimes D(\mfb^+)$, $\mfZ^\pm\in \mfgl(V)^\mfg \otimes \mfz_V^\pm$, and the sum $\pi(\mfg)+\mfgl(V)^\mfg$ is direct, the image of $\Phi_\scriptr$ contains the coefficients of $\mbbL^\pm$ and the coefficients of $\mfZ^\pm$. As the coefficients of $\mbbL^\pm$ span $\mfn^\pm\oplus \mfh^\pm$ and the coefficents of $\mfZ^\pm$ span $\mfz_V^\pm$,  we can conclude that $\Phi_\scriptr$ is surjective. In order to conclude that $\Phi_\scriptr$ is an isomorphism, it is sufficient to establish that  $\dim \mfg_\scriptr\leq \dim(D(\mfb^+)\oplus \mfz_V^+\oplus \mfz_V^-)$. 

To this end,  note that the defining relations \eqref{gr:LpmLpm} and \eqref{gr:L+L-} imply that the coefficients of $L^+$ and $L^-$ span $\mfg_\scriptr$. Moreover, by Part \eqref{wtgr:Z2} of Lemma \ref{L:Z-wtgr}, $L^\pm \in (\pi(\mfb^\pm)\oplus \mfgl(V)^\mfg)\otimes \mfg_\scriptr$. It follows that the span of the coefficients of $L^\pm$ has dimension at most $\dim(\mfb^\pm)+\dim \mfgl(V)^\mfg$, and hence that 
\begin{equation*}
 \dim \mfg_\scriptr \leq \dim(\mfb^+)+\dim(\mfb^-)+2\dim \mfgl(V)^\mfg = \dim(D(\mfb^+)\oplus \mfz_V^+\oplus \mfz_V^-).
 \end{equation*}
 Therefore, $\Phi_\scriptr$ is an isomorphism of Lie algebras. 
 
 Consider now the second statement of the theorem, concerning the Lie cobracket $\delta_\scriptr= (\Phi_{\scriptr}^{-1})^{\otimes 2} \circ \delta \circ \Phi_\scriptr$ on $\mfg_\scriptr$. To see that $\delta_\scriptr(L^\pm)=[L_{[1]}^\pm,L_{[2]}^\pm]$, we must show that 
 \begin{equation*}
 \delta(\mathbb{L}^\pm \pm \mfZ^\pm)=[\mathbb{L}^\pm_{[1]} \pm \mfZ^\pm_{[1]},\mathbb{L}^\pm_{[2]} \pm \mfZ^\pm_{[2]}].
 \end{equation*}
 By definition of the Lie cobracket on $\mfz_V^\pm$ (see Section \ref{ssec:g-inv-dual}), the left-hand side is $\delta(\mbbL^\pm)+[\mfZ_{[1]}^\pm,\mfZ_{[2]}^\pm]$. On the other hand, since $\mbbL^\pm\in \pi(\mfg)\otimes D(\mfb^+)$ and $\mfZ^\pm\in \mfgl(V)^\mfg\otimes \mfz_V^\pm$, the right-hand side is 
 \begin{equation*}
[\mathbb{L}^\pm_{[1]} \pm \mfZ^\pm_{[1]},\mathbb{L}^\pm_{[2]} \pm \mfZ^\pm_{[2]}]=[\mathbb{L}^\pm_{[1]},\mathbb{L}^\pm_{[2]}]+[\mfZ_{[1]}^\pm,\mfZ_{[2]}^\pm]. 
 \end{equation*}%
 We are thus left to establish that $\delta(\mbbL^\pm)=[\mathbb{L}^\pm_{[1]},\mathbb{L}^\pm_{[2]}]$. As $D(\mfb^+)$ is quasitriangular with associated $r$-matrix $\scriptr^D$, we have 
 \begin{equation*}
 \delta(\mbbL^\pm)=[\mbbL^\pm_{[1]}+\mbbL^\pm_{[2]}, \scriptr_{23}^D]=(\pi \otimes \omega^{\otimes 2})([\scriptr^\pm_{12}+\scriptr^\pm_{13},\scriptr_{32}^D]),
 \end{equation*}
 where we have set $\scriptr^+=\scriptr^D$ and $\scriptr^-=-\scriptr_{21}^D$ and used that $(\omega\otimes \omega)\scriptr^D=\scriptr^D_{21}$. The right-hand side will equal $[\mathbb{L}^\pm_{[1]},\mathbb{L}^\pm_{[2]}]$ provided  $[\scriptr^\pm_{12},\scriptr^\pm_{13}]=[\scriptr_{12}^\pm + \scriptr_{13}^\pm, \scriptr_{32}^D]$ or, equivalently, provided 
 \begin{equation*}
[\scriptr_{12}^D,\scriptr_{13}^D]=[\scriptr_{12}^D + \scriptr_{13}^D, \scriptr_{32}^D] \quad \text{ and }\quad  [\scriptr_{21}^D,\scriptr_{31}^D]=[\scriptr_{32}^D,\scriptr_{21}^D + \scriptr_{31}^D].
 \end{equation*}
 Each of these relations is independently equivalently to the classical Yang--Baxter equation \eqref{CYBE-rD}, and are thus satisfied. 
\end{proof}
The above theorem implies that the Drinfeld double $D(\mfb^+)$ can be recovered as both a subalgebra and a quotient of $\mfg_\scriptr$. 
In more detail, let $\mfz_{\scriptr}^\pm$ be Lie ideal of $\mfg_\scriptr$ spanned by the coefficients of 
\begin{equation*}
K^\pm=L^\pm-\kappa^{-1}\Casimir(L^\pm).
\end{equation*}
Note that $\Phi_{\scriptr}(K^\pm)=\mbbL^\pm\pm \mfZ^\pm-\kappa^{-1}\Casimir(\mbbL^\pm\pm \mfZ^\pm) =\pm \mfZ^\pm$. In particular, $\mfz_{\scriptr}^\pm$ is a Lie bialgebra ideal of $\mfg_\scriptr=(\mfg_\scriptr,\delta_{\scriptr})$, and one has $\Phi_{\scriptr}(\mfz_{\scriptr}^\pm)=\mfz_V^\pm$. As an immediate consequence of Theorem \ref{T:g-main}, we obtain the following corollary. 
\begin{corollary}\label{C:gr->D(b)}
The composition of $\Phi_\scriptr$ with the natural projection $D(\mfb^+)\oplus \mfz_V^+\oplus \mfz_V^-\onto D(\mfb^+)$ induces an isomorphism of Lie bialgebras 
\begin{equation*}
 {\mfg_\scriptr}/(\mfz_{\scriptr}^+ + \mfz_{\scriptr}^-)\iso D(\mfb^+). 
\end{equation*}
Moreover, the subspace $\mfg_{\scriptr,\msL}$  of $\mfg_\scriptr$ spanned by the coefficients of $\msL^+$ and $\msL^-$ is a Lie sub-bialgebra and $\Phi_\scriptr$ restricts to an isomorphism 
\begin{equation*}
\Phi_\scriptr|_{\mfg_{\scriptr,\msL}}:\mfg_{\scriptr,\msL}\iso D(\mfb^+).
\end{equation*}
\end{corollary}
\begin{remark}\label{R:gr->g}
Since $D(\mfb^+)\cong \mfg \oplus \mfh$ (see Proposition \ref{P:b-dbl}), $\mfg$ itself can be realized both as Lie subalgebra of $\mfg_\scriptr$ and as a Lie bialgebra quotient. As a Lie subalgebra, it is spanned by the coefficients of $\msF:=\msL^- - \msL^+$, which satisfy the defining relations of $\mfg$ spelled out in \cite{WRTT}*{Prop.~4.4}. Similarly, it is the quotient of $\mfg_\scriptr$ by the ideal spanned by the coefficients of the (central) matrices
\begin{equation*}
K^\pm=L^\pm - \tfrac{1}{\kappa}\Casimir(L^\pm), \quad L^+_0+L^-_0,
\end{equation*}
where $L^\pm_0=\mathbf{1}_0(L^\pm)$ is the projection of $L^\pm$ onto its weight zero component. The quantum analogues of these observations will be established in Theorem \ref{T:URg->Uhg}.
\end{remark}
%
%
\subsection{The Borel subalgebras $\mfb^\pm_\scriptr$}\label{ssec:g-borel}

We conclude this section by noting that the above results naturally output matrix presentations of (central extension of) the Borel subalgebras $\mfb^\pm$ of $\mfg$. 

\begin{definition}\label{D:br}
Let $\mfb^+_\scriptr$ and $\mfb^-_\scriptr$ be the Lie algebras generated by $\{l_{ij}^+\}_{i,j\in \mcI}$ and $\{l_{ij}^-\}_{i,j\in \mcI}$, respectively, subject to the relations
\begin{gather}
L^\pm_{\lambda}=0 \quad \forall\; \lambda\in \dot{\mathsf{Q}}_{\mp}, \\
[L_1^\pm,L_2^\pm]=-[\scriptr_\pi,L_1^\pm+L_2^\pm],
\end{gather}
where $L^\pm$ is the generating matrix 
\begin{equation*}
L^\pm=\sum_{i,j\in \mcI} E_{ij}\otimes l_{ij}^\pm \in \End(V)\otimes \mfb^\pm_\scriptr.
\end{equation*}
\end{definition}
Here we emphasize that each symbol $\pm$ and $\mp$ only takes its upper value for $\mfb^+_\scriptr$, and its lower value for $\mfb^-_\scriptr$. 
It follows from this definition that, for each choice of the symbol $\pm$, there is a Lie algebra homomorphism 
\begin{equation*}
\imath^{\scriptscriptstyle{\pm}}_\scriptr:\mfb^\pm_\scriptr\to \mfg_\scriptr, \quad \imath^{\scriptscriptstyle{\pm}}_\scriptr(L^\pm)=L^\pm. 
\end{equation*}
Our choice of notation for the generators of $\mfb^\pm_\scriptr$ is justified by the following corollary.

\begin{corollary}\label{C:b_r}
The composite $\Phi^\pm_\scriptr=\Phi_\scriptr \circ \imath^{\scriptscriptstyle{\pm}}_\scriptr$ is an isomorphism of Lie algebras 
\begin{equation*}
\Phi^\pm_\scriptr:\mfb^\pm_\scriptr\iso \mfb^\pm \oplus \mfz_V^\pm.
\end{equation*}
In particular, $\imath^{\scriptscriptstyle{\pm}}_\scriptr$ is injective and identifies $\mfb^\pm_\scriptr$ with the Lie sub-bialgebra of $\mfg_\scriptr$ generated by the coefficients of $L^\pm$. 
\end{corollary}
\begin{proof}
That $\Phi^\pm_\scriptr$ is surjects onto $\mfb^\pm \oplus \mfz_V^\pm$ follows by the same argument as used to establish the surjectivity of $\Phi_\scriptr$ in the proof of Theorem \ref{T:g-main}. Next, note that the generating matrix $L^\pm$ of $\mfb^\pm_\scriptr$ satisfies 
\begin{equation*}
L^\pm\in (\pi(\mfb^\pm)\oplus \mfgl(V)^\mfg)\otimes \mfb^\pm_\scriptr. 
\end{equation*}
Indeed, this is proven identically to Part \eqref{wtgr:Z2} of Lemma \ref{L:Z-wtgr}. 
It follows from this fact and the same type of dimension argument as given in the proof of Theorem \ref{T:g-main} that $\Phi^\pm_\scriptr$ is injective, and thus an isomorphism. \qedhere
\end{proof}

\section{Recollections on \texorpdfstring{$\Uh{\mfg}$}{U\_hg}}\label{sec:Uhg}

\subsection{Topological Hopf algebras and quantizations}\label{ssec:h-adic}

For the remainder of this article, we will mostly be concerned with topological modules, algebras and Hopf algebras defined over the formal power series ring $\C[\![\hbar]\!]$. In this section we briefly summarize a number of facts about these structures, without proof, which are pertinent to our main results. We will follow the exposition given in \cite{WRQD}*{\S2} closely, though we refer the reader to \cite{KasBook95}*{\S XVI}, for example, for a more complete background. 

To begin, recall that a $\C[\![\hbar]\!]$-module $\mcV$ is said to be \textit{topologically free} if $\mcV\cong V[\![\hbar]\!]$ for some complex vector space $V$. This is equivalent to the requirement that $\mcV$ is separated, complete and torsion free as a $\C[\![\hbar]\!]$-module. The former two conditions mean precisely that the $\C[\![\hbar]\!]$-linear map 
\begin{equation*}
\mcV\to \varprojlim_{n}\left(\mcV/\hbar^n \mcV\right)
\end{equation*}
is injective and surjective, respectively. 
The \textit{semiclassical limit} of a $\C[\![\hbar]\!]$-module $\mcV$ is the space $\mcV/\hbar \mcV$. Note that if $\mcV$ is topologically free with $\mcV\cong V[\![\hbar]\!]$, then the underlying space $V$ can naturally be identified with this limit. Similarly, if $\mcV$ and $\mcW$ are two $\C[\![\hbar]\!]$-linear maps and $\varphi:\mcV\to \mcW$ is a $\C[\![\hbar]\!]$-linear map, then the semiclassical limit of $\varphi$ is the $\C$-linear map $\bar{\varphi}:\mcV/\hbar\mcV\to \mcW/\hbar\mcW$ uniquely determined by the commutativity of the diagram 
\begin{equation*}
\begin{tikzcd}[column sep=13ex]
 \mcV \arrow[two heads]{d}  \arrow{r}{\varphi} &  \mcW \arrow[two heads]{d}\\
 \mcV/\hbar\mcV \arrow{r}{\bar\varphi}       & \mcW/\hbar\mcW
\end{tikzcd}
\end{equation*}
where the vertical arrows represent the canonical quotient maps. We note the following useful facts: 
\begin{enumerate}[label=(L\arabic*)]
\item\label{L1} Suppose that $\mcV$ is separated and $\mcW$ is torsion free. Then $\varphi$ is injective provided $\bar\varphi$ is. 
\item\label{L2}  Suppose that $\mcV$ is complete and $\mcW$ is separated. Then $\varphi$ is surjective provided $\bar\varphi$ is. 
\end{enumerate}
These properties are explicitly stated in Lemma 2.1 of \cite{WRQD}, and are relatively straightforward applications of the relevant definitions. 

The topological tensor product $\mcV\,\widehat{\otimes}\,\mcW$ of the $\C[\![\hbar]\!]$-modules $\mcW$ and $\mcV$ is the $\hbar$-adic completion of their algebraic tensor product:
\begin{equation*}
\mcV\,\widehat{\otimes}\,\mcW:=\varprojlim_n (\mcV\otimes_{\C[\![\hbar]\!]}\mcW)/\hbar^n(\mcV\otimes_{\C[\![\hbar]\!]}\mcW).
\end{equation*}
The tensor product $\wh{\otimes}$ endows the category of separated and complete $\C[\![\hbar]\!]$-modules with a symmetric monoidal structure. Moreover, if $\mcV$ and $\mcW$ are topologically free with $\mcV\cong V[\![\hbar]\!]$ and $\mcW\cong W[\![\hbar]\!]$, then $\mcV\, \wh{\otimes}\, \mcW$ is topologically free with $\mcV\, \wh{\otimes}\, \mcW\cong (V\otimes_\C W)[\![\hbar]\!]$.  We refer the reader to \cite{KasBook95}*{\S XVI.3}, for example, for a comprehensive and elementary exposition to $\wh{\otimes}$ and its key properties. 

We will say that $\msA$ is a topological algebra over $\C[\![\hbar]\!]$ if it is an algebra over $\C[\![\hbar]\!]$ which is both separated and complete as a $\C[\![\hbar]\!]$-module. For instance, all unital associative $\C[\![\hbar]\!]$-algebras defined via (topological) generators and relations are understood as topological algebras; see \cite{KasBook95}*{\S XVII.2}. Similarly, a topological Hopf algebra $\msH$ over $\C[\![\hbar]\!]$ is a topological algebra equipped with a coproduct $\Delta:\msH\to \msH\,\wh{\otimes}\, \msH$, a counit $\veps: \msH\to \C[\![\hbar]\!]$ and an antipode $S:\msH\to \msH$, which satisfy the axioms of a Hopf algebra with all tensor products given by $\wh{\otimes}$.

 Let us now recall a few standard definitions from the theory of quantum groups, following \cite{ES}*{\S9}. 
\begin{definition}
A topological Hopf algebra $\msH$ over $\C[\![\hbar]\!]$ is called a \textit{quantized enveloping algebra} if it satisfies the following two conditions:
\begin{enumerate}\setlength{\itemsep}{3pt}
 \item $\msH$ has semiclassical limit $\msH/\hbar\msH$ isomorphic to the enveloping algebra $U(\mfa)$ of a complex Lie algebra $\mfa$ as a Hopf algebra. 
 \item $\msH$ is topologically free, and thus isomorphic to $U(\mfa)[\![\hbar]\!]$ as a $\C[\![\hbar]\!]$-module.
 \end{enumerate}
\end{definition}
A quantized enveloping algebra $\Uh{\mfa}$ with semiclassical limit $U(\mfa)$ automatically induces a Lie bialgebra structure on $\mfa$ with cobracket given by
\begin{equation*}
\delta(x):= \frac{\Delta(\dot{x})-\Delta^{\mathrm{op}}(\dot{x})}{\hbar} \mod \hbar U_\hbar\mfa\,\wh{\otimes}\, U_\hbar\mfa \quad \forall \; x\in \mfa,
\end{equation*}
where $\dot{x}\in U_\hbar\mfa$ is any lift of $x$. Conversely, if $(\mfa,\delta)$ is any Lie bialgebra, then a \textit{quantization} of $(\mfa,\delta)$ is a quantized enveloping algebra $\Uh{\mfa}$ with semiclassical limit $U(\mfa)$, for which  $\delta$ is recovered by the above formula.
\begin{remark}
For the remainder of this article, we will shall denote the topological tensor product $\widehat{\otimes}$ by $\otimes$. More generally, the use of this symbol will always be clear from the underlying context. 
\end{remark}
%
%

\subsection{The quantized enveloping algebra $\Uh{\mfg}$}\label{ssec:Uhg}

Throughout the rest of this paper, it is understood that $q=e^{\hbar/2}\in 1+\hbar\C[\![\hbar]\!]$. In addition, we shall employ the standard notation for Gaussian integers and binomial coefficient. Namely, if $m,n,r\in \Z$ with $n\geq r\geq 0$, then we set
\begin{gather*}
\sbinom{n}{r}_q=\frac{[n]_q!}{[r]_q![n-r]_q!}, \quad [m]_q!=[m]_q[m-1]_q\cdots [1]_q,\\ 
[m]_q=\frac{q^m-q^{-m}}{q-q^{-1}}.
\end{gather*}
In the following definition, $\{h_i\}_{i\in \mbI}\subset \mfh$ and $\{d_i\}_{i\in \mbI}\subset \Z_{>0}$ are as in Section \ref{ssec:g}. 
\begin{definition}\label{D:Uhg}
The Drinfeld--Jimbo algebra $\Uh{\mfg}$ is the unital, associative $\C[\![\hbar]\!]$-algebra topologically generated by $\mfh\cup\{E_i, F_i\}_{i\in \mbI}$, subject to the following relations for all $h,h^\prime\in \mfh$ and $i,j\in \mbI$:
\begin{gather*}
[h,h^\prime]=0, \\
 [h,E_j]=\alpha_j(h)E_j, \quad [h,F_j]=-\alpha_j(h)F_j,\\
[E_i,F_j]=\delta_{ij}\frac{q^{ h_i}-q^{-h_i}}{q_i-q_i^{-1}} \\ 
\sum_{b=0}^{1-a_{ij}}(-1)^b \sbinom{1-a_{ij}}{b}_{q_i} E_i^b E_j E_i^{1-a_{ij}-b}=0,\\
\sum_{b=0}^{1-a_{ij}}(-1)^b \sbinom{1-a_{ij}}{b}_{q_i} F_i^b F_j F_i^{1-a_{ij}-b}=0,
\end{gather*}
where  $q_i=q^{d_i}=e^{\hbar d_i/2}$ and in the last two relations $i\neq j$.
\end{definition}
\begin{remark}
Here we note that if $H_i\in \mfh$ is defined by $d_iH_i=h_i=\nu^{-1}(\alpha_i)$ (see Section \ref{ssec:g}), then $\{H_i,E_i,F_i\}_{i\in \mbI}$ generates $\Uh{\mfg}$ as a topological $\C[\![\hbar]\!]$-algebra and one has the familiar relations
\begin{equation*}
[H_i,E_j]=a_{ij}E_j, \quad [H_i,F_j]=-a_{ij}F_j,\quad [E_i,F_j]=\delta_{ij}\frac{q_i^{H_i}-q_i^{-H_i}}{q_i-q_i^{-1}}. 
\end{equation*}
\end{remark}
The $\C[\![\hbar]\!]$-algebra $\Uh{\mfg}$ admits the structure of a topological Hopf algebra over $\C[\![\hbar]\!]$, with coproduct $\Delta$, antipode $S$, and counit $\veps$ uniquely determined by the requirement that the image of $\mfh$ in $\Uh{\mfg}$ is primitive and $\{E_i,F_i\}_{i\in \mbI}$ satisfy
\begin{gather*}
\Delta(E_i)=E_i\otimes q^{h_i} + 1\otimes E_i, \quad S(E_i)=-E_i q^{-h_i}, \quad \veps(E_i)=0,\\
\Delta(F_i)=F_i\otimes 1 + q^{-h_i}\otimes F_i, \quad S(F_i)=-q^{h_i}F_i,\quad \veps(F_i)=0.
\end{gather*}
These formulas, together with Definition \ref{D:Uhg}, imply that $\Uhg$ is a Hopf algebra deformation of $U(\mfg)$ over $\C[\![\hbar]\!]$. Namely, if $d_i$,  $h_i$ and $x_i^\pm:=x_{\alpha_i}^\pm$ are as in Sections \ref{ssec:g} and \ref{ssec:Db-Chev}, then the assignment $h_i\mapsto h_i$, $E_i\mapsto \sqrt{d_i} x_i^+$, $F_i\mapsto \sqrt{d_i} x_i^-$ gives rise to a surjective Hopf algebra morphism $\Uhg\onto U(\mfg)$ (where $\hbar$ operates as $0$ in $U(\mfg)$) which induces an isomorphism 
\begin{equation*}
\Uhg/\hbar \Uhg \iso U(\mfg)
\end{equation*}
of Hopf algebras over $\C$. In addition, it is well-known that $\Uhg$ is topologically free over $\C[\![\hbar]\!]$, and thus a quantized enveloping algebra with semiclassical limit $U(\mfg)$. Furthermore, it is quasitriangular and the Lie bialgebra structure it quantizes is the so-called standard structure recalled below Proposition \ref{P:b-dbl}. As we shall review in Section \ref{ssec:DUhb}, this is best captured by realizing $\Uhg$ as a Hopf quotient of the quantum double $D(\Uh{\mfb})$ of its Borel subalgebra $\Uh{\mfb}$, in complete analogy with the classical story summarized in Section \ref{ssec:D(b)}.

To conclude this subsection, we recall that the adjoint action of $\mfh\subset \Uhg$ on $\Uhg$ gives rise to a $\msQ$-graded topological Hopf algebra structure on $\Uhg$ with homogeneous components 
\begin{equation*}
\Uhg_\beta=\{x\in \Uhg:[h,x]=\beta(h)x \quad \forall\; h\in \mfh\} \quad \forall\; \beta\in \msQ. 
\end{equation*}
Equivalently, each subspace $\Uhg_\beta$ is a closed $\C[\![\hbar]\!]$-submodule of $\Uhg$ and 
\begin{equation*}
\Uhg_\msQ=\bigoplus_{\beta\in\msQ} \Uhg_\beta
\end{equation*}
is a dense, $\msQ$-graded $\C[\![\hbar]\!]$-subalgebra of $\Uhg$ with induced topology that coincides with its $\hbar$-adic topology, and the structure maps $\Delta$, $S$ and $\veps$ are all $\msQ$-graded.

\subsection{The Borel subalgebra $\Uh{\mfb}$ and its dual}\label{ssec:Uhb}

Now let $\Uh{\mfh}$, $\Uh{\mfb}$ and $\Uh{\mfn}$ be the $\C[\![\hbar]\!]$-subalgebras of $\Uhg$ topologically generated by $\mfh$,  $\mfh\cup \{E_i\}_{i\in \mbI}$ and $\{E_i\}_{i\in \mbI}$, respectively. Then $\Uh{\mfh}$ and $\Uh{\mfb}$ are topological Hopf subalgebras of $\Uhg$ which provide quantizations of $\mfh$ and $\mfb$ (viewed as Lie sub-bialgebras of $\mfg$), with $\Uh{\mfh}$ isomorphic to the trivial deformation $U(\mfh)[\![\hbar]\!]\cong \msS(\mfh)[\![\hbar]\!]$ as a topological Hopf algebra. Furthermore, both $\Uh{\mfb}$ and $\Uh{\mfn}$ inherit $\msQ_+$-gradings from the $\msQ$-grading on $\Uhg$, and the multiplication map
\begin{equation*}
m:\Uh{\mfh}\otimes \Uh{\mfn}\to \Uh{\mfb}
\end{equation*}
provides an isomorphism of $\msQ_+$-graded topological $\C[\![\hbar]\!]$-modules. In particular, 
\begin{equation*}
(\Uh{\mfn})_{\alpha_i}=\C[\![\hbar]\!] E_i \quad \forall\; i\in \mbI.
\end{equation*}
 We shall also set $\Uh{\mfb^-}=\omega_\hbar(\Uh{\mfb})$ and $\Uh{\mfn^-}=\omega_\hbar(\Uh{\mfn})$, where $\omega_\hbar$ is the \textit{Chevalley involution} on $\Uhg$. That is, it is the involutive $\C[\![\hbar]\!]$-algebra automorphism of $\Uhg$ uniquely determined by $\omega_\hbar(h)=-h$ for all $h\in \mfh\subset \Uhg$, while
\begin{equation*}
\omega_\hbar(E_i)=-F_i \quad \text{ and }\quad \omega_\hbar(F_i)=-E_i \quad \forall\; i\in \mbI.
\end{equation*}
When discussing $\Uh{\mfb}$ and $\Uh{\mfb^-}$ (resp. $\Uh{\mfn}$ and $\Uh{\mfn^-}$) we will sometimes write $\Uh{\mfb^+}$ (resp. $\Uh{\mfn^+}$) for the former. 

Let us now recall the definition of the quantized enveloping algebra dual of $\Uh{\mfb}$, following \cite{DrQG}*{\S7}, \cite{Etingof-Kazhdan-I}*{\S4.4} and \cite{Gav02}; see also \cite{CPBook}*{\S6.3C}. We first consider the more general situation where we are given a quantization $\Uh{\mfa}$ of an arbitrary finite-dimensional Lie bialgebra $\mfa$. Consider the $\C[\![\hbar]\!]$-module 
\begin{equation*}
\Uh{\mfa}^\prime=\{x\in \Uh{\mfa}: (\id-\veps)^{\otimes n} \Delta^n(x)\in \hbar^n \Uh{\mfa}^{\otimes n} \; \forall\; n\geq 0\}\subset \Uh{\mfa},
\end{equation*}
where $\veps$ and $\Delta$ are the counit and coproduct on $\Uh{\mfa}$, and $\Delta^n: \Uh{\mfa}\to \Uh{\mfa}^{\otimes n}$ is defined recursively by $\Delta^0=\veps$, $\Delta^1 =\id$, and 
\begin{equation*}
\Delta^n=(\Delta\otimes \id^{\otimes (n-2)})\circ \Delta^{n-1} \quad \forall\; n\geq 2.
\end{equation*}
Then, by \cite{DrQG}*{\S7} (see \cite{Gav02}*{Prop.~3.6} for a detailed proof), $\Uh{\mfa}^\prime$ is a \textit{quantized formal series Hopf algebra}. In particular, it is a topological Hopf algebra with respect to the subspace topology, which coincides with the $\mcJ_\mfa$-adic topology, where 
\begin{equation*}
\mcJ_\mfa=\hbar\Uh{\mfa} \cap \Uh{\mfa}^\prime= \veps_{\Uh{\mfa}^\prime}^{-1}(\hbar\C[\![\hbar]\!]).
\end{equation*}
Moreover, its semiclasical limit is isomorphic, as an algebra, to the completion $\widehat{\msS(\mfa)}$ of the symmetric algebra $\msS(\mfa)=\bigoplus_{n\geq 0}\msS^n(\mfa)$ with respect to its standard grading. In the present article, we call $\Uh{\mfa}^\prime$ the \textit{Drinfeld--Gavarini subalgebra} of $\Uh{\mfa}$.

The \textit{quantized enveloping algebra dual} $\Uh{\mfa}{\vphantom{)}}^\bullet$ to $\Uh{\mfa}$ is then the subspace of the $\C[\![\hbar]\!]$-linear dual $(\Uh{\mfa}^\prime){\vphantom{)}}^\ast=\Hom_{\C[\![\hbar]\!]}( \Uh{\mfa}^\prime, \C[\![\hbar]\!])$ consisting of continuous linear forms with respect to the aforementioned topology. The general theory dictates that it is a quantized enveloping algebra which quantizes the Lie bialgebra dual $\mfa^\ast$ of $\mfa$.

Now let us narrow our focus to the particular case where $\mfa=\mfb$. Our present aim is to identify a collection of elements in $\Uh{\mfb}{\vphantom{)}}^\bullet$ which will play an important role in Section \ref{ssec:DUhb}. 
To this end, note that the formulas for $\Delta$ and $\veps$ given in Section \ref{ssec:Uhg} imply that 
$\hbar \mfh$ and $\hbar E_i$ belong to\footnote{They do not, however, generate it as topological algebra; we refer the reader to \cite{Gav02}*{\S3.5} for an explicit description of $\Uh{\mfb}^\prime$.} $\Uh{\mfb}^\prime$, for each $i\in \mbI$. Let 
\begin{equation*}
\mathbf{1}_{\alpha_i}: \Uh{\mfn}\to (\Uh{\mfn})_{\alpha_i} = \C[\![\hbar]\!]E_i
\end{equation*}
  be the projection onto $(\Uh{\mfn})_{\alpha_i}$,  arising from the $\msQ_+$-grading on $\Uh{\mfn}$. Let $\partial_i$ be the $\C[\![\hbar]\!]$-linear derivative with respect to $E_i$ on $(\Uh{\mfn})_{\alpha_i}$: $\partial_i(E_i)=1$. We then define $f_i,\chi_i\in \Uh{\mfb}^\ast$ by
\begin{gather*}
f_i:= \partial_i\circ \mathbf{1}_{\alpha_i}\circ (\veps\otimes \id): \Uh{\mfb}\cong \Uh{\mfh}\otimes \Uh{\mfn} \to \C[\![\hbar]\!],\\
\chi_i:=\omega_i \circ (\id \otimes \veps):\Uh{\mfb}\cong \Uh{\mfh}\otimes \Uh{\mfn} \to \C[\![\hbar]\!],
\end{gather*}
where $\omega_i\in \mfh^\ast \subset \msS(\mfh)^\ast$ is the $i$-th fundamental weight of $\mfg$, and is extended by $\C[\![\hbar]\!]$-linearity to an element of $\msS(\mfh)[\![\hbar]\!]^\ast \cong \Uh{\mfh}^\ast$. 

Since $\hbar\mfh + \hbar \bigoplus_{i\in \mbI}(\Uh{\mfn})_{\alpha_i}\subset \Uh{\mfb}^\prime$, we may then define $\upxi_i$ and $\eta_i$ in $\Uh{\mfb}{\vphantom{)}}^\bullet$ by 
\begin{equation*}
\upxi_i= \hbar^{-1} \chi_i|_{\Uh{\mfb}^\prime} \quad \text{ and }\quad \eta_i= \hbar^{-1} f_i|_{\Uh{\mfb}^\prime} \quad \forall \quad i\in \mbI.
\end{equation*}
Though we shall not need this fact in what follows, it is worth pointing out that the above elements generate $\Uh{\mfb}{\vphantom{)}}^\bullet$ as a topological $\C[\![\hbar]\!]$-algebra, and that the co-opposite Hopf algebra $\chk{\Uh{\mfb}}:=(\Uh{\mfb}{\vphantom{)}}^\bullet)^{\mathrm{cop}}$ is isomorphic to $\Uh{\mfb^-}$, as is readily recovered from Theorem \ref{T:dbl-Uhg(x)S} below and the relations above \eqref{DUhb->Uhg}. 
%
\subsection{The quantum double $D(\Uh{\mfb})$}\label{ssec:DUhb}
In this subsection, we briefly review the definition of the quantum double $D(\Uh{\mfb})$ and its realization as $\Uhg\otimes \msS(\mfh)[\![\hbar]\!]$. We begin by recalling Drinfeld's original definition (see \cite{DrQG}*{\S13}) for the quantum double of an arbitrary quantized enveloping algebra of finite type, following the summary given in \cite{WRQD}*{\S8.1}. 

Let $\mfa$ be a finite-dimensional Lie bialgebra, and let $\Uh{\mfa}$ be a quantization of $\mfa$, as in Section \ref{ssec:Uhb}. Following the notation of Section \ref{ssec:Uhb}, we set $\chk{\Uh{\mfa}}:=(\Uh{\mfa}{\vphantom{)}}^\bullet)^{\mathrm{cop}}$, where $\Uh{\mfa}{\vphantom{)}}^\bullet$ is the quantized enveloping algebra dual of $\Uh{\mfa}$. 
 Given this data, there exists a unique topological Hopf algebra $D(\Uh{\mfa})$ over $\C[\![\hbar]\!]$ with the following properties:
\begin{enumerate}\setlength{\itemsep}{3pt}
\item\label{DUha:1} There are embeddings of topological Hopf algebras 
\begin{equation*}
\imath^{\scriptscriptstyle{+}}: \!\Uh{\mfa}\into D(\Uh{\mfa})\quad \text{ and }\quad \imath^{\scriptscriptstyle{-}}\!: \chk{\Uh{\mfa}}\into D(\Uh{\mfa}).
\end{equation*}
\item\label{DUha:2} Let $m$ be the product on $D(\Uh{\mfa})$. Then the composition $m\circ (\imath^{\scriptscriptstyle{-}}\otimes \imath^{\scriptscriptstyle{+}})$ is an  isomorphism of $\C[\![\hbar]\!]$-modules:
\begin{equation*}
m\circ (\imath^{\scriptscriptstyle{-}}\otimes \imath^{\scriptscriptstyle{+}}):  \chk{\Uh{\mfa}}\otimes \Uh{\mfa}\iso D(\Uh{\mfa}).
\end{equation*}

\item\label{R:can} The canonical element 
\begin{equation*}
R\in \Uh{\mfa}\otimes \Uh{\mfa}{\vphantom{)}}^\ast\subset \Uh{\mfa}\otimes \chk{\Uh{\mfa}}\cong \imath^{\scriptscriptstyle{+}}(\Uh{\mfa})\otimes \imath^{\scriptscriptstyle{-}}(\chk{\Uh{\mfa}})\subset D(\Uh{\mfa})^{\otimes 2}
\end{equation*}
 defines a quasitriangular structure on $D(\Uh{\mfa})$. That is, one has
\begin{gather*}
\Delta^{\mathrm{op}}(x)=R \Delta(x)   R^{-1} \quad \forall\; x\in D(\Uh{\mfa}),\\
\Delta\otimes \id ( R ) =R_{13}R_{23} \quad \text{ and }\quad \id\otimes \Delta (R)=R_{13}R_{12}.
\end{gather*}
\end{enumerate}
The uniquely defined topological Hopf algebra $D(\Uh{\mfa})$ is called the \textit{quantum double} of $\Uh{\mfa}$, and the above characterization implies that it provides a quantization of the Lie bialgebra double of $\mfa$. It can be explicitly realized on the space $\chk{\Uh{\mfa}}\otimes \Uh{\mfa}$ as the double cross product Hopf algebra 
\begin{equation*}
D(\Uh{\mfa})=\chk{\Uh{\mfa}}\bowtie \Uh{\mfa},
\end{equation*}
with respect to the left and right coadjoint actions $\rhd$ and $\lhd$  of $\Uh{\mfa}$ on $\chk{\Uh{\mfa}}$ and $\chk{\Uh{\mfa}}$ on $\Uh{\mfa}$, respectively. We refer the reader to \cite{Andrea-Valerio-18}*{\S A} for the relevant details, adapted to the setting of quantized enveloping algebras.

Now let us return to the particular case $\mfa=\mfb\subset \mfg$. Our present goal is to recall the identification of the quantum double $D(\Uh{\mfb})$ with the $\C[\![\hbar]\!]$-algebra $\Uhg\otimes \msS(\mfh)[\![\hbar]\!]$.  Following Section \ref{ssec:D(b)}, we will denote the natural inclusion of $\msS(\mfh)[\![\hbar]\!]$ into $\Uhg\otimes \msS(\mfh)[\![\hbar]\!]$  by $\zeta$. That is, one has
\begin{equation*}
\zeta(h):=1\otimes h \quad \forall\quad  h\in \msS(\mfh)[\![\hbar]\!]. 
\end{equation*}
We shall need the following preliminary lemma, whose proof is straightforward.
\begin{lemma}\label{L:Uhg(x)S}
$\Uhg\otimes \msS(\mfh)[\![\hbar]\!]$ is isomorphic to the unital, associative $\C[\![\hbar]\!]$-algebra topologically generated by $\{h_i^\pm, X_i^\pm\}_{i\in \mbI}$, subject to the relations
\begin{gather*}
[h_i^\pm,h_j^\mp]=0=[h_i^\pm,h_j^\pm],\\
[h_i^\eps,X_j^\pm]=\pm (\alpha_i,\alpha_j)X_j^\pm,\\
[X_i^+,X_j^-]=\delta_{ij}\frac{q^{h_i^+}-q^{-h_i^-}}{q_i-q_i^{-1}},\\
\sum_{b=0}^{1-a_{ij}}(-1)^b \sbinom{1-a_{ij}}{b}_{q_i} (X_i^\pm)^b X_j^\pm (X_i^\pm)^{1-a_{ij}-b}=0,
\end{gather*}
where $\eps$ takes value $+$ or $-$ and in the last relation we assume that $i\neq j$. Explicitly, an isomorphism from this algebra into $\Uhg\otimes \msS(\mfh)[\![\hbar]\!]$ is given by
\begin{equation*}
h_i^\pm\mapsto h_i\pm \zeta(h_i), \quad X_i^+\mapsto q^{\zeta(h_i)/2}E_i,\quad X_i^-\mapsto q^{\zeta(h_i)/2} F_i  \quad \forall \; i\in \mbI.
\end{equation*}
\end{lemma}
Henceforth, we shall assume the above realization of $\Uhg\otimes \msS(\mfh)[\![\hbar]\!]$ without further mention of the underlying isomorphism. With this in mind, the following well-known result, originally due to Drinfeld \cite{DrQG}*{\S13}, provides the desired identification of $D(\Uh{\mfb})$ with $\Uhg\otimes \msS(\mfh)[\![\hbar]\!]$; see also \cite{CPBook}*{\S8.3} and \cite{Rosso89}*{\S II} (for $\mfg=\mfsl_{n+1}$). 
\begin{theorem}\label{T:dbl-Uhg(x)S}
There is an isomorphism of $\C[\![\hbar]\!]$-algebras
\begin{equation*}
\Uppsi: \Uhg\otimes \msS(\mfh)[\![\hbar]\!]\iso D(\Uh{\mfb})
\end{equation*}
uniquely determined by the requirement that, for each $i\in \mbI$, one has
\begin{gather*}
\Uppsi(h_i^+)=h_i, \quad \Uppsi(X_i^+)=E_i,\\
\Uppsi(h_i^-)=2\sum_{j\in \mbI} a_{ji}\upxi_j \quad \text{ and }\quad  \Uppsi(X_i^-)=\frac{\hbar}{q_i-q_i^{-1}}\eta_i.
\end{gather*}
\end{theorem}
With respect to the above identification, the coproduct $\Delta$, antipode $S$, and counit $\veps$ of $D(\Uh{\mfb})$ are determined by the requirement that $\{h_i^\pm\}_{i\in \mbI}$ are primitive elements and that $X_i^\pm$ satisfy
\begin{gather*}
\Delta(X_i^+)=X_i^+\otimes q^{h_i^+} + 1\otimes X_i^+, \quad S(X_i^+)=-X_i^+ q^{-h_i^+}, \quad \veps(X_i^+)=0,\\
\Delta(X_i^-)=X_i^-\otimes 1 + q^{-h_i^-}\otimes X_i^-, \quad S(X_i^-)=-q^{h_i^-}X_i^-,\quad \veps(X_i^-)=0,
\end{gather*}
for all $i\in \mbI$. In particular, if $\upepsilon_\mfh=\varepsilon|_{\msS(\mfh)[\![\hbar]\!]}$, then the projection 
\begin{equation}\label{DUhb->Uhg}
\uppsi_\hbar:=\id_{\Uhg}\otimes \upepsilon_\mfh: D(\Uh{\mfb})\cong \Uhg\otimes \msS(\mfh)[\![\hbar]\!]\onto U_\hbar \mfg
\end{equation}
is an epimorphism of Hopf algebras which quantizes the Lie bialgebra surjection $\uppsi:D(\mfb^+)\onto \mfg$ defined below Proposition \ref{P:b-dbl}. 
\subsection{Automorphisms of $D(\Uh{\mfb})$}\label{ssec:DUhb-aut}
Before discussing in more detail the quasitriangularity of the quantum double $D(\Uh{\mfb})$, we apply the results of the above section to study certain automorphisms $D(\Uh{\mfb})$, beginning with the following corollary. 
\begin{corollary} \label{C:D-fixed}
For each $h\in \mfh$, the assignment 
\begin{equation*}
\upgamma_h^D(h_i^\pm)=h_i^\pm\pm \alpha_i(h), \quad \upgamma_h^D(X_i^\pm)=q^{\alpha_i(h)/2}X_i^\pm \quad \forall\quad i\in \mbI
\end{equation*}
uniquely extends to a $\C[\![\hbar]\!]$-algebra automorphism of $D(\Uh{\mfb})$. 
Moreover, the natural embedding $U_\hbar\mfg\into U_\hbar\mfg\otimes \msS(\mfh)[\![\hbar]\!]\cong D(\Uh{\mfb})$ identifies $\Uhg$ with the subalgebra fixed by all automorphisms $\upgamma_h^D$:
\begin{equation*}
U_\hbar\mfg=\{x\in D(\Uh{\mfb}): \upgamma_h^D(x)=x\;\forall \; h\in \mfh\}.
\end{equation*}
\end{corollary}
\begin{proof}
For each $h\in \mfh$, the assignment $h_i\mapsto h_i+\alpha_i(h)$, for all $i\in \mbI$, uniquely extends to a $\C[\![\hbar]\!]$-algebra automorphism $\upgamma_h^\msS$ of $\msS(\mfh)[\![\hbar]\!]$. Moreover, it is easy to see that an element $x\in \msS(\mfh)[\![\hbar]\!]$ is fixed by all $\upgamma_h^\msS$ precisely when $x\in \C[\![\hbar]\!]$. It follows readily that the subalgebra of  $U_\hbar\mfg\otimes \msS(\mfh)[\![\hbar]\!]$ fixed by all automorphisms of the form $\upgamma_h^D:=\id_{\Uhg}\otimes \upgamma_h^\msS$ coincides with $\Uhg$. Moreover, by working through the isomorphism of Lemma \ref{L:Uhg(x)S}, one sees that $\upgamma_h^D$ is indeed given as in the statement of the corollary.
\end{proof}
Another consequence of Theorem \ref{T:dbl-Uhg(x)S} is that the Chevalley involution $\omega_\hbar$ of $\Uhg$, defined in Section \ref{ssec:Uhb}, extends to an involutive automorphism $\dot{\omega}_\hbar$ of $D(\Uh{\mfb})\cong \Uhg\otimes \msS(\mfh)[\![\hbar]\!]$ by setting $\dot{\omega}_\hbar=\omega_\hbar \otimes \id_{\msS(\mfh)[\![\hbar]\!]}$. In terms of the generators of Lemma \ref{L:Uhg(x)S}, one has 
\begin{equation*}
\dot{\omega}_\hbar(h_i^\pm)=-h_i^\mp \quad \text{ and }\quad \dot{\omega}_\hbar(X_i^\pm)=-X_i^\mp \quad \forall\; i\in \mbI.
\end{equation*}
Note that $\dot{\omega}_\hbar$ quantizes the Chevalley involution $\dot{\omega}$ on the Lie bialgebra double $D(\mfb^+)$ introduced in Section \ref{ssec:Db-Chev}. In keeping with the notation of that section, we shall henceforth write $\omega_\hbar$ for $\dot{\omega}_\hbar$.

The involution $\omega_\hbar$ of $D(\Uh{\mfb})$ is also a coalgebra anti-automorphism intertwining $S$ and $S^{-1}$. More precisely, one has the following lemma, which follows easily from the formulas above \eqref{DUhb->Uhg}. 
\begin{lemma}\label{L:Chev-D}
The Chevalley involution $\omega_\hbar$ is an isomorphism of topological Hopf algebras $D(\Uh{\mfb})\iso D(\Uh{\mfb})^{\mathrm{cop}}$. In particular, it satisfies
\begin{equation*}
(\omega_\hbar\otimes \omega_\hbar)\circ \Delta^{\mathrm{op}}= \Delta \circ \omega_\hbar, \quad \omega_\hbar \circ S^{-1}=S\circ \omega_\hbar, \quad \veps \circ \omega_\hbar = \veps. 
\end{equation*}
\end{lemma}
%

\subsection{The universal $R$-matrix}\label{ssec:R}
By property \eqref{R:can} in the characterization of the quantum double recalled at the beginning of Section \ref{ssec:DUhb},  $D(\Uh{\mfb})$ is a quasitriangular Hopf algebra, with universal $R$-matrix given by the canonical element 
\begin{equation*}
R^D\in \Uh{\mfb} \otimes \Uh{\mfb}^\ast \subset \Uh{\mfb} \otimes \chk{\Uh{\mfb}} \cong U_\hbar \mfb^+ \otimes \Uh{\mfb^-}\subset D(\Uh{\mfb})^{\otimes 2}.
\end{equation*}
Consequently, $\Uhg$ is quasitriangular with universal $R$-matrix given by the image 
\begin{equation*}
R=(\uppsi_\hbar \otimes \uppsi_\hbar)(R^D),
\end{equation*}
where $\uppsi_\hbar$ is as in \eqref{DUhb->Uhg}. Both of these $R$-matrices have been computed explicitly and studied extensively; for instance, explicit factorizations of $R$ were obtained over thirty years ago in the work of Kirillov--Reshetikhin \cite{KR90} and Levendorskii--Soibelman \cite{LeSo90}; see also \cite{CPBook}*{\S8.3}. We shall not need this level of precision in the present paper. Rather, for our purposes it will be sufficient to note that $R^D$ quantizes the classical universal $r$-matrix $\scriptr^D\in D(\mfb^+)^{\otimes 2}$ defined in Proposition \ref{P:b-dbl}: 
\begin{equation}\label{R^D-limit}
\scriptr^D=\hbar^{-1}\left(R^D-1\right) \mod \hbar,
\end{equation}
and, in addition, admits a multiplicative decomposition consistent with the $\C[\![\hbar]\!]$-module isomorphism $\Uh{\mfb}\cong\Uh{\mfh}\otimes \Uh{\mfn}$ and the $\msQ_+$-grading on $\Uh{\mfn}$. Namely, one has
\begin{equation}\label{R^D-fac}
R^D= q^{\dot{\Omega}_\mfh}\cdot\sum_{\beta\in \msQ_+}R_\beta^+,
\end{equation}
where $R_\beta^+\in (\Uh{\mfn^+})_\beta \otimes (\Uh{\mfn^-})_{-\beta}$ with $R_0^+=1$, and $\dot{\Omega}_\mfh$ is the natural image of the element $\Omega_\mfh\in \mfh\otimes \mfh$ in  $\Uh{\mfb^+}\otimes \Uh{\mfb^-}\subset D(\Uh{\mfb})^{\otimes 2}$:
\begin{equation*}
\dot{\Omega}_\mfh=\sum_{i\in \mbI} (\varpi_i^\vee)^+ \otimes h_i^-\in \Uh{\mfb^+}\otimes \Uh{\mfb^-}\subset D(\Uh{\mfb})^{\otimes 2}.
\end{equation*}
We refer the reader to \cite{CPBook}*{\S8.3C}, for instance, for a proof of the above assertion. 
We now define auxiliary elements $\msR_\omega^\pm\in\Uh{\mfb^\pm}\otimes \Uh{\mfb^\pm}$ by setting
\begin{equation}
\begin{gathered}\label{R-omega}
\msR_\omega^+:=(\id\otimes \omega_\hbar)(R^D) \quad \text{ and }\quad \msR_\omega^-=(\id\otimes \omega_\hbar)((R^D_{21})^{-1}).
\end{gathered}
\end{equation}
%
%
In addition, we set $\mathsf{\Omega}_\mfh^+=-(\id \otimes \omega_\hbar)(\dot{\Omega}_\mfh)$ and $\mathsf{\Omega}_\mfh^-=-(\omega_\hbar\otimes \id)(\dot{\Omega}_\mfh)$. 
\begin{proposition}\label{P:R-omega}
The elements $\msR_\omega^+$ and $\msR_\omega^-$ have the following properties:
\begin{enumerate}[font=\upshape]
\item\label{R-omega:1} They admit multiplicative decompositions
\begin{equation*}
\msR_\omega^+= q^{-\mathsf{\Omega}^+_\mfh}\msR^+ \quad \text{ and }\quad  \msR_\omega^-=\msR^- q^{\mathsf{\Omega}^-_\mfh},
\end{equation*} 
where $\msR^\pm=\sum_{\beta \in \msQ_+}\msR_\beta^\pm \in (\Uh{\mfn^\pm})^{\otimes 2}$ with $\msR_\beta^\pm\in (\Uh{\mfn^\pm_{\pm \beta}})^{\otimes 2}$ and $\msR_0^\pm =1$. 
\item\label{R-omega:2}  They satisfy the following in $D(\Uh{\mfb})^{\otimes 3}$:
\begin{align*}
R_{12}^D (\msR_\omega^\pm)_{13} (\msR_\omega^\pm)_{23}&=(\msR_\omega^\pm)_{23}(\msR_\omega^\pm)_{13}R_{12}^D,\\
R_{12}^D (\msR_\omega^+)_{13} (\msR_\omega^-)_{23}&=(\msR_\omega^-)_{23}(\msR_\omega^+)_{13}R_{12}^D.
\end{align*}
\item\label{R-omega:3} They satisfy the Hopf algebraic relations 
\begin{gather*}
(\id\otimes \Delta)(\msR_\omega^\pm)=(\msR_\omega^\pm)_{12}(\msR_\omega^\pm)_{13},\\
(\id \otimes S)(\msR_\omega^\pm)=(\msR_\omega^\pm)^{-1}
\quad \text{ and }\quad (\id\otimes \veps)(\msR_\omega^\pm)=1.
\end{gather*}
\item\label{R-omega:4}  Their images under $(\id \otimes \upgamma_h^D)$ are given by 
\begin{equation*}
(\id \otimes \upgamma_h^D)(\msR^\pm_\omega)=(q^{-h/2})_1 \msR^\pm_\omega (q^{-h/2})_1 \quad \forall \; h\in \mfh.
\end{equation*}
\end{enumerate}
\end{proposition} 
\begin{proof}
The factorizations of Part \eqref{R-omega:1} follow from the multiplicative decomposition \eqref{R^D-fac} for $R^D$, using that $\omega_\hbar((\Uh{\mfn^\pm})_{\pm\beta})=(\Uh{\mfn^\mp})_{\mp\beta}$ for all $\beta \in \msQ_\pm$. Parts \eqref{R-omega:2} and \eqref{R-omega:3} both follow from standard arguments; see \cite{KS-book}*{Prop.~8.27}, for example. For the sake of completeness, we note that the former is a consequence of the fact that $R^D$ necessarily satisfies the quantum-Yang Baxter equation
%
\begin{equation*}
R^D_{12}R^D_{13}R^D_{23}=R^D_{23}R^D_{13}R^D_{12},
\end{equation*}
while Part \eqref{R-omega:3} follows from Lemma \ref{L:Chev-D} and that $R^D$ satisfies
\begin{gather*}
(\Delta\otimes\id)(R^D)=R^D_{13}R^D_{23}, \quad (\id\otimes \Delta)(R^D)=R^D_{13}R^D_{12},\\
(S\otimes \id)(R^D)=(R^D)^{-1}=(\id \otimes S^{-1})(R^D), \quad (\veps\otimes \id)(R^D)=1=(\id\otimes \veps)(R^D).
\end{gather*}
Consider now Part \eqref{R-omega:4}. The definition of $\upgamma_h^D$ (see Corollary \ref{C:D-fixed}) implies that $\upgamma_h^D|_{\Uh{\mfn^\pm_\beta}}=q^{\beta(h)/2}\id_{\Uh{\mfn^\pm_\beta}}$ for each $\beta\in \msQ_\pm$. Letting $\msR^\pm$ be as in Part \eqref{R-omega:1}, we then have
\begin{gather*}
(\id \otimes \upgamma_h^D)(\msR^\pm)=\sum_{\beta \in \msQ_+}  q^{\beta(h)/2}\msR_\beta^\pm =\sum_{\beta \in \msQ_+}  \mathrm{Ad}((q^{\pm h/2})_2)\msR_\beta^\pm  =(q^{\pm h/2})_2 \msR^\pm (q^{\mp h/2})_2,\\
(\id \otimes \upgamma_h^D)(q^{\mp\mathsf{\Omega}^\pm_\mfh})=(q^{-\sum_{i\in \mbI} \alpha_i(h)(\varpi_i^\vee)^\pm})_1q^{\mp\mathsf{\Omega}^\pm_\mfh}=(q^{-h})_1q^{\mp\mathsf{\Omega}^\pm_\mfh}=q^{\mp\mathsf{\Omega}^\pm_\mfh}(q^{-h})_1.
\end{gather*}
Since $(q^{\pm h/2})_2 \msR^\pm (q^{\mp h/2})_2=(q^{\pm h/2})_1 \msR^\pm (q^{\mp h/2})_1$, it follows by Part \eqref{R-omega:1} that $(\id \otimes \upgamma_h^D)(\msR^\pm_\omega)=(q^{-h/2})_1 \msR^\pm_\omega (q^{-h/2})_1$. \qedhere 
\end{proof}

\subsection{Finite-dimensional representations}\label{ssec:Uhg-fdreps}
A representation $\mcV$ of $\Uhg$ is said to be \textit{finite-dimensional} if it is a free $\C[\![\hbar]\!]$-module of finite rank, and thus realized on a space of the form $V[\![\hbar]\!]$, where $V$ is a finite-dimensional complex vector space. If $\mcV=V[\![\hbar]\!]$ is such a $\Uhg$-module with action given by 
\begin{equation*}
\pi_\hbar: \Uhg\to \End_{\C[\![\hbar]\!]}(\mcV)\cong \End(V)[\![\hbar]\!],
\end{equation*}
then the semiclassical limit $\pi=\bar{\pi}_\hbar:\Uhg\to \End(V)$ equips $V$ with the structure of a finite-dimensional $\mfg$-module. 

The (exact, but non-abelian) category of finite-dimensional representations of $\Uhg$ is well-understood, having been characterized by Tanisaki \cite{Tanisaki91} and Drinfeld \cite{Dr-almost}*{\S4}; see also \cite{Rosso87}. In particular, every finite-dimensional representation of $\Uhg$ decomposes as a direct sum of indecomposable representations, each of which is uniquely characterized up to isomorphism by its semiclassical limit, which is a finite-dimensional irreducible representation of $\mfg$.

More generally, every finite-dimensional representation $V[\![\hbar]\!]$ of $\Uhg$ is uniquely determined, up to isomorphism, by its semiclassical limit $V$. Indeed, the action of $\Uhg$ on $V[\![\hbar]\!]$ can be recovered from $\pi:U(\mfg)\to \End(V)$, up to equivalence, as  the composite 
\begin{equation*}
\Uhg \xrightarrow{\upvarphi} U(\mfg)[\![\hbar]\!]\xrightarrow{\pi} \End(V)[\![\hbar]\!]\cong \End_{\C[\![\hbar]\!]}V[\![\hbar]\!],
\end{equation*}
where $\upvarphi: \Uhg\iso U(\mfg)[\![\hbar]\!]$ is a fixed $\C[\![\hbar]\!]$-algebra isomorphism with semiclassical limit $\bar{\upvarphi}=\id_{U(\mfg)}$,  which exists (and is unique up to conjugation by an element of $1+\hbar U(\mfg)[\![\hbar]\!]$) by the rigidity of semisimple Lie algebras; see \cite{Dr-almost}*{\S4}, \cite{KasBook95}*{\S XVIII.2} and \cite{Ger-deform}. Moreover, this correspondence respects $\mfh$-weight spaces. That is,  $\mcV$ decomposes as $\mcV=\bigoplus_{\mu}V_\mu[\![\hbar]\!]$ where $V_\mu$ is the $\mu\in \mfh^\ast$ weight space of $V$, and one has the equality
\begin{equation*}
V_\mu[\![\hbar]\!]=\{v\in \mcV: h\cdot v= \mu(h)\cdot v \; \forall \; h\in \mfh\subset \Uhg\} \quad \forall\; \mu\in \mfh^\ast.
\end{equation*}
This follows from \cite{Dr-almost}*{Prop.~4.3}, which dictates that $\upvarphi$ may be chosen so that $\upvarphi|_{\mfh}=\id_\mfh$. 
 
\section{The \texorpdfstring{$R$}{R}-matrix construction}\label{sec:URg}

\subsection{The $R$-matrix algebra $\UR{\mfg}$}\label{ssec:URg-def}
As in Section \ref{sec:g-r}, we let $V$ be an arbitrary faithful representation of the finite-dimensional simple Lie algebra $\mfg$, with associated algebra homomorphism $\pi:U(\mfg)\to \End(V)$. Then, by the results recalled in Section \ref{ssec:Uhg-fdreps}, there is a unique, up to isomorphism,  $\Uhg$-module structure on $V[\![\hbar]\!]$ deforming the $\mfg$-module structure on $V$. Moreover, the associated action can be chosen so as to be given by an algebra homomorphism
\begin{equation*}
\pi_\hbar: \Uhg\to \End_{\C[\![\hbar]\!]}(V[\![\hbar]\!])\cong \End(V)[\![\hbar]\!]
\end{equation*}
satisfying $\pi_\hbar|_{\mfh}=\pi|_{\mfh}$ and $\pi_\hbar(\Uhg)\subset \pi(U(\mfg))[\![\hbar]\!]$. We henceforth fix $\pi_\hbar$ with these properties, keeping in mind that any finite-dimensional representation of $\Uhg$ can be realized in this way.
In addition, we extend $\pi_\hbar$ to a representation of $D(\Uh{\mfb})$ by pulling back via the surjection $\uppsi_\hbar$ defined in \eqref{DUhb->Uhg}, and set 
\begin{equation*}
\mrR_\pi=(\pi_\hbar\otimes \pi_\hbar)(R)=(\pi_\hbar\otimes \pi_\hbar)(R^D). 
\end{equation*}
Since $\End_{\C[\![\hbar]\!]}(V[\![\hbar]\!])$ is torsion free and $R^D$ is equal to $1$ modulo $\hbar$, we may define $\dot{\mrR}_\pi\in \End_{\C[\![\hbar]\!]}(V[\![\hbar]\!])^{\otimes 2}\cong \End(V)^{\otimes 2} [\![\hbar]\!]$ by
\begin{equation*}
\dot{\mrR}_\pi=\hbar^{-1}\left(\mrR_\pi-1\right).
\end{equation*}
In particular, this definition and \eqref{R^D-limit} imply that $\dot{\mrR}_\pi$ reduces to $\scriptr_\pi$ modulo $\hbar$. 

With the above notation at our disposal, we are now prepared to introduce the main object of study of the remainder of this article.
\begin{definition}\label{D:URg}
Let $\UR{\mfg}$ denote the unital, associative $\C[\![\hbar]\!]$-algebra topologically generated by $\{t_{ij}^\pm\}_{i,j\in \mcI}$, subject to the relations
\begin{gather}
\mrT_\lambda^\pm=0 \quad \forall \; \lambda \in \dot{\msQ}_\mp,\label{URg-T-tri}\\
[\mrT_2^\pm,\mrT_1^\pm]=[\dot{\mrR}_\pi,\mrT_1^\pm + \mrT_2^\pm]+\hbar\left(\dot{\mrR}_\pi \mrT_1^\pm \mrT_2^\pm - \mrT_2^\pm\mrT_1^\pm \dot{\mrR}_\pi\right), \label{URg-RTT:1}\\
[\mrT_2^-,\mrT_1^+]=[\dot{\mrR}_\pi,\mrT_1^+ + \mrT_2^-]+\hbar\left(\dot{\mrR}_\pi \mrT_1^+ \mrT_2^- - \mrT_2^-\mrT_1^+ \dot{\mrR}_\pi\right),
\label{URg-RTT:2}
\end{gather}
where $\mrT^\pm$ is the generating matrix  
\begin{equation*}
\mrT^\pm = \sum_{i,j\in \mcI}E_{ij}\otimes t_{ij}^\pm \in \End(V)\otimes_\C  \UR{\mfg}. 
\end{equation*}
\end{definition}
\begin{remark}\label{R:URg-def}
The relation \eqref{URg-T-tri} should be viewed as an identity in $\End(V)\otimes_\C \UR{\mfg}$, while \eqref{URg-RTT:1} and \eqref{URg-RTT:2} are relations in $\End(V)^{\otimes 2}\otimes_\C \UR{\mfg}$, with both tensor products taken over $\C$. Here we note that, since $\mcV=V[\![\hbar]\!]$ satisfies $\End_{\C[\![\hbar]\!]}(\mcV)\cong \End(V)[\![\hbar]\!]$ and $V$ is a finite-dimensional complex vector space, one has
\begin{equation*}
\End_{\C[\![\hbar]\!]}(\mcV)\otimes \End_{\C[\![\hbar]\!]}(\mcV)\cong \End(V)^{\otimes 2} \otimes_\C \C[\![\hbar]\!].
\end{equation*}
Hence, $\dot{\mrR}_\pi=(\dot{\mrR}_\pi)_{12}$ may be naturally viewed as an element in $\End(V)^{\otimes 2}\otimes_\C \UR{\mfg}$.
In addition, we emphasize that, just as $\mfg_\scriptr$ depends on the evaluation $\scriptr_\pi$ of $\scriptr$ (see Remark \ref{R:g_r-depends}), the algebra $\UR{\mfg}$ takes as input the evaluation $\mrR_\pi$ of $R$ on the underlying finite-dimensional representation $V[\![\hbar]\!]$. As this representation will remain fixed, suppressing the subscript $\pi$ in the notation $\UR{\mfg}$ shall not cause any ambiguity. 
\end{remark}
Since $\dot{\mrR}_\pi$ reduces to $\scriptr_\pi$ modulo $\hbar$, it follows from the above definition  that $\UR{\mfg}$ is a $\C[\![\hbar]\!]$-algebra deformation of the enveloping algebra $U(\mfg_\scriptr)$ (see Definition \ref{Def:gr}). In more detail, the assignment $\mrT^\pm \mapsto L^\pm$ induces an epimorphism of $\C[\![\hbar]\!]$-algebras $\UR{\mfg}\onto U(\mfg_\scriptr)$ (where $\hbar$ operates as $0$ in $U(\mfg_\scriptr)$) which descends to an isomorphism of $\C$-algebras 
\begin{equation}\label{UR-sc-limit}
\UR{\mfg}/\hbar \UR{\mfg}\iso U(\mfg_\scriptr). 
\end{equation}
We will prove in Section \ref{ssec:URg-main} that $\UR{\mfg}$ is a quantization of the Lie bialgebra $(\mfg_\scriptr,\delta_\scriptr)$ introduced in Theorem \ref{T:g-main}. In particular, it is topologically free and thus isomorphic to $U(\mfg_\scriptr)[\![\hbar]\!]$ as a $\C[\![\hbar]\!]$-module. 

We note that $\mrL^\pm=I+\hbar\mrT^\pm$ satisfy the more familiar matrix relations 
\begin{gather}\label{Eq:RLL1}
\mathrm{L}_\lambda^\pm=0 \quad \forall \; \lambda \in \dot{\msQ}_\mp,\\
\label{Eq:RLL2}\mrR_\pi \mrL_1^\pm \mrL_2^\pm = \mrL_2^\pm \mrL_1^\pm \mathrm{R}_\pi, \quad \mrR_\pi \mrL_1^+ \mrL_2^- = \mrL_2^- \mrL_1^+\mathrm{R}_\pi.
\end{gather}
Indeed, the first relation is a direct consequence of \eqref{URg-T-tri}, while the second and third are consequences of the sequence of equalities
\begin{equation}\label{T->L}
\begin{aligned}
\mrR_\pi &\mrL_1^{\eps_1}\mrL_2^{\eps_2}-\mrL_2^{\eps_2}\mrL_1^{\eps_1}\mrR_\pi\\
&
= \hbar \mrR_\pi \mrT_1^{\eps_1}+ \hbar \mrR_\pi\mrT_2^{\eps_2} +\hbar^2\mrR_\pi \mrT_1^{\eps_1}\mrT_2^{\eps_2}
- \hbar \mrT_1^{\eps_1}\mrR_\pi - \hbar \mrT_2^{\eps_2}\mrR_\pi - \hbar^2  \mrT_2^{\eps_2}\mrT_1^{\eps_1}\mrR_\pi\\
& 
= \hbar [\mrR_\pi, \mrT_1^{\eps_1}+ \mrT_2^{\eps_2}]+ \hbar^2 \left(\mrR_\pi \mrT_1^{\eps_1}\mrT_2^{\eps_2}-  \mrT_2^{\eps_2}\mrT_1^{\eps_1}\mrR_\pi\right),
\end{aligned}
\end{equation}
where $(\eps_1,\eps_2)$ takes value $(\pm,\pm)$ or $(+,-)$. It is the relations \eqref{Eq:RLL1} and \eqref{Eq:RLL2}, rather than \eqref{URg-RTT:1} and \eqref{URg-RTT:2}, which have primarily featured in the literature; see \cite{FRT}, \cite{KS-book}*{\S8.5}, \cite{Majid-book}*{\S4.1} and references therein. The coefficients of $\mathrm{L}^+$ and $\mathrm{L}^-$ do not, however, generate $\UR{\mfg}$ as a topological $\C[\![\hbar]\!]$-algebra. Rather, they naturally arise as a system of topological generators for the Drinfeld--Gavarini subalgebra of $\UR{\mfg}$; see Remark \ref{R:URg-main}. Nonetheless, we have the following useful lemma.
\begin{lemma}\label{L:URg-hom}
Let $\mcA$ be a topologically free $\C[\![\hbar]\!]$-algebra, and suppose that $\mbT^+$ and $\mbT^-$ belong to $\End(V)\otimes_\C \mcA$. Then the assignment $\mrT^\pm \mapsto \mbT^\pm$ extends to a $\C[\![\hbar]\!]$-algebra homomorphism $\UR{\mfg}\to \mcA$ if and only if $\mbL^\pm=I+\hbar \mbT^\pm$ satisfy the relations
\begin{gather}
\mbL_\lambda^\pm=0 \quad \forall\; \lambda\in \dot{\mathsf{Q}}_{\mp},\label{URg-L-triangle}\\
\mathrm{R}_\pi \mbL_1^\pm \mbL_2^\pm = \mbL_2^\pm \mbL_1^\pm \mathrm{R}_\pi, \quad \mathrm{R}_\pi \mbL_1^+ \mbL_2^- = \mbL_2^- \mbL_1^+ \mathrm{R}_\pi. \label{URg-RLL}
\end{gather}
\end{lemma}
\begin{proof}
This is a consequence of the computation \eqref{T->L} with $\mrL^\pm$ and $\mrT^\pm$ replaced by $\mbL^\pm$ and $\mbT^\pm$, respectively. \qedhere
\end{proof}
%
%

\subsection{Automorphisms of $\UR{\mfg}$}\label{ssec:URg-aut}
Consider the group $\mathrm{GL}_I(V)$ of invertible transformations in $\End(V)[\![\hbar]\!]$ based at $I$:
\begin{equation*}
\mathrm{GL}_I(V)=I+\hbar \mfgl(V)[\![\hbar]\!]\subset \End(V)[\![\hbar]\!].
\end{equation*}
Let $\mathrm{GL}_I(\mrR_\pi)$ and $\mathrm{GL}_I(V)^\mfg$ be the subgroups of $\mathrm{GL}_I(V)$ defined by
\begin{gather*}
\mathrm{GL}_I(\mrR_\pi)=\{D\in \mathrm{GL}_I(V): \mathrm{Ad}(D\otimes D)(\mrR_\pi)=\mrR_\pi\}\cap \mfgl(V)^\mfh[\![\hbar]\!],\\
\mathrm{GL}_I(V)^\mfg= \mathrm{GL}_I(V)\cap \mfgl(V)^\mfg[\![\hbar]\!]. 
\end{gather*}
Here  $\mfgl(V)^\mfh$ is the centralizer of $\mfh\cong \pi(\mfh)$ in $\End(V)$, which is nothing but the zero weight space $\mfgl(V)_0\subset \mfgl(V)$ with respect to to the adjoint action of $\mfg\cong \pi(\mfg)$ on $\mfgl(V)$.  

Note that, since $\mathrm{GL}_I(V)^\mfg\subset \mfgl(V)^\mfg[\![\hbar]\!]$, the group $\mathrm{GL}_I(V)^\mfg$ is contained in the centralizer of $\pi_\hbar(\Uh{\mfg})=\pi(U(\mfg))[\![\hbar]\!]$ in $\End(V)[\![\hbar]\!]$.
\begin{proposition}\label{P:aut}
Let $C$ and $D$ belong to $\mathrm{GL}_I(\mrR_\pi)$, and $\mathscr{C}\in \mathrm{GL}_I(V)^\mfg\times \mathrm{GL}_I(V)^\mfg$. Then there exists unique $\C[\![\hbar]\!]$-algebra automorphisms $\uptheta_C^D$ and $\upchi_{\mathscr{C}}$ of $\UR{\mfg}$ satisfying 
\begin{equation*}
\uptheta_C^D(\mrL^\pm)=C\mrL^\pm D \quad \text{ and }\quad \upchi_{\mathscr{C}}(\mrL^\pm)=\mrL^\pm C^\pm,
\end{equation*}
where $C^+$ and $C^-$ are the first and second components of $\mathscr{C}$, respectively.  
\end{proposition}
\begin{proof}
It is a consequence of Theorem \ref{T:URg-main}, proven below, that $\UR{\mfg}$ is a torsion free $\C[\![\hbar]\!]$-module. In what follows we shall make use of this fact. 

Since $\UR{\mfg}$ is torsion free, to see $\uptheta_C^D$ and $\upchi_{\mathscr{C}}$ extend to $\C[\![\hbar]\!]$-algebra endomorphisms of $\UR{\mfg}$, it is enough to show that the matrices $\uptheta_C^D(\mrL^\pm)$ and $\upchi_{\mathscr{C}}(\mrL^\pm)$ satisfy the relations \eqref{URg-L-triangle} and \eqref{URg-RLL} of Lemma \ref{L:URg-hom}. Since each of the matrices $C,D,C^+$ and $C^-$ belong to $\mfgl(V)^\mfh[\![\hbar]\!]$, one has $\uptheta_C^D(\mrL^\pm)_\lambda= C\mrL^\pm_\lambda D$ and $\upchi_{\mathscr{C}}(\mrL^\pm)_\lambda=\mrL^\pm_\lambda  C^\pm$ for any $\lambda\in \mfh^\ast$, from which we see that 
\begin{equation*}
\uptheta_C^D(\mrL^\pm)_\lambda=0=\upchi_{\mathscr{C}}(\mrL^\pm)_\lambda \quad \;\forall \lambda \in \dot{\msQ}_\mp. 
\end{equation*}
Let us now verify that $\uptheta_C^D(\mrL^\pm)$ and $\upchi_{\mathscr{C}}(\mrL^\pm)$ both satisfy the relations of \eqref{URg-RLL}, beginning with the former. Using that $\mrR_\pi C_1 C_2=C_1 C_2 \mrR_\pi$ (as $C\in \mathrm{GL}_I(\mrR_\pi)$) and applying \eqref{URg-RLL} for $\mathrm{L}^\pm$, we obtain
\begin{align*}
\mrR_\pi \uptheta_C^D(\mrL^{\eps_1})_1 \uptheta_C^D(\mrL^{\eps_2})_2
&
=\mrR_\pi C_1C_2 \mrL^{\eps_1}_1  \mrL^{\eps_2}_2 D_1D_2 \\
&
=C_1C_2 \mrR_\pi  \mrL^{\eps_1}_1  \mrL^{\eps_2}_2 D_1D_2\\
&
=C_1C_2     \mrL^{\eps_2}_2\mrL^{\eps_1}_1 \mrR_\pi D_1D_2,
\end{align*}
where $(\veps_1, \veps_2)$ takes value $(\pm,\pm)$ or $(+,-$). Since $\mrR_\pi D_1D_2= D_1 D_2 \mrR_\pi$, this gives 
\begin{equation*}
\mrR_\pi \uptheta_C^D(\mrL^{\eps_1})_1 \uptheta_C^D(\mrL^{\eps_2})_2 = C_1C_2     \mrL^{\eps_2}_2\mrL^{\eps_1}_1 D_1D_2\mrR_\pi = 
\uptheta_C^D(\mrL^{\eps_2})_2\uptheta_C^D(\mrL^{\eps_1})_1\mrR_\pi.
\end{equation*}
Consider now $\upchi_{\mathscr{C}}(\mrL^\pm)$. Using \eqref{URg-RLL} for $\mathrm{L}^\pm$, we obtain 
\begin{equation}\label{chi-hom}
\mrR_\pi \upchi_{\mathscr{C}}(\mrL^{\eps_1})_1 \upchi_{\mathscr{C}}(\mrL^{\eps_2})_2
= 
\mrR_\pi \mrL^{\eps_1}_1  \mrL^{\eps_2}_2 C_1^{\eps_1}C_2^{\eps_2}= \mrL^{\eps_2}_2\mrL^{\eps_1}_1\mrR_\pi C_1^{\eps_1}C_2^{\eps_2}.
\end{equation}
Since $\mrR_\pi\in \pi_\hbar(\Uh{\mfg})\otimes \pi_\hbar(\Uh{\mfg})$ and $C^\pm\in \mathrm{GL}_I(V)^\mfg$,  we have $C_i^\pm\mrR_\pi=\mrR_\pi C_i^\pm$.  Hence, the right-hand side above is just $\upchi_{\mathscr{C}}(\mrL^{\eps_2})_2\upchi_{\mathscr{C}}(\mrL^{\eps_1})_1 \mrR_\pi$, as desired.

To complete the proof of the lemma, we are left to explain why $\uptheta_C^D$ and $\upchi_{\mathscr{C}}$ are invertible. Letting $\mathscr{C}^{-1}$ denote the inverse of $\mathscr{C}$ in the direct product group $\mathrm{GL}_I(V)^\mfg\times \mathrm{GL}_I(V)^\mfg$, we have 
\begin{gather*}
(\uptheta_C^D \circ \uptheta_{C^{-1}}^{D^{-1}} )(\mrL^\pm)=(\uptheta_{C^{-1}}^{D^{-1}} \circ  \uptheta_C^D)(\mrL^\pm)=\mrL^\pm= (\upchi_{\mathscr{C}} \circ \upchi_{\mathscr{C}^{-1}})(\mrL^\pm)= (\upchi_{\mathscr{C}^{-1}} \circ \upchi_{\mathscr{C}} )(\mrL^\pm).
\end{gather*}
As $\UR{\mfg}$ is torsion free, it follows that $\uptheta_{C^{-1}}^{D^{-1}}=(\uptheta_C^D)^{-1}$ and $\upchi_{\mathscr{C}^{-1}}=\upchi_{\mathscr{C}}^{-1}$. \qedhere 
\end{proof}
A particularly important subgroup of $\mathrm{GL}_I(\mrR_\pi)$ consists of the $\hbar$-formal torus 
\begin{equation*}
\mathrm{G}:=\exp(\hbar \pi(\mfh))=\{q^{\pi(h)}: h\in \mfh\}. 
\end{equation*}
%
It will be convenient for us to introduce the abbreviations
\begin{equation}\label{h-aut}
\upvartheta_h:= \uptheta_{q^{\pi(h)}}^I  \quad \text{ and }\quad \upvartheta_h^{\mathsf{Ad}}:= \uptheta_{q^{\pi(h)}}^{q^{-\pi(h)}} \quad \forall\; h\in \mfh.
\end{equation}
%
%
%
\subsection{Quantizing $\mfz_V^\pm$}\label{ssec:quant-z}
We now shift our attention towards establishing quantum analogues of the main results from Section \ref{sec:g-r}. The goal of this particular subsection is to introduce a quantization of the Lie bialgebra $\mfz_V^\pm$ from Section \ref{ssec:g-inv-dual}. 

To this end, recall that $\mfz_V^+$ denotes the Lie bialgebra dual to the Lie algebra of $\mfg$-invariants $\mfgl(V)^\mfg$, equipped with trivial Lie cobracket, and $\mfz_V^-$ denotes the opposite Lie bialgebra to $\mfz_V^+$. As in Section \ref{sec:g-r}, we let $\mfZ^\pm \in \mfgl(V)^\mfg \otimes \mfz_V^\pm$ denote the canonical element. In what follows we shall view $\mfZ^\pm$ as a generating matrix for the $\C[\![\hbar]\!]$-algebra $\msS(\mfz_V^\pm)[\![\hbar]\!]$, where $\msS(\mfz_V^\pm)\cong U(\mfz_V^\pm)$ is the symmetric algebra on $\mfz_V^\pm$.
\begin{lemma}\label{L:q-z_V}
$\msS(\mfz_V^+)[\![\hbar]\!]$ is a quantized enveloping algebra with coproduct $\Delta$, counit $\veps$ and antipode $S$ uniquely determined by
\begin{equation*}
\Delta(\mfZ^+)=\mfZ^+_{[1]}+\mfZ^+_{[2]}+\hbar \mfZ^+_{[1]}\mfZ^+_{[2]}, \quad S(\mfZ^+)=\sum_{b>0} \hbar^{b-1} (-\mfZ^+)^b,\quad \veps(\mfZ^+)=0.
\end{equation*}
\end{lemma}
\begin{proof}
It is clear that the formulas for $\Delta(\mfZ^+)$, $S(\mfZ^+)$ and $\veps(\mfZ^+)$ given in the statement of the lemma uniquely detetermine $\C[\![\hbar]\!]$-algebra homomorphisms 
\begin{gather*}
\Delta:\msS(\mfz_V^+)[\![\hbar]\!]\to \msS(\mfz_V^+)[\![\hbar]\!]\otimes \msS(\mfz_V^+)[\![\hbar]\!], \quad S:\msS(\mfz_V^+)[\![\hbar]\!]\to \msS(\mfz_V^+)[\![\hbar]\!],\\
\veps:\msS(\mfz_V^+)[\![\hbar]\!]\to \C[\![\hbar]\!].
\end{gather*}
Note that on $\mathring{\mfZ}^+:=I+\hbar \mfZ^+$, one has the more familiar matrix formulas 
\begin{equation*}
\Delta(\mathring{\mfZ}^+)=\mathring{\mfZ}^+_{[1]}\mathring{\mfZ}^+_{[2]}, \quad S(\mathring{\mfZ}^+)=(\mathring{\mfZ}^+)^{-1}, \quad \veps(\mathring{\mfZ}^+)=I. 
\end{equation*}
In particular, $(\Delta\otimes \id)\Delta(\mathring{\mfZ}^+)$ and  $(\id\otimes \Delta)\Delta(\mathring{\mfZ}^+)$ both take value $\mathring{\mfZ}^+_{[1]}\mathring{\mfZ}^+_{[2]}\mathring{\mfZ}^+_{[3]}$ and $(\veps\otimes\id)\Delta(\mathring{\mfZ}^+)$ and $(\id \otimes \veps)\Delta(\mathring{\mfZ}^+)$ both take value $\mathring{\mfZ}^+$, where we work through the canonical identifications $\C[\![\hbar]\!]\otimes \msS(\mfz_V^+)[\![\hbar]\!]\cong \msS(\mfz_V^+)[\![\hbar]\!]\cong \msS(\mfz_V^+)[\![\hbar]\!]\otimes \C[\![\hbar]\!]$. 
As $\msS(\mfz_V^+)[\![\hbar]\!]$ is torsion free, it follows that $\Delta$ is both coassociative and counital. By similar reasoning, to check that $S$ is indeed an antipode it is enough to see that the identity 
\begin{equation*}
m\circ (S\otimes \id)\circ \Delta=\iota\circ \veps = m \circ (\id\otimes S)\circ \Delta
\end{equation*}
is satisfied on $\mathring{\mfZ}^+$, where $m$ is the product on $S(\mfz_V^+)[\![\hbar]\!]$ and $\iota:\C[\![\hbar]\!]\into S(\mfz_V^+)[\![\hbar]\!]$ is the unit map. This is a consequence of the relations $(S\otimes\id)\Delta(\mathrm{\mfZ}^+)=(\mathring{\mfZ}^+_{[1]})^{-1}\mathring{\mfZ}^+_{[2]}$ and $(\id \otimes S)\Delta(\mathrm{\mfZ}^+)=\mathring{\mfZ}^+_{[1]}(\mathring{\mfZ}^+_{[2]})^{-1}$.

The above argument proves that $\msS(\mfz_V^+)[\![\hbar]\!]$ is a topological Hopf algebra over $\C[\![\hbar]\!]$, which is a trivially topologically free. As the coefficients of $\mfZ^\pm$ are primitive modulo $\hbar$, it is a flat Hopf algebra deformation of $\msS(\mfz_V^+)\cong U(\mfz_V^+)$, and thus a quantized enveloping algebra with semiclasical limit $U(\mfz_V^+)$.  \qedhere
\end{proof}
\begin{remark}
We note that the Lie bialgebra structure on $\mfz_V^+$ induced by the coproduct $\Delta$ from the above lemma coincides with that from Section \ref{ssec:g-inv-dual}, referred to at the beginning of the section. Indeed, this follows immediately from the relation
\begin{equation*}
\hbar^{-1}(\Delta-\Delta^{\mathrm{op}})(\mfZ^+)=[\mfZ^+_{[1]},\mfZ^+_{[2]}].
\end{equation*}
\end{remark}
Henceforth, we shall denote the topological Hopf algebra introduced in Lemma \ref{L:q-z_V} by $\msS_\hbar(\mfz_V^+)$. That is, one has 
\begin{equation}\label{S_h-def}
\msS_\hbar(\mfz_V^+)=\msS(\mfz_V^+)[\![\hbar]\!]
\end{equation}
as an algebra, with coproduct, antipode, and counit as defined in Lemma \ref{L:q-z_V}. This notation is intended to avoid confusing $\msS_\hbar(\mfz_V^+)$ with the \textit{trivial Hopf algebra deformation} of the symmetric algebra $\msS(\mfz_V^+)$, which will always be denoted $\msS(\mfz_V^+)[\![\hbar]\!]$. In general, these two structures on $\msS(\mfz_V^+)[\![\hbar]\!]$ are not equivalent, though this is always the case when the underlying representation $V$ of $\mfg$ has no repeated composition factors; see Section \ref{sec:Quasi}. 

Similarly, the notation $\msS_\hbar(\mfz_V^-)$ will be used to denote the topological Hopf algebra which coincides with $\msS(\mfz_V^-)[\![\hbar]\!]$ as a $\C[\![\hbar]\!]$-algebra and has coproduct, counit and antipode determined by the requirement that the algebra isomorphism $\msS(\mfz_V^-)[\![\hbar]\!]\iso \msS_\hbar(\mfz_V^+)^{\mathrm{cop}}$, $\mfZ^-\mapsto \mfZ^+$, is a homomorphism of Hopf algebras. In particular, $\msS_\hbar(\mfz_V^-)$ is a quantization of the Lie bialgebra $\mfz_V^-$. Henceforth, we shall set $\dot{\mfZ}^+=\mfZ^+$ and $\dot{\mfZ}^-=S(\mfZ^-)$, so that 
\begin{equation*}
\dot{\mfZ}^-=\sum_{b>0} \hbar^{b-1} (-\mfZ^-)^b \in \mfgl(V)^\mfg \otimes_\C \msS_\hbar(\mfz_V^-).
\end{equation*}
In particular, one has $I+\hbar\dot\mfZ^-=(I+\hbar \mfZ^-)^{-1}$ and $\dot{\mfZ}^\pm=\pm \mfZ^\pm$ modulo $\hbar$.

\subsection{$\UR{\mfg}$ as a quantization of $\mfg_\scriptr$}\label{ssec:URg-main}
In this section we prove the first main result of Section \ref{sec:URg}, which provides a quantization of Theorem \ref{T:g-main}; see Theorem \ref{T:URg-main} below. 
 
As the universal $R$-matrix $R^D$ of $D(\Uh{\mfb})$ is equal to $1$ modulo $\hbar$ and $D(\Uh{\mfb})$ is torsion free, there are unique matrices $\msT^\pm_\omega\in \End(V)\otimes_\C D(\Uh{\mfb})$ such that
\begin{equation}\label{mdL-def}
\msL^\pm_\omega:=(\pi_\hbar \otimes \id)(\msR_\omega^\pm)=I+\hbar\msT^\pm_\omega,
\end{equation}
where we recall from \eqref{R-omega} that $\msR_\omega^+=(\id \otimes \omega)(R^D)$ and $\msR_\omega^-=(\id\otimes \omega)((R^D_{21})^{-1})$. 
In the following theorem, we equip $D(\Uh{\mfb})\otimes \mathsf{S}_\hbar(\mfz_V^+)\otimes \mathsf{S}_\hbar(\mfz_V^-)$ with the standard tensor product of Hopf algebras structure.
\begin{theorem}\label{T:URg-main}
The assignment $\mrT^\pm \mapsto \msT_\omega^\pm + \dot{\mfZ}^\pm + \hbar \msT_\omega^\pm \dot{\mfZ}^\pm$ uniquely extends to an isomorphism of $\C[\![\hbar]\!]$-algebras 
\begin{equation*}
\Upsilon: \UR{\mfg}\iso D(\Uh{\mfb})\otimes \mathsf{S}_\hbar(\mfz_V^+)\otimes \mathsf{S}_\hbar(\mfz_V^-).
\end{equation*}
Moreover, the unique topological Hopf algebra structure on $\UR{\mfg}$ compatible with $\Upsilon$ has coproduct $\Delta_\mrR$, antipode $S_\mrR$ and counit $\veps_\mrR$ uniquely determined by
\begin{equation*}
\Delta_\mrR(\mathrm{L}^\pm)=\mathrm{L}_{[1]}^\pm\mathrm{L}_{[2]}^\pm, \quad S_\mrR(\mathrm{L}^\pm)=(\mathrm{L}^\pm)^{-1}, \quad \veps_\mrR(\mathrm{L}^\pm)=I
\end{equation*}
and realizes $\UR{\mfg}$ as a quantization of the Lie bialgebra $(\mfg_\scriptr,\delta_\scriptr)$. 
\end{theorem}
\begin{proof}
As $D(\Uh{\mfb})\otimes \mathsf{S}_\hbar(\mfz_V^+) \otimes \mathsf{S}_\hbar(\mfz_V^-)$ is topologically free, to prove that $\mrT^\pm \mapsto \msT_\omega^\pm + \dot{\mfZ}^\pm + \hbar \msT_\omega^\pm \dot{\mfZ}^\pm$ extends to a $\C[\![\hbar]\!]$-algebra homomorphism, it is sufficient to show that the matrices
\begin{equation*}
\mbL^\pm:=\msL^\pm_\omega \cdot (I+\hbar \dot{\mfZ}^\pm) 
\end{equation*}
satisfy the relations \eqref{URg-L-triangle} and \eqref{URg-RLL}. Let us set  $\mathring{\mfZ}^\pm=I+\hbar\dot{\mfZ}^\pm=(I+\hbar \mfZ^\pm)^{\pm 1}$, so that $\mbL^\pm=\msL_\omega^\pm\mathring{\mfZ}^\pm$. As $\mfgl(V)^\mfg$ is contained  in the zero weight component of $\mfgl(V)$, to see that the triangularity relations \eqref{URg-L-triangle} hold, it suffices to prove that $(\msL^\pm_\omega)_\lambda=0$ for $\lambda\in \dot{\msQ}_\mp$. As $\msL_\omega^\pm\in \Uh{\mfb^\pm}\otimes \Uh{\mfb^\pm}$ and the notion of $\Uh{\mfg}$-weights and $\mfg$-weights for $V[\![\hbar]\!]$ are compatible (see Section \ref{ssec:Uhg-fdreps}),  $\pi_\hbar(\Uh{\mfb^\pm})\subset \bigoplus_{\mu\in \msQ_{\pm}}\mfgl(V)_\mu[\![\hbar]\!]$, and therefore $(\msL^\pm_\omega)_\lambda=0$ for $\lambda\in \dot{\msQ}_\mp$, as desired.

Let us now turn to establishing that $\mbL^\pm$ satisfy the quadratic matrix relations \eqref{URg-RLL}. 
Applying $\pi_\hbar\otimes \pi_\hbar \otimes \id$ to the relations of \eqref{R-omega:2} in Proposition \ref{P:R-omega}, we find that 
\begin{equation*}
\mathrm{R}_\pi (\msL_\omega^\pm)_1 (\msL_\omega^\pm)_2 = (\msL_\omega^\pm)_2 (\msL_\omega^\pm)_1 \mathrm{R}_\pi \quad \text{ and }\quad \mathrm{R}_\pi (\msL_\omega^+)_1 (\msL_\omega^-)_2 = (\msL_\omega^-)_2 (\msL_\omega^+)_1 \mathrm{R}_\pi.
\end{equation*}
That $\mbL^\pm=\msL_\omega^\pm \mathring{\mfZ}^\pm$ satisfy the relations of \eqref{URg-RLL} now follows by repeating the computation used in the proof of Proposition \ref{P:aut} to show that $\upchi_\mathscr{C}(\mrL^\pm)=\mrL^\pm  C^\pm$ satisfy the same set of relations,
with $\mrL^\pm$ replaced by $\msL_\omega^\pm$ and $C^\pm$ replaced by $\mathring{\mfZ}^\pm$; see \eqref{chi-hom}. 
%
%
 We may thus conclude that the assignment from the statement of the theorem uniquely extends to a $\C[\![\hbar]\!]$-algebra homomorphism 
\begin{equation*}
\Upsilon: \UR{\mfg}\to D(\Uh{\mfb})\otimes \mathsf{S}_\hbar(\mfz_V^+) \otimes \mathsf{S}_\hbar(\mfz_V^-).
\end{equation*}
As the codomain is a topologically free $\C[\![\hbar]\!]$-module and $\UR{\mfg}$ is separated and complete (by definition), to see that $\Upsilon$ is an isomorphism it suffices to prove that its semiclassical limit $\bar{\Upsilon}$ is (see \ref{L1} and \ref{L2} of Section \ref{ssec:h-adic}). As  $\UR{\mfg}$ and $D(\Uh{\mfb})\otimes \mathsf{S}_\hbar(\mfz_V^+) \otimes \mathsf{S}_\hbar(\mfz_V^-)$ are $\C[\![\hbar]\!]$-algebra deformations of  $U(\mfg_\scriptr)$ and $U(D(\mfb^+)\oplus\mfz_V^+\oplus \mfz_V^-)$, respectively, we may view $\bar{\Upsilon}$ as a $\C$-algebra homomorphism 
\begin{equation*}
\bar{\Upsilon}: U(\mfg_\scriptr)\to U(D(\mfb^+)\oplus\mfz_V^+\oplus \mfz_V^-).
\end{equation*}
As the semiclassical limits of $\pi_\hbar$ and $\omega_\hbar$  are the underlying representation $\pi:U(\mfg)\to \End(V)$ and the Chevalley involution $\omega$ on $D(\mfb^+)$ (see Section \ref{ssec:Db-Chev}), respectively, and   $R^D-1=\hbar \scriptr^D$ modulo $\hbar$, the images of $\msT^+_\omega$ and $\msT^-_\omega$ in $\End(V)\otimes U(D(\mfb^+)\oplus\mfz_V^+\oplus \mfz_V^-)$ are given by  $(\pi\otimes \omega)\scriptr^D$ and $-(\pi\otimes \omega)\scriptr^D_{21}$. Hence, $\bar{\Upsilon}$ satisfies
\begin{equation}\label{Upsilon-scl}
\bar{\Upsilon}(L^+)=(\pi\otimes \omega)\scriptr^D +\mfZ^+ \quad \text{ and }\quad \bar{\Upsilon}(L^-)=-(\pi\otimes \omega)\scriptr^D_{21}-\mfZ^-.
\end{equation}
It follows that $\bar{\Upsilon}$ is the algebra homomorphism obtained from the Lie algebra isomorphism $\mfg_\scriptr\iso D(\mfb^+)\oplus\mfz_V^+\oplus \mfz_V^-$ of Theorem \ref{T:g-main} by taking enveloping algebras, which is necessarily an isomorphism by the Poincar\'{e}--Birkhoff--Witt Theorem. This completes the proof of the first assertion of the theorem. 

Consider now the second assertion of the Theorem, regarding the topological Hopf structure on $\UR{\mfg}$ inherited from  $D(\Uh{\mfb})\otimes \mathsf{S}_\hbar(\mfz_V^+) \otimes \mathsf{S}_\hbar(\mfz_V^-)$ via $\Upsilon$. Since $D(\Uh{\mfb})\otimes \mathsf{S}_\hbar(\mfz_V^+)\otimes \mathsf{S}_\hbar(\mfz_V^-)$ is a topologically free $\C[\![\hbar]\!]$-module, $\UR{\mfg}$ is also topologically free. This observation implies that the structure maps $\Delta_\mrR$, $S_\mrR$ and $\veps_\mrR$ are uniquely determined by their values on $\mrL^+$ and $\mrL^-$; see Lemma \ref{L:URg-hom}. Let us now prove that these values are as given in the statement of the theorem. 

By Lemma \ref{L:q-z_V} and the definition of the Hopf algebra structure on $\msS_\hbar(\mfz_V^-)$, we have $\Delta(\mathring{\mfZ}^\pm)=\mathring{\mfZ}^\pm_{[1]}\mathring{\mfZ}^\pm_{[2]}$, $S(\mathring{\mfZ}^\pm)=(\mathring{\mfZ}^\pm)^{-1}$ and $\veps(\mathring{\mfZ}^\pm)=I$. Since $\pi_\hbar(\Uhg)$ is contained in the centralizer of $\mfgl(V)^\mfg$ in $\End(V)[\![\hbar]\!]$,  $(\msL^\pm_\omega)_{[2]}$ and $\mathring{\mfZ}^\pm_{[1]}$ commute.  It  thus suffices to show that in $D(\Uh{\mfb})$ one has the relations
\begin{equation*}
\Delta(\msL^\pm_\omega)=(\msL^\pm_\omega)_{[1]}(\msL^\pm_\omega)_{[2]}, \quad S(\msL^\pm_\omega)=(\msL^\pm_\omega)^{-1}\; \text{ and }\; \veps(\msL^\pm_\omega)=I.
\end{equation*}
These relations are immediately obtained by applying $\pi_\hbar$ to the first tensor factor of each identity in Part \eqref{R-omega:3} of Proposition \ref{P:R-omega}. 

To complete the proof of the theorem, we are left to esablish that the topological Hopf algebra $\UR{\mfg}$ is a quantization of $(\mfg_\scriptr,\delta_\scriptr)$. To this end, note that the formulas for $\Delta_\mrR(\mrL^\pm)$ imply that the coefficients of $\mrT^\pm$ are primitive modulo $\hbar$. Explicitly, one has 
\begin{equation*}
\Delta_{\mrR}(\mathrm{T}^\pm)
=\mathrm{T}_{[1]}^\pm+\mrL_{[1]}^\pm \mrT_{[2]}^\pm=\mrT_{[2]}^\pm+\mrT_{[1]}^\pm \mrL_{[2]}^\pm
=\mathrm{T}_{[1]}^\pm+\mathrm{T}_{[2]}^\pm + \hbar \mathrm{T}_{[1]}^\pm \mathrm{T}_{[2]}^\pm.
\end{equation*}
%
It follows that \eqref{UR-sc-limit} is an isomorphism of Hopf algebras over $\C$. As $\UR{\mfg}$ is topologically free, we may conclude that it is a quantized enveloping algebra with semiclassical limit $U(\mfg_\scriptr)$. Moreover, the Lie bialgebra structure it induces on $\mfg_\scriptr$ has Lie cobracket determined by 
\begin{equation*}
\delta(L^\pm)=\hbar^{-1}(\Delta-\Delta^{\mathrm{op}})(\mrT^\pm) \mod \hbar\UR{\mfg}\otimes\UR{\mfg}.
\end{equation*}
As $(\Delta-\Delta^{\mathrm{op}})(\mrT^\pm)=\hbar[\mathrm{T}_{[1]}^\pm,\mathrm{T}_{[2]}^\pm]$, this coincides with
$[L_{[1]}^\pm,L_{[2]}^\pm]$, which is exactly $\delta_\scriptr(L^\pm)$.
Therefore, $\UR{\mfg}$ is a quantization of the Lie bialgebra $(\mfg_\scriptr,\delta_\scriptr)$.\qedhere
\end{proof}
\begin{remark}\label{R:URg-main}
We now give a few remarks pertinent to the Hopf structure on $\UR{\mfg}$ defined in the above theorem. 
\begin{enumerate}\setlength{\itemsep}{3pt}\setlength{\parskip}{3pt}
\item From the formula for $\Delta_\mrR(\mrL^\pm)$, we obtain that $\Delta^n_\mrR(\mrL^\pm)=\mrL_{[1]}^\pm \mrL_{[2]}^\pm \cdots \mrL_{[n]}^\pm$ for all $n\geq 0$, where $\Delta^n_\mrR:\UR{\mfg}\to \UR{\mfg}^{\otimes n}$ is the iterated coproduct defined in Section \ref{ssec:Uhb}. This implies that the coefficients of $\mrL^+$ and $\mrL^-$ belong to the Drinfeld--Gavarini subalgebra $\UR{\mfg}^\prime$ of $\UR{\mfg}$. Indeed, one has 
\begin{equation*}
(\id-\veps_\mrR)^{\otimes n}\Delta_\mrR^n(\mrL^\pm)=\hbar^n\mrT_{[1]}^\pm \mrT_{[2]}^\pm \cdots \mrT_{[n]}^\pm\in \hbar^n\UR{\mfg}^{\otimes n}
\end{equation*}
which implies that $\mrL^\pm \in \End(V)\otimes_\C \UR{\mfg}^\prime$ by definition of $\UR{\mfg}^\prime$; see Section \ref{ssec:Uhb}. In fact, since $\mrT^\pm$ reduces to the generating matrix $L^\pm$ of $\mfg_\scriptr$ modulo $\hbar$ and the coefficients of $L^\pm$ span $\mfg_\scriptr$, $\UR{\mfg}^\prime$ is generated as a topological algebra (with respect to the subspace topology) by the coefficients of $\mrL^+$ and $\mrL^-$; see \cite{Gav02}*{\S3.5}. 
\item The expressions for $\Delta_\mrR(\mrL^\pm)$, $S_\mrR(\mrL^\pm)$ and $\veps_\mrR(\mrL^\pm)$ given in the above theorem coincide with those output by the $R$-matrix formalism developed by Faddeev, Reshetikhin and Takhtajan; see Theorems 1 and 9 of \cite{FRT}, in addition to \cite{KS-book}*{Prop.~8.32}, for instance. 

We note that it is easy to verify directly that they satisfy the relations \eqref{URg-RLL} with $\mcA$ taken to be $\UR{\mfg}^{\otimes 2}$, $\UR{\mfg}^{\mathrm{op}}$ or $\C[\![\hbar]\!]$, respectively; see Propositions 8.32 and 9.1 of \cite{KS-book} and \cite{MoBook}*{Thm.~1.5.1}, for example. However, it is perhaps less clear that the triangularity relations \eqref{URg-L-triangle} are preserved. This difficulty dissipates if one takes into account that $\mrL^\pm$ satisfies the seemingly stronger relation
$
\mrL^\pm=\sum_{\lambda\in\msQ_\pm} \mrL_\lambda^\pm,
$
which in fact follows from the definition of $\UR{\mfg}$, as is made clear by Theorem \ref{T:URg-main}.
\end{enumerate}
\end{remark}
Recall from Section \ref{ssec:Uhg} that $\Uhg$ is $\msQ$-graded as a topological Hopf algebra. Using Theorem \ref{T:dbl-Uhg(x)S}, we may extend this to a topological $\msQ$-grading on 
\begin{equation*}
D(\Uh{\mfb})\otimes \msS_\hbar(\mfz_V^+)\otimes \msS_\hbar(\mfz_V^-) \cong \Uhg\otimes \msS(\mfh)[\![\hbar]\!] \otimes \msS_\hbar(\mfz_V^+)\otimes \msS_\hbar(\mfz_V^-), 
\end{equation*}
by imposing that $\msS_\hbar(\mfz_V^\pm)$ and $\msS(\mfh)[\![\hbar]\!]$ consist of degree zero elements. It follows by Theorem \ref{T:URg-main} that $\UR{\mfg}$ is also $\msQ$-graded as a topological Hopf algebra. This grading admits a natural interpretation in terms of the automorphisms introduced in Section \ref{ssec:URg-aut}, as we now explain.

As seen in Section \ref{ssec:URg-aut}, the $\hbar$-formal torus $\mathrm{G}=\exp(\hbar \pi(\mfh))$ acts on $\UR{\mfg}$ by the algebra automorphisms $\upvartheta_h^\mathsf{Ad}$ defined in \eqref{h-aut}. Explicitly, one has 
\begin{equation*}
q^{\pi(h)}\cdot \mrT^\pm=\upvartheta_h^\mathsf{Ad}(\mrT^\pm)=q^{\pi(h)}\mrT^\pm q^{-\pi(h)} \quad \forall\; h\in \mfh.
\end{equation*}
\begin{corollary}\label{C:URg-Qgrad}
The action of $\mathrm{G}$ on $\UR{\mfg}$ induces on $\UR{\mfg}$ the structure of a $\msQ$-graded topological Hopf algebra, with homogeneous components
\begin{equation*}
\UR{\mfg}_\beta=\{x\in \UR{\mfg}: \upvartheta_h^{\mathsf{Ad}}(x)=q^{\beta(h)}x \; \forall\; h\in \mfh\} \quad \forall\; \beta \in \msQ.
\end{equation*}
In addition, $\Upsilon$ is an isomorphism of $\msQ$-graded topological Hopf algebras. 
\end{corollary}
\begin{proof}
For each $h\in \mfh$, let $\dot{\upvartheta}_h^{\mathsf{Ad}}$ be the gradation automorphism of $D(\Uh{\mfb})\cong \Uhg\otimes \msS(\mfh)[\![\hbar]\!]$ uniquely determined by 
\begin{equation*}
\dot{\upvartheta}_h^{\mathsf{Ad}}(h_i^\pm)=h_i^\pm \quad \text{ and }\quad \dot{\upvartheta}_h^{\mathsf{Ad}}(X_i^\pm)=q^{\pm\alpha_i(h)}X_i^\pm \quad \forall \; i\in \mbI, 
\end{equation*}
where $h_i^\pm$ and $X_i^\pm$ are as in Lemma \ref{L:Uhg(x)S}. Each $\dot{\upvartheta}_h^{\mathsf{Ad}}$ extends to an automorphism $\ddot{\upvartheta}_h^{\mathsf{Ad}}$ of the topological Hopf algebra $D(\Uh{\mfb})\otimes \msS_\hbar(\mfz_V^+)\otimes \msS_\hbar(\mfz_V^-)$ defined by $\ddot{\upvartheta}_h^{\mathsf{Ad}}=\dot{\upvartheta}_h^{\mathsf{Ad}}\otimes \mathrm{id}_{\msS_\hbar(\mfz_V^+)}\otimes \mathrm{id}_{\msS_\hbar(\mfz_V^-)}$, and one has 
\begin{equation*}
(D(\Uh{\mfb})\otimes \msS_\hbar(\mfz_V^+)\otimes \msS_\hbar(\mfz_V^-))_\beta=\{x: \ddot{\upvartheta}_h^{\mathsf{Ad}}(x)=q^{\beta(h)}x \; \forall \; h\in \mfh\} \quad \forall \; \beta\in \msQ. 
\end{equation*}
It thus suffices to see that $\Upsilon \circ \upvartheta_h=\ddot{\upvartheta}_h^{\mathsf{Ad}}\circ \Upsilon$ for all $h\in \mfh$. This follows readily from the definition of $\Upsilon$ and the observation that $(\id\otimes \dot{\upvartheta}_h^{\mathsf{Ad}})(\msL_\omega^\pm)=q^{\pi(h)}\msL_\omega^\pm q^{-\pi(h)}$, which is deduced from Part \eqref{R-omega:1} of Proposition \ref{P:R-omega} following the same type of argument as used to prove \eqref{R-omega:4} therein. \qedhere
\end{proof}
%
\subsection{Recovering $D(\Uh{\mfb})$}\label{ssec:UR->DUhb}

Recall that, for each $\mathscr{C}$ belonging to the direct product of groups $\msG:=\mathrm{GL}_I(V)^\mfg\times \mathrm{GL}_I(V)^\mfg$,
there is a corresponding automorphism $\upchi_{\mathscr{C}}$ of $\UR{\mfg}$, defined in Proposition \ref{P:aut}. Let $\UR{\mfg}^\upchi$ denote the (closed) unital, associative $\C[\![\hbar]\!]$-subalgebra of $\UR{\mfg}$ consisting of elements fixed by all such automorphisms. That is, if $\UR{\mfg}^{\upchi_{\mathscr{C}}}=\{x\in \UR{\mfg}: \upchi_{\mathscr{C}}(x)=x\}$, then
\begin{equation*}
\UR{\mfg}^\upchi=\bigcap_{\mathscr{C}\in \msG}\UR{\mfg}^{\upchi_{\mathscr{C}}}.
\end{equation*}
The following lemma is a straightforward consequence of the definition of $\upchi_\mathscr{C}$.
\begin{lemma}\label{L:dot-URg-Hopf}
For each $\mathscr{A},\mathscr{C}\in \msG=\mathrm{GL}_I(V)^\mfg\times \mathrm{GL}_I(V)^\mfg$, one has 
\begin{equation*}
\upchi_{\mathscr{A}}\otimes \upchi_{\mathscr{C}} \circ \Delta_{\mrR}=\Delta_{\mrR}\circ \upchi_{\mathscr{A}\cdot \mathscr{C}}
\quad \text{ and }\quad \upchi_{\mathscr{C}^{-1}}\circ S_{\mrR}=S_R\circ \upchi_{\mathscr{C}}.
\end{equation*}
Consequently, $\UR{\mfg}^\upchi$ is a Hopf subalgebra of the topological Hopf algebra $\UR{\mfg}$. 
\end{lemma}
Consider now the Hopf algebra epimorphism 
\begin{equation*}
\dot{\Upsilon}:=(\id_{D(\Uh{\mfb})}\otimes \upepsilon^+\otimes \upepsilon^-)\circ \Upsilon: \UR{\mfg}\onto D(\Uh{\mfb}),
\end{equation*}
where $\upepsilon^\pm$ denotes the counit of $\msS_\hbar(\mfz_V^\pm)$, defined by $\upepsilon^\pm(\mfZ^\pm)=0$ (in particular, $\upepsilon^+=\varepsilon$ from Lemma \ref{L:q-z_V}). Let $\msZ(\UR{\mfg})$ denote the center of the $\C[\![\hbar]\!]$-algebra $\UR{\mfg}$. The following theorem, which provides the second main result of this section, realizes $D(\Uh{\mfb})$ as both a fixed point Hopf subalgebra of $\UR{\mfg}$ and as a quotient by a Hopf ideal generated by certain distinguished central elements. 
\begin{theorem}\label{T:URg->DUhb}
$\dot{\Upsilon}$ restricts to an isomorphism of topological Hopf algebras 
\begin{equation*}
\dot{\Upsilon}|_{\UR{\mfg}^\upchi}:\UR{\mfg}^\upchi\iso D(\Uh{\mfb}).
\end{equation*}
Moreover, there are $\Upsigma^\pm\in \mfgl(V)^\mfg\otimes \msZ(\UR{\mfg})$ whose coefficients generate the kernel of $\dot{\Upsilon}$ as an ideal and admit the following properties:
\begin{enumerate}[font=\upshape]
\item\label{Upsigma:1} The subspace $\mfz_{\Upsigma}^\pm=\{f\otimes \id(\Upsigma^\pm):f\in \mfz_V^\pm\}$  has dimension $\dim\mfgl(V)^\mfg$.
\item\label{Upsigma:2} The $\C[\![\hbar]\!]$-subalgebra of $\UR{\mfg}$ topologically generated by $\mfz_\Upsigma^+$ and $\mfz_\Upsigma^-$ is isomorphic to the symmetric algebra
\begin{equation*}
\msS(\mfz_\Upsigma^+\oplus \mfz_\Upsigma^-)[\![\hbar]\!]\cong \msS(\mfz_\Upsigma^+)[\![\hbar]\!]\otimes \msS(\mfz_\Upsigma^-)[\![\hbar]\!].
\end{equation*}
\item\label{Upsigma:3} The matrices $\mathring{\Upsigma}^\pm= (I+\hbar\Upsigma^\pm)^{\pm 1}$ satisfy the Hopf algebraic relations 
\begin{equation*}
\Delta_\mrR(\mathring{\Upsigma}^\pm)=\mathring{\Upsigma}^\pm_{[1]}\mathring{\Upsigma}^\pm_{[2]}, \quad S_\mrR(\mathring{\Upsigma}^\pm)^{-1},\quad \veps_\mrR(\mathring{\Upsigma}^\pm)=I.
\end{equation*}
%

\end{enumerate}
\end{theorem}
\begin{proof}
First note that, since $\Upsilon$ is an isomorphism of topological Hopf algebras, the matrices $\Upsigma^\pm:=\Upsilon^{-1}(\mfZ^\pm)$ satisfy the properties \eqref{Upsigma:1} -- \eqref{Upsigma:3}. In addition, since the kernel of the epimorphism $\id_{D(\Uh{\mfb})}\otimes \upepsilon^+\otimes \upepsilon^-$ is generated by $\mfz_V^+$ and $\mfz_V^-$ as a topological ideal, the coefficients of $\Upsigma^+$ and $\Upsigma^-$ generate $\Ker(\dot\Upsilon)$ as a topological ideal in $\UR{\mfg}$. 

We are thus left to prove the first assertion of the theorem, concerning the restriction $\dot\Upsilon |_{\UR{\mfg}^\upchi}$. To this end, note that, for each $\mathscr{C}=(C^+,C^-)\in \msG$, the assignment 
\begin{equation*}
\upchi_{\mathscr{C}}^\msS: \mfZ^+ \mapsto \mfZ^+ + \dot{C}^+ + \hbar\mfZ^+ \dot{C}^+, \quad \mfZ^- \mapsto \mfZ^- + \dot{C}^- + \dot{C}^-\hbar\mfZ^-, 
\end{equation*}
uniquely extends to a $\C[\![\hbar]\!]$-algebra automorphism $\upchi_{\mathscr{C}}^\msS$ of $\msS(\mfz_V^+\oplus \mfz_V^-)[\![\hbar]\!]\cong \msS_\hbar(\mfz_V^+)\otimes\msS_\hbar(\mfz_V^-)$, 
where $\dot{C}^\pm\in \mfgl(V)^\mfg[\![\hbar]\!]$ are defined by $C^\pm=(1+\hbar \dot{C}^\pm)^{\pm 1}$. 

Indeed, that $\upchi_{\mathscr{C}}^\msS$ extends to a $\C[\![\hbar]\!]$-algebra endomorphism is immediate. To see that it is in fact an automorphism, note that the elements $\mathring{\mfZ}^\pm=(I+\hbar \mfZ^\pm)^{\pm 1}$ satisfy 
\begin{equation}\label{upchi-mfZ}
\upchi_{\mathscr{C}}^\msS(\mathring{\mfZ}^\pm)=\mathring{\mfZ}^\pm C^\pm, 
\end{equation}
from which it follows readily that $(\upchi_{\mathscr{C}}^\msS)^{-1}=\upchi_{\mathscr{C}^{-1}}^\msS$. Consider now the subalgebra
\begin{equation*}
\msS(\mfz_V^+\oplus \mfz_V^-)[\![\hbar]\!]^{\upchi^\msS}=\bigcap_{\mathscr{C}\in \msG} \msS(\mfz_V^+\oplus \mfz_V^-)[\![\hbar]\!]^{\upchi^\msS_\mathscr{C}}
\end{equation*}
consisting of all elements fixed by all automorphisms $\upchi_{\mathscr{C}}^\msS$. 

\noindent \textit{Claim:} $\msS(\mfz_V^+\oplus \mfz_V^-)[\![\hbar]\!]^{\upchi^\msS}$ is trivial. That is, it is equal to $\C[\![\hbar]\!]$.

\begin{proof}[Proof of claim]
Consider the semiclassical limit $\bar{\upchi}_{\mathscr{C}}^\msS:\msS(\mfz_V^+\oplus \mfz_V^-)\to \msS(\mfz_V^+\oplus \mfz_V^-)$ of $\upchi_{\mathscr{C}}^\msS$. It is a $\C$-algebra automorphism uniquely determined by 
\begin{equation*}
\bar{\upchi}_{\mathscr{C}}^\msS(\mfZ^\pm)=\mfZ^\pm + \dot{C}^\pm_{0}, 
\end{equation*}
where we have written $\dot{C}^\pm=\sum_k \dot{C}^\pm_{k} \hbar^k \in \mfgl(V)^\mfg[\![\hbar]\!]$. Let us expand $\mfZ^\pm$ as $\mfZ^\pm=\sum_j E_j \otimes \msz_j^\pm$, where $\{E_j\}_{j\in \mcJ}$ is any fixed basis of $\mfgl(V)^\mfg$. Then the subalgebra  $\msS(\mfz_V^+\oplus \mfz_V^-)^{\bar{\upchi}^\msS}$ fixed by all these automorphisms consists of all polynomials in $\{\msz_j^+,\msz_j^-\}_{j\in \mcJ}$ which are invariant under the natural action of the addititive group $\C^\mcJ \times \C^\mcJ$ on $\msS(\mfz_V^+\oplus \mfz_V^-)\cong \C[\msz_j^+,\msz_j^-:j\in \mcJ]$ by translations: $\msz_j^\pm \mapsto \msz_j^\pm + c_j^\pm$ for $((c_j^+)_{j\in\mcJ}; (c_j^-)_{j\in \mcJ})\in \C^\mcJ\times \C^\mcJ$. The only such polynomials are the constant polynomials:
\begin{equation*}
\msS(\mfz_V^+\oplus \mfz_V^-)^{\bar{\upchi}^\msS}=\C.
\end{equation*}
Now suppose that $x\in \msS(\mfz_V^+\oplus \mfz_V^-)[\![\hbar]\!]^{\upchi^\msS}$, and in addition that $x\notin \C[\![\hbar]\!]$. Write $x=\sum_k x_k \hbar^k\in \msS(\mfz_V^+\oplus \mfz_V^-)[\![\hbar]\!]$. Let $b\geq 0$ be minimal such that $x_b\notin \C$. Then $y=\sum_{k\geq b} x_k \hbar^{k-b}$ also belongs to $\msS(\mfz_V^+\oplus \mfz_V^-)[\![\hbar]\!]^{\upchi^\msS}$.
It follows that $x_b\in \msS(\mfz_V^+\oplus \mfz_V^-)^{\bar{\upchi}^\msS}$, and hence $x_b\in \C$ by the above. This contradicts the choice of $b$, and therefore we may conclude that we must have $x\in \C[\![\hbar]\!]$, as desired. \qedhere
\end{proof}

Consider now the $\C[\![\hbar]\!]$-algebra automorphism $\dot{\upchi}_\mathscr{C}=\id_{D(\Uh{\mfb})}\otimes \upchi_{\mathscr{C}}^\msS$ of $D(\Uh{\mfb})\otimes \msS(\mfz_V^+\oplus \mfz_V^-)[\![\hbar]\!]$, for each $\mathscr{C}\in \msG$. Using the above claim and the exactness of the topological tensor product $\otimes$ over $\C[\![\hbar]\!]$, we see that the subalgebra fixed by all such automorphisms is given by
\begin{equation*}
(D(\Uh{\mfb})\otimes \msS(\mfz_V^+\oplus \mfz_V^-)[\![\hbar]\!])^{\dot{\upchi}}
= D(\Uh{\mfb})\otimes \msS(\mfz_V^+\oplus \mfz_V^-)[\![\hbar]\!]^{\upchi^\msS} =D(\Uh{\mfb}). 
\end{equation*}
Finally, it follows from the definitions of $\upchi_{\mathscr{C}}$ and $\Upsilon$ (see Proposition \ref{P:aut} and Theorem \ref{T:URg-main}), together with the equality \eqref{upchi-mfZ}, that $\dot{\upchi_{\mathscr{C}}} \circ \Upsilon = \Upsilon \circ \upchi_{\mathscr{C}}$, where we work through the identification of algebras
\begin{equation*}
D(\Uh{\mfb})\otimes \msS(\mfz_V^+\oplus \mfz_V^-)[\![\hbar]\!]\cong D(\Uh{\mfb})\otimes \mathsf{S}_\hbar(\mfz_V^+) \otimes \mathsf{S}_\hbar(\mfz_V^-).
\end{equation*}
Therefore, we have $\UR{\mfg}^\upchi=\Upsilon^{-1}(D(\Uh{\mfb}))$, as claimed. \qedhere
\end{proof}
\begin{remark}
In Section \ref{ssec:Z-find}, we will obtain uniform formulas for the coefficients of the central matrices $\Upsigma^\pm$ in terms of the entries of $\mrT^\pm$, under the assumption that the underlying $\mfg$-module $V$ is a direct sum of \textit{distinct} irreducible representations. 
\end{remark}

\subsection{Recovering $\Uh{\mfg}$}\label{ssec:UR->Uhg}
As the quantum double of $\Uh{\mfb}$ admits the $\C[\![\hbar]\!]$-algebra decomposition $D(\Uh{\mfb})\cong \Uh{\mfg}\otimes \msS(\mfh)[\![\hbar]\!]$, the quantized enveloping algebra $\Uh{\mfg}$ can also be recovered as a both a quotient and subalgebra of $\UR{\mfg}$. 
In this subsection, we unravel this in detail. 

Let $\mrL_0^\pm=I+\hbar \mrT_0^\pm$, where $\mrT^\pm_0=\mathbf{1}_0(\mrT^\pm)\in \mfgl(V)^\mfh \otimes_\C \UR{\mfg}$ and we recall that $\mathbf{1}_0$ is the projection of $\End(V)=\mfgl(V)$ onto its zero weight component $\mfgl(V)_0=\mfgl(V)^\mfh$. In what follows, we set $\mathring{\Upsigma}=\mathring{\Upsigma}^+\mathring{\Upsigma}^-$, where $\mathring{\Upsigma}^\pm=(I+\hbar\Upsigma^\pm)^{\pm 1}$, as in Theorem \ref{T:URg->DUhb}.
\begin{lemma}\label{L:Theta}
There is a unique matrix $\Theta\in \pi(\mfh) \otimes_\C \msZ(\UR{\mfg})$ satisfying
\begin{equation*}
\mrL_0^+\mrL_0^-=q^\Theta \mathring{\Upsigma}\in \mfgl(V)^\mfh \otimes_\C \msZ(\UR{\mfg}). 
\end{equation*}
Moreover, one has $\Upsilon(\Theta)=-\sum_{i\in \mbI} \pi(\varpi_i^\vee)\otimes \zeta(h_i)\in \pi(\mfh)\otimes_\C \msS(\mfh)[\![\hbar]\!]$.
\end{lemma}
\begin{proof}
If $\Theta$ and $\Theta^\prime$ both lay in $\pi(\mfh) \otimes_\C \msZ(\UR{\mfg})$ and solve $\mrL_0^+\mrL_0^-=q^\msX \mathring{\Upsigma}$ for $\msX$, then they commute and thus satisfy $q^{\Theta-\Theta^\prime}=1$. As $\End(V)\otimes_\C \UR{\mfg}$ is topologically free, this is only possible if $\Theta=\Theta^\prime$. Let us now establish that a solution $\Theta$ to this equation exists, and is given as claimed. 

Since $\mathring{\mfZ}^\pm=(I+\hbar \mfZ^\pm)^{\pm 1}\in \mfgl(V)^\mfh \otimes_\C \msS_\hbar(\mfz_V^\pm)$ and $\Upsilon(\mrL_0^\pm)=\msL_\omega^\pm \mathring{\mfZ}^\pm$ (see \eqref{mdL-def} and Theorem \ref{T:URg-main}), Part \eqref{R-omega:1} of Proposition \ref{P:R-omega} implies that 
\begin{equation*}
\Upsilon(\mrL_0^\pm)=(\msL_\omega^\pm)_0\mathring{\mfZ}^\pm= q^{\mp\sum_{i\in \mbI} \pi(\omega_i^\vee)\otimes h_i^\pm}\mathring{\mfZ}^\pm.
\end{equation*}
It follows readily that $\Theta=-\sum_{i\in \mbI} \pi(\varpi_i^\vee)\otimes \Upsilon^{-1}(\zeta(h_i))$ is the sought after solution. Indeed, we have
\begin{equation*}
\Upsilon(\mrL_0^+\mrL_0^-)=q^{\sum_{i\in \mbI} \pi(\omega_i^\vee)\otimes (h_i^- - h_i^+)}\mathring{\mfZ}^+\mathring{\mfZ}^-
=
q^{-\sum_{i\in \mbI} \pi(\omega_i^\vee)\otimes \zeta(h_i)}\Upsilon(\mathring{\Upsigma}). \qedhere
\end{equation*}
\end{proof}
Recall that, for each $h\in \mfh$, $\upvartheta_h$ and $\upvartheta^{\mathsf{Ad}}_{h/2}$ are the automorphisms of $\UR{\mfg}$ given by \eqref{h-aut}. Let us introduce an additional automorphism $\upgamma_h$ by setting
\begin{equation}\label{upgamma-h}
\upgamma_h:=\upvartheta_{-h} \circ \upvartheta^{\mathsf{Ad}}_{h/2} \quad \forall\quad h\in \mfh.
\end{equation}
Explicitly, one has $\upgamma_h(\mrL^\pm)=q^{-\pi(h)/2}\mrL^\pm q^{-\pi(h)/2}$. In particular, the automorphism $\upgamma_h$ commutes with $\upchi_\mathscr{C}$ for each $\mathscr{C}\in \msG$, and so $\upgamma_h$ restricts to an automorphism of $\UR{\mfg}^{\upchi}$. Consider now the subalgebra $\UR{\mfg}^{\upchi\circ\upgamma}$ of $\UR{\mfg}$ fixed by all automorphisms of the form $\upchi_\mathscr{C}\circ \upgamma_h$: 
\begin{equation*}
\UR{\mfg}^{\upchi\circ\upgamma}=\bigcap_{(h,\mathscr{C})\in \mfh\times \msG}\UR{\mfg}^{\upchi_\mathscr{C}\circ \upgamma_h} = \bigcap_{h\in\mfh}\left(\UR{\mfg}^{\upchi}\right)^{\upgamma_h}.
\end{equation*}
We will prove in Theorem \ref{T:URg->Uhg} that $\UR{\mfg}^{\upchi\circ \upgamma}$ is isomorphic to $\Uhg$ as a $\C[\![\hbar]\!]$-algebra. To make this precise, consider the Hopf algebra epimorphism 
\begin{equation*}
\ddot{\Upsilon}:=\uppsi_\hbar\circ \dot{\Upsilon}: \UR{\mfg}\onto \Uhg,
\end{equation*}
where we recall from \eqref{DUhb->Uhg} that $\uppsi_\hbar: D(\Uh{\mfb})\onto \Uhg$ coincides with the projection $\id_{\Uhg}\otimes \upepsilon_{\mfh}: \Uh{\mfg}\otimes \msS(\mfh)[\![\hbar]\!]\onto \Uh{\mfg}$ under the identification $D(\Uh{\mfb})\cong \Uh{\mfg}\otimes \msS(\mfh)[\![\hbar]\!]$ of Theorem \ref{T:dbl-Uhg(x)S}, where $\upepsilon_\mfh:\msS(\mfh)[\![\hbar]\!]\to \C[\![\hbar]\!]$ is the counit. We then have the following theorem, which provides the third main result of this section.
\begin{theorem}\label{T:URg->Uhg}
$\ddot{\Upsilon}$ restricts to an isomorphism of $\C[\![\hbar]\!]$-algebras
\begin{equation*}
\ddot{\Upsilon}|_{\UR{\mfg}^{\upchi\circ \upgamma}}:\UR{\mfg}^{\upchi\circ\upgamma}\iso \Uhg.
\end{equation*}
Moreover, the coefficients of $\Theta$, $\Upsigma^+$ and $\Upsigma^-$ generate the kernel of $\ddot{\Upsilon}$ as an ideal.
\end{theorem}
\begin{proof}
Since $\Upsilon(\Theta)=-\sum_{i\in \mbI} \pi(\varpi_i^\vee)\otimes \zeta(h_i)$ is a generating matrix for the central subalgebra $\msS(\mfh)[\![\hbar]\!]\subset D(\Uh{\mfb})\cong \Uh{\mfg}\otimes \msS(\mfh)[\![\hbar]\!]$ and  $\mfZ^\pm=\Upsilon(\Upsigma^\pm)$ generates $\msS_\hbar(\mfz_V^\pm)$ as a topological algebra, the kernel of $\ddot{\Upsilon}$ is generated as a topological ideal by $\Theta$, $\Upsigma^+$ and $\Upsigma^-$.

Let us now prove the first statement of the theorem.
Consider the automorphism $\upgamma_h^\Upsilon$ of $D(\Uh{\mfb})\otimes \mathsf{S}_\hbar(\mfz_V^+) \otimes \mathsf{S}_\hbar(\mfz_V^-)$ induced by $\upgamma_h$ via the isomorphism $\Upsilon$ of Theorem \ref{T:URg-main}. Recall that $\Upsilon$ is uniquely determined by the property that $\Upsilon(\mrL^\pm)=\msL^\pm_\omega\mathring{\mfZ}^\pm$, where $\msL^\pm_\omega$ is defined in \eqref{mdL-def} and $\mathring{\mfZ}^\pm=(I+\hbar \mfZ^\pm)^{\pm 1}$. 
Since every element of the group $\mathrm{G}=\exp(\hbar \pi(\mfh))$ commutes with $\mathring{\mfZ}^+$ and $\mathring{\mfZ}^-$ (as they lay in $\mfgl(V)^\mfg \otimes_\C \left(\mathsf{S}_\hbar(\mfz_V^+) \otimes \mathsf{S}_\hbar(\mfz_V^-)\right)$), we have 
\begin{equation*}
\upgamma_h^\Upsilon(\msL^\pm_\omega\mathring{\mfZ}^\pm)=q^{-\pi(h)/2}\msL^\pm_\omega q^{-\pi(h)/2}\mathring{\mfZ}^\pm.
\end{equation*}
By Theorem \ref{T:URg->DUhb} and Corollary \ref{C:D-fixed}, it therefore suffices to show that the automorphism $\upgamma_h^D$ of Corollary \ref{C:D-fixed} satisfies $\upgamma_h^D(\msL^\pm_\omega)=q^{-\pi(h)/2}\msL^\pm_\omega q^{-\pi(h)/2}$. This follows by applying $\pi\otimes \id$ to the identity $(\id \otimes \upgamma_h^D)(\msR_\omega^\pm)=(q^{-h/2})_1\msR_\omega^\pm (q^{-h/2})_1$ which was established in Part \eqref{R-omega:4} of Proposition \ref{P:R-omega}.\qedhere
\end{proof}

\subsection{The quantum Borel subalgebras}\label{ssec:URb}

To conclude Section \ref{sec:URg}, we apply the above machinery to quantize (and upgrade) the results of Section \ref{ssec:g-borel}.

\begin{definition}\label{D:URb}
Let $\UR{\mfb^+}$ and $\UR{\mfb^-}$ denote the unital, associative $\C[\![\hbar]\!]$-algebras topologically generated by $\{t_{ij}^+\}_{i,j\in \mcI}$ and $\{t_{ij}^-\}_{i,j\in \mcI}$, respectively, subject to the relations
\begin{gather*}
\mrT_\lambda^\pm=0 \quad \forall \; \lambda \in \dot{\msQ}_\mp,\\
[\mrT_2^\pm,\mrT_1^\pm]=[\dot{\mrR}_\pi,\mrT_1^\pm + \mrT_2^\pm]+\hbar\left(\dot{\mrR}_\pi \mrT_1^\pm \mrT_2^\pm - \mrT_1^\pm\mrT_2^\pm \dot{\mrR}_\pi\right), 
\end{gather*}
where $\mrT^\pm$ is the generating matrix  
\begin{equation*}
\mrT^\pm = \sum_{i,j\in \mcI}E_{ij}\otimes t_{ij}^\pm \in \End(V)\otimes_\C  \UR{\mfb^\pm}. 
\end{equation*}
\end{definition}
Here we follow the same conventions as in Section \ref{ssec:g-borel}: each symbol $\pm$ and $\mp$ takes only its upper value for $\UR{\mfb^+}$, and its lower value for $\UR{\mfb^-}$. 
This definition implies that, for each value of $\pm$, there is a $\C[\![\hbar]\!]$-algebra homomorphism
\begin{equation*}
\imath^{\scriptscriptstyle{\pm}}:\UR{\mfb^\pm}\to \UR{\mfg}, 
\end{equation*}
uniquely determined by $\imath^{\scriptscriptstyle{\pm}}(\mrT^\pm)=\mrT^\pm$. We then have the following corollary, which provides a quantization of Corollary \ref{C:b_r}. 
\begin{corollary}\label{C:URb} 
The composite $\Upsilon^{\scriptscriptstyle{\pm}}=\Upsilon \circ \imath^{\scriptscriptstyle{\pm}}$ is an isomorphism of $\C[\![\hbar]\!]$-algebras
\begin{equation*}
\Upsilon^{\scriptscriptstyle{\pm}}:\UR{\mfb^\pm}\iso \Uh{\mfb^\pm} \otimes \msS_\hbar(\mfz_V^\pm).
\end{equation*}
In particular, $\imath^{\scriptscriptstyle{\pm}}$ is injective and identifies $\UR{\mfb^\pm}$ with the Hopf subalgebra of $\UR{\mfg}$ generated by the coefficients of $\mrT^\pm$. 
\end{corollary}
\begin{proof}
This follows by the same argument as used to establish that $\Upsilon$ is an isomorphism in the proof of Theorem \ref{T:URg-main}, where the role of Theorem \ref{T:g-main} is played instead by Corollary \ref{C:b_r}.  We refer the reader to the proof of Theorem \ref{T:URg-main}, and to \eqref{Upsilon-scl} in particular, for further details.
\end{proof}

Next, note that for any $\mathscr{C}\in \mathrm{GL}_I(V)^\mfg \times \mathrm{GL}_I(V)^\mfg $ the automorphism $\upchi_\mathscr{C}$ of $\UR{\mfg}$ defined in Proposition \ref{P:aut} restricts to an automorphism of $\UR{\mfb^\pm}\cong \imath^{\scriptscriptstyle{\pm}}(\UR{\mfb^\pm})$. This automorphism only depends on one component of $\mathscr{C}$: if $\mathscr{C}$ and $\mathscr{D}$ satisfy $\mathscr{C}\mathscr{D}^{-1}\in \mathrm{GL}_I(V)^\mfg \times \{I\}$ then 
\begin{equation*}
\upchi_{\mathscr{C}}|_{\UR{\mfb^+}}=\upchi_{\mathscr{D}}|_{\UR{\mfb^+}}
\end{equation*}
and, similarly, if $\mathscr{C}\mathscr{D}^{-1}\in \{I\}\times \mathrm{GL}_I(V)^\mfg$ then $\upchi_{\mathscr{C}}|_{\UR{\mfb^-}}=\upchi_{\mathscr{D}}|_{\UR{\mfb^-}}$. Let $\UR{\mfb^\pm}^\upchi$ denote the subalgebra of $\UR{\mfb^\pm}$ fixed by all such automorphisms. That is, one has 
\begin{equation*}
\UR{\mfb^\pm}^\upchi=\UR{\mfb^\pm}\cap \UR{\mfg}^\upchi,
\end{equation*}
where we recall that $\UR{\mfg}^\upchi$ is the subalgebra of $\UR{\mfg}$ fixed by all $\upchi_\mathscr{C}$; see Section \ref{ssec:UR->DUhb}. 
Consider now the Hopf algebra homomorphism 
\begin{equation*}
\dot{\Upsilon}^{\scriptscriptstyle{\pm}}:=\dot{\Upsilon}\circ \imath^{\scriptscriptstyle{\pm}}:\UR{\mfb^\pm}\onto \Uh{\mfb^\pm},
\end{equation*}
where $\dot{\Upsilon}$ is as in Theorem \ref{T:URg->DUhb}. Equivalently, $\dot{\Upsilon}^{\scriptscriptstyle{\pm}}=(\id_{\Uh{\mfb^\pm}}\otimes \upepsilon^\pm)\circ  \Upsilon^{\scriptscriptstyle{\pm}}$, where $\upepsilon^\pm:\msS_\hbar(\mfz_V^\pm)\to \C[\![\hbar]\!]$ is the counit. Combining Theorem \ref{T:URg->DUhb} with Corollary \ref{C:URb} then outputs the following characterizations of $\Uh{\mfb^\pm}$ as a subalgebra and quotient of $\UR{\mfb^\pm}$.
\begin{corollary}\label{C:URb->Ub}
$\dot{\Upsilon}^{\scriptscriptstyle{\pm}}$ restricts to an isomorphism of topological Hopf algebras 
\begin{equation*}
\dot{\Upsilon}^{\scriptscriptstyle{\pm}}|_{\UR{\mfb^\pm}^\upchi}:\UR{\mfb^\pm}^\upchi\iso \Uh{\mfb^\pm}.
\end{equation*}
Moreover, the kernel of $\dot{\Upsilon}^{\scriptscriptstyle{\pm}}$ is generated as an ideal by the coefficients of the central matrix $\Upsigma^\pm\in \mfgl(V)^\mfg\otimes \msZ(\UR{\mfb^\pm})$.
\end{corollary}
Here we note that $\Upsigma^\pm$ is as defined in the statement of Theorem \ref{T:URg->DUhb}, and may be viewed as an element of $\mfgl(V)^\mfg\otimes \msZ(\UR{\mfb^\pm})$ as a consequence of Corollary \ref{C:URb}. Indeed, one has $\Upsigma^\pm=(\Upsilon^{\scriptscriptstyle{\pm}})^{-1}(\mfZ^\pm)$, where we work through the identification $\UR{\mfb^\pm}\cong\imath^{\scriptscriptstyle{\pm}}(\UR{\mfb^\pm})\subset \UR{\mfg}$. 
%
\section{Quasitriangularity and the space of \texorpdfstring{$\mfg$}{g}-invariants}\label{sec:Quasi}

In this section, we address the problem of obtaining a sufficient and necessary condition on the finite-dimensional $\mfg$-module $V$ for which $\UR{\mfg}$ is quasitriangular and, in addition,  isomorphic to the quantum double of $\UR{\mfb}:=\UR{\mfb^+}$; see Theorem \ref{T:Quasi}. 
\subsection{Characterizing $\msS_\hbar(\mfz_V^+)$ as a trivial deformation}\label{ssec:g-invar-comm}
To begin, we prove the following elementary lemma which characterizes the cocommutativity of the topological Hopf algebra $\msS_\hbar(\mfz_V^+)$.
\begin{lemma}\label{L:gl(V)^g}
The following three statements are equivalent: 
\begin{enumerate}[font=\upshape]
\item\label{gl(V)^g:1} The finite-dimensional $\mfg$-module $V$ is multiplicity free.
\item\label{gl(V)^g:2} The algebra of invariants $\mfgl(V)^\mfg=\End_\mfg(V)$ is commutative.
\item\label{gl(V)^g:3} The topological Hopf algebra $\msS_\hbar(\mfz_V^+)$ is cocommutative.
\end{enumerate}
\end{lemma}
\begin{proof}
Let $\{V_j\}_{j\in \mcJ}$ be the distinct composition factors of $V$, and write $n_j$ for the multiplicity of $V_j$, so that $V\cong\bigoplus_{j\in \mcJ} V_j^{n_j}$ as a $\mfg$-module. Then, by Schur's lemma, we have
\begin{equation*}
\End_\mfg(V)\cong  \bigoplus_{j\in \mcJ}\End_\mfg(V_j^{n_j})\cong\bigoplus_{j\in \mcJ} \End(\C^{n_j})
\end{equation*}
and so $\End_\mfg(V)$ is commutative if and only if $n_j=1$ for all $j\in \mcJ$. This establishes the equivalence of \eqref{gl(V)^g:1} and \eqref{gl(V)^g:2}. 

Let us now argue that \eqref{gl(V)^g:2} and \eqref{gl(V)^g:3} are equivalent. By Lemma \ref{L:q-z_V}, the generating matrix $\mfZ^+$ of $\msS_\hbar(\mfz_V^+)$ satisfies $\Delta(\mfZ^+)=\mfZ^+_{[1]}+\mfZ^+_{[2]}+\hbar \mfZ^+_{[1]}\mfZ^+_{[2]}$, and thus $\msS_\hbar(\mfz_V^+)$ is cocommutative if and only if $[\mfZ^+_{[1]},\mfZ^+_{[2]}]=0$. Since $\mfZ^+$ belongs to $\mfgl(V)^\mfg \otimes_\C \msS_\hbar(\mfz_V^+)$ and its coefficients span the $\dim \mfgl(V)^\mfg$ space $\mfz_V^+\subset \msS_\hbar(\mfz_V^+)\cong \msS(\mfz_V^+)[\![\hbar]\!]$, the commutator $[\mfZ^+_{[1]},\mfZ^+_{[2]}]$ vanishes if and only if $\mfgl(V)^\mfg$ is commutative.\qedhere
\end{proof}
 For the remainder of this subsection, we narrow our focus to the case where $V$ is multiplicity free, as in the above lemma. 
 
 Let $\msZ^\pm \in \mfgl(V)^\mfg \otimes_\C \msS_\hbar(\mfz_V^\pm)$ be the unique solution of the equation
\begin{equation}\label{def:msZ}
q^{\mathsf{Z}^\pm}:=\exp\!\left(\tfrac{\hbar}{2} \mathsf{Z}^\pm\right)=I+\hbar\mfZ^\pm. 
\end{equation}
Then $\msZ^\pm$ coincides with $2\cdot\mfZ^\pm$ modulo $\hbar$ and, since $\mfgl(V)^\mfg \otimes_\C \msS_\hbar(\mfz_V^\pm)$ is commutative, Lemma \ref{L:q-z_V} implies that the coefficients of $\mathsf{Z}^\pm$ are primitive elements:
\begin{equation*}
\Delta(\msZ^\pm)=\msZ^\pm_{[1]}+\msZ^\pm_{[2]}.
\end{equation*}
Consequently, if $\mfz_\msZ^\pm$ denotes the complex vector space spanned by the coefficients of $\msZ^\pm$, then the assignment $\msZ^\pm \mapsto \msZ^\pm$ uniquely extends to an isomorphism of topological Hopf algebras 
\begin{equation*}
\msS(\mfz_\msZ^\pm)[\![\hbar]\!]\iso \msS_\hbar(\mfz_V^\pm),
\end{equation*}
where $\msS(\mfz_\msZ^\pm)[\![\hbar]\!]$ is the trivial Hopf algebra deformation of the symmetric algebra  $\msS(\mfz_\msZ^\pm)\cong U(\mfz_\msZ^\pm)$. It follows from this observation and the general theory of quantum duality that $\msS(\mfz_\msZ^-)[\![\hbar]\!]$ (and thus $\msS_\hbar(\mfz_V^-)$) can be identified with the quantized enveloping algebra dual to $\msS(\mfz_\msZ^+)[\![\hbar]\!]\cong \msS_\hbar(\mfz_V^+)$.
 
For the sake of completeness, let us spell out some of the relevant details.
Let
\begin{equation*}
\msR_\hbar(\msS(\mfz_\msZ^+))=\bigoplus_{n\in \N}\hbar^n \msS^n( \mfz_\msZ^+) \subset \msS(\mfz_\msZ^+)[\hbar]
\end{equation*}
 denote the Rees algebra of the symmetric algebra $\msS(\mfz_\msZ^+)$, with respect to its standard filtration. Equivalently, it is the subalgebra $\msS(\hbar\mfz_\msZ^+)[\hbar]$ of $\msS(\mfz_\msZ^+)[\hbar]$.  By Proposition 3.8 of \cite{KasTu00}, the Drinfeld--Gavarini subalgebra $\msS(\mfz_\msZ^+)[\![\hbar]\!]^\prime$ of $\msS(\mfz_\msZ^+)[\![\hbar]\!]$ (see Section \ref{ssec:Uhb}) is given explicitly by
\begin{equation*}
\msS(\mfz_\msZ^+)[\![\hbar]\!]^\prime=\widehat{\msR_\hbar(\msS(\mfz_\msZ^+))}\subset \msS(\mfz_\msZ^+)[\![\hbar]\!],
\end{equation*}
where $\msR_\hbar(\msS(\mfz_\msZ^+))$ is completed with respect to its natural grading. 
The quantized enveloping algebra dual of $\msS(\mfz_\msZ^+)[\![\hbar]\!]$ can then be identified with the trivial deformation of the graded dual $\msS(\mfz_\msZ^+){\vphantom{)}}^\star =\bigoplus_{n\in \N} \msS^n(\mfz_\msZ^+)^\ast$. Explicitly, there is an isomorphism $\uptheta:\msS(\mfz_\msZ^+){\vphantom{)}}^\star[\![\hbar]\!]\iso \msS(\mfz_\msZ^+)[\![\hbar]\!]{\vphantom{)}}^\bullet$ uniquely determined by 
\begin{equation*}
\uptheta(f)(\hbar^n x)=f(x) \quad \forall \; x\in \msS^n(\mfz_\msZ^+)\; \text{ and }\; f\in \msS(\mfz_\msZ^+){\vphantom{)}}^\star.
\end{equation*}
Since the symmetrization map $\upsigma:\msS((\mfz_\msZ^+)^\ast)\iso \msS(\mfz_\msZ^+){\vphantom{)}}^\star$, defined on $\msS^n((\mfz_\msZ^+)^\ast)$ by
\begin{equation*}
\upsigma(f_1f_2\cdots f_n)(z_1z_2\cdots z_m)=\delta_{n,m}\sum_{\pi\in S_n}\prod_{j=1}^n f_j(z_{\pi(j)}) \quad \forall \; z_i\in \mfz_\msZ^+,
\end{equation*}
is an isomorphism of Hopf algebras over $\C$, it extends trivially to an isomorphism of topological Hopf algebras $\upsigma:\msS((\mfz_\msZ^+)^\ast)[\![\hbar]\!]\iso \msS(\mfz_\msZ^+){\vphantom{)}}^\star[\![\hbar]\!]$. Collecting all of the above facts, we obtain the following. 
\begin{corollary}\label{C:S+=S-}
For each non-degenerate bilinear form $\langle\,,\,\rangle:\mfz_\msZ^-\times \mfz_\msZ^+\to \C$, there is an isomorphism 
of topological Hopf algebras
$
\upvartheta: \msS(\mfz_\msZ^-)[\![\hbar]\!]\iso \msS(\mfz_\msZ^+)[\![\hbar]\!]{\vphantom{)}}^\bullet
$
satisfying
\begin{equation*}
\upvartheta(z_1^-z_2^-\cdots z_n^-)(\hbar^m z_1^+z_2^+\cdots z_m^+)=\delta_{n,m}\sum_{\pi\in S_n}\prod_{j=1}^n \langle z_j^-,z_{\pi(j)}^+\rangle \quad \forall \; z_i^\pm\in \mfz_\msZ^\pm.
\end{equation*}
In particular, $\msS_\hbar(\mfz_V^-)$ is the quantized enveloping algebra dual of $\msS_\hbar(\mfz_V^+)$.
\end{corollary}
%
%

\subsection{Quasitriangularity}\label{ssec:Quasi}
With Lemma \ref{L:gl(V)^g} and Corollary \ref{C:S+=S-} at our disposal, we are now prepared to formulate and prove the main result of this section. 
\begin{theorem}\label{T:Quasi}
$\UR{\mfg}$ is quasitriangular if and only if the underlying $\mfg$-module $V$ is multiplicity free. In this case, one has 
\begin{equation*}
\UR{\mfg}\cong D(\UR{\mfb})\cong D(\Uh{\mfb}\otimes \msS_\hbar(\mfz_V^+)).
\end{equation*}
\end{theorem}
\begin{proof}
Suppose that $\UR{\mfg}$ is quasitriangular. Then, by Theorem \ref{T:URg-main}, $D(\Uh{\mfb})\otimes \msS_\hbar(\mfz_V^+)\otimes \msS_\hbar(\mfz_V^-)$ is as well.  Consider the Hopf subalgebra $\msS_\hbar(\mfz_V^+)$. Since it is both a Hopf algebra and contained in the center of $ D(\Uh{\mfb})\otimes \msS_\hbar(\mfz_V^+)\otimes \msS_\hbar(\mfz_V^-)$, the quasitriangularity assumption implies that it is cocommutative. Indeed, we have 
\begin{equation*}
(\Delta-\Delta^{\mathrm{op}})(z)=\mcR\Delta(z)\mcR^{-1}-\Delta^{\mathrm{op}}(z)=0 \quad \forall\; z\in \msS_\hbar(\mfz_V^+),
\end{equation*}
where $\mcR$ is the associated universal $R$-matrix of $D(\Uh{\mfb})\otimes \msS_\hbar(\mfz_V^+)\otimes \msS_\hbar(\mfz_V^-)$. 
 By Lemma \ref{L:gl(V)^g}, this is only possible if $V$ is multiplicity free.

Conversely, if $V$ is multiplicity free, then $\msS_\hbar(\mfz_V^+)$ and $\msS_\hbar(\mfz_V^-)$ are both cocommutative (by Lemma \ref{L:gl(V)^g}), and the universal $R$-matrix $R^D$ of $D(\Uh{\mfb})$ defines a quasitriangular structure on $D(\Uh{\mfb})\otimes \msS_\hbar(\mfz_V^+)\otimes \msS_\hbar(\mfz_V^-)$, for instance. Hence, $\UR{\mfg}$ is quasitriangular. 

The last assertion of the theorem follows from the observation that, when $V$ is multiplicity free,  $\UR{\mfg}\cong D(\Uh{\mfb})\otimes \msS_\hbar(\mfz_V^+)\otimes \msS_\hbar(\mfz_V^-)$ satisfies the defining properties \eqref{DUha:1}--\eqref{R:can} of the quantum double $D(\Uh{\mfb}\otimes \msS_\hbar(\mfz_V^+))\cong D(\UR{\mfb})$ spelled out in Section \ref{ssec:DUhb}. In detail, first observe that
\begin{equation*}
\chk{(\Uh{\mfb}\otimes \msS_\hbar(\mfz_V^+))}\cong \chk{\Uh{\mfb}}\otimes \chk{\msS_\hbar(\mfz_V^+)}\cong \chk{\Uh{\mfb}}\otimes\msS_\hbar(\mfz_V^-).
\end{equation*} 
where the last isomorphism is due to Corollary \ref{C:S+=S-}, and depends on a fixed choice of perfect pairing $\langle\,,\,\rangle:\mfz_\msZ^-\times \mfz_\msZ^+\to \C$; see Corollary \ref{C:S+=S-}. It follows that the Hopf algebra $D(\Uh{\mfb})\otimes \msS_\hbar(\mfz_V^+)\otimes \msS_\hbar(\mfz_V^-)$ satisfies the defining properties \eqref{DUha:1} and \eqref{DUha:2} of $D(\Uh{\mfb}\otimes \msS_\hbar(\mfz_V^+))$. As for the property \eqref{R:can}, the canonical element
\begin{equation*}
\mcR\in (\Uh{\mfb}\otimes \msS_\hbar(\mfz_V^+))\otimes (\Uh{\mfb}\otimes \msS_\hbar(\mfz_V^+))^\ast \subset D(\Uh{\mfb}\otimes \msS_\hbar(\mfz_V^+))^{\otimes 2}
\end{equation*}
admits the factorization $\mcR=R^D e^{\hbar \Omega_\msZ}$, where $\Omega_\msZ\in \mfz_\msZ^+\otimes \mfz_\msZ^-\subset \msS_\hbar(\mfz_V^+)\otimes \msS_\hbar(\mfz_V^-)$ is the canonical element associated to $\langle\,,\,\rangle$. Since $\mfz_\msZ^+$ and $\mfz_\msZ^-$ consist of primitive central elements and $R^D$ is a universal $R$-matrix for $D(\Uh{\mfb})$, the element $\mcR$
indeed defines a quasitriangular structure on $D(\Uh{\mfb})\otimes \msS_\hbar(\mfz_V^+)\otimes \msS_\hbar(\mfz_V^-)$, as required. \qedhere
\end{proof}
\begin{remark}\label{R:Quasi}
The problem of characterizing the quasitriangularity of $R$-matrix algebras defined similarly to $\UR{\mfg}$ (or, more precisely, $\UR{\mfg}^\prime$) has been considered from several different perspectives, perhaps most notably in the work of Majid; see Corollaries 4.1.8--4.1.9  and Lemma 4.1.10 of \cite{Majid-book}, in addition to \cites{Majid-90c,Majid-90d,Majid-90b}.

 The characterization provided by Theorem \ref{T:Quasi} is particularly natural from the point of view Lie bialgebra quantization. Indeed, $\UR{\mfg}$ is a quantization of the Lie bialgebra $\mfg_\scriptr$ (see Theorems \ref{T:g-main} and \ref{T:URg-main}) which, as a Lie algebra, coincides with the trivial central extension $\mfg\oplus (\mfh\oplus \mfz_V^+\oplus \mfz_V^-)$. By Proposition 3.13 and Remark 3.14 of \cite{FarJan-13}, such an extension is a coboundary precisely when the Lie cobracket $\delta_\scriptr$ annihilates $\mfh\oplus \mfz_V^+\oplus \mfz_V^-$. As the central copy of $\mfh$ satisfies $\delta_\scriptr(\mfh)=0$, this occurs exactly when $\mfz_V^\pm$ is trivial as a Lie coalgebra or, equivalently, precisely when $\mfgl(V)^\mfg$ is commutative. 
\end{remark}
%
%
\subsection{From diagonal entries to central elements}\label{ssec:Z-find}
To conclude this section, we obtain an explicit description of the coefficients of the central matrices $\Upsigma^\pm$ defined in Theorem \ref{T:URg->DUhb} in terms of the diagonal entries of $\mrL^\pm$ and $\mrT^\pm$, under the hypothesis that $V$ is multiplicity free. 

In what follows, we will denote the composition factors of $V$ by $\{V_j\}_{j\in \mcJ}$; note that these are pairwise non-isomorphic by assumption. In addition, for each $j\in \mcJ$ and $\lambda \in \mfh^\ast$, we let $V_{j,\lambda}$ denote the $\lambda$-weight space of $V_j$, so that
\begin{equation*}
V_{j,\lambda}=\{v\in V_j: h\cdot v=\lambda(h) v\quad \forall \; h\in \mfh\}\subset V_j.
\end{equation*}

Let $\{\mathbf{1}_j\}_{j\in \mcJ}\subset \mfgl(V)^\mfg$ be the basis consisting of the orthogonal idempotents $\mathbf{1}_j:V\onto V_j$ associated to the decomposition $V=\bigoplus_{j\in \mcJ}V_j$. We may then write $\Upsilon^{-1}(\mathsf{Z}^\pm)=\sum_{j\in \mcJ} \mathbf{1}_j \otimes \mathsf{z}_j^\pm$, where $\mathsf{Z}^\pm$ is as in \eqref{def:msZ} and $\Upsilon$ is the isomorphism of Theorem \ref{T:URg-main}. It follows that
\begin{equation*}
I+\hbar \Upsigma^\pm=q^{\sum_j \mathbf{1}_j\otimes \mathsf{z}_j^\pm}=\sum_j \mathbf{1}_j \otimes q^{\msz_j^\pm}.
\end{equation*}
In particular, $\{\mathsf{z}_j^\pm\}_{j\in \mcJ}$ are primitive, central elements in $\UR{\mfg}$ which topologically generate a Hopf algebra isomorphic to $\msS(\mfz_\Upsigma^+ \oplus \mfz_\Upsigma^- )[\![\hbar]\!]$, and generate the kernel of $\dot{\Upsilon}:\UR{\mfg}\onto D(\Uh{\mfb})$ as an ideal; see Theorem \ref{T:URg->DUhb} and Section \ref{ssec:g-invar-comm} above. 
\begin{corollary}\label{C:Z-compute}
The $\mfh$-invariant part $\mrL_0^\pm$ of $\mrL^\pm$ admits the block diagonal form 
\begin{equation*}
\mrL_0^\pm=\sum_{j,\lambda}\mathscr{l}_{j,\lambda}^\pm \cdot \mathrm{id}_{V_{j,\lambda}}=I+\hbar\sum_{j,\lambda}\mathrm{t}_{j,\lambda}^\pm \cdot \mathrm{id}_{V_{j,\lambda}}, 
\end{equation*}
where the summation runs  over all pairs $(j,\lambda)\in \mcJ\times \mfh^\ast$ such that $V_{j,\lambda}\neq 0$, and $\{\mathscr{l}_{j,\lambda}=1+\hbar\mathrm{t}_{j,\lambda}^\pm\}_{j,\lambda}$ are grouplike elements generating a commutative subalgebra of $\UR{\mfg}$. Moreover:
\begin{equation*}
\prod_{\lambda\in \mfh^\ast} (\mathscr{l}_{j,\lambda}^\pm)^{\msd_{j,\lambda}}=q^{\pm \dim V_j \cdot \msz_j^\pm} \quad \forall \; j \in \mcJ,
\end{equation*}
where $\msd_{j,\lambda}=\dim V_{j,\lambda}$. In particular, $\msz_j^\pm$ is given explicitly by 
\begin{equation*}
\mathsf{z}_j^\pm=\pm\sum_{\lambda\in \mfh^\ast} \frac{2\msd_{j,\lambda}}{\dim V_j \cdot \hbar} \log(1+\hbar \mathrm{t}_{j,\lambda}^\pm) \quad \forall \; j\in \mcJ.
\end{equation*}
\end{corollary}
\begin{proof}
We have seen in the proof of Lemma \ref{L:Theta} that
\begin{equation*}
\Upsilon(\mrL_0^\pm)=q^{\mp\sum_{i\in \mbI} \pi(\omega_i^\vee)\otimes h_i^\pm}\mathring{\mfZ}^\pm=q^{\mp\sum_{i\in \mbI} \pi(\omega_i^\vee)\otimes h_i^\pm}\sum_{j\in \mcJ}\mathbf{1}_j\otimes q^{\pm \Upsilon(\msz_j^\pm)}.
\end{equation*}
Hence, on each weight space $V_{j,\lambda}$ of $V_j$, we have
\begin{equation*}
\Upsilon(\mrL_0^\pm|_{V_{j,\lambda}})=q^{\mp\sum_{i\in \mbI} \lambda(\omega_i^\vee)h_i^\pm}q^{\pm \Upsilon(\msz_j^\pm)} \cdot \mathrm{id}_{V_{j,\lambda}}.
\end{equation*}
The first assertion of the corollary therefore follows by taking $\mathscr{l}_{j,\lambda}^\pm$ and $\mathrm{t}_{j,\lambda}^\pm$ to be the elements uniquely determined by the formulas 
\begin{equation}\label{l-j-lambda}
\mathscr{l}_{j,\lambda}^\pm=1+\hbar \mathrm{t}_{j,\lambda}^\pm=\Upsilon^{-1}\left(q^{\mp\sum_{i\in \mbI} \lambda(\omega_i^\vee)h_i^\pm}\right)q^{\pm \msz_j^\pm}.
\end{equation}
The second assertion of the corollary now follows from the observation that, for each $j\in \mcJ$, we have
\begin{equation*}
\prod_{\lambda\in \mfh^\ast} \Upsilon(\mathscr{l}_{j,\lambda}^\pm q^{\mp \msz_j^\pm} )^{\msd_{j,\lambda}}
=
\prod_{\lambda\in \mfh^\ast} q^{\mp\sum_{i\in \mbI} \msd_{j,\lambda} \lambda(\omega_i^\vee)h_i^\pm}=q^{\mp\sum_{i\in \mbI} \mathsf{Tr}_{V_j}(\pi(\omega_i^\vee))\cdot h_i^\pm}=1. \qedhere
\end{equation*}
\end{proof}
\begin{remark}\label{R:weight-basis}
Suppose that the underlying basis $\{v_i\}_{i\in \mcI}$ of $V$ is taken to be a weight basis which is compatible with the $\mfg$-module decomposition $V=\bigoplus_{j\in \mcJ} V_j$. That is, $\mcI$ admits a partition $\mcI=\bigsqcup_{j}\mcI_j=\bigsqcup_{j,\lambda}\mcI_{j,\lambda}$ for which $\{v_i\}_{i\in \mcI_{j}}$ and  $\{v_i\}_{i\in \mcI_{j,\lambda}}$ are bases of $V_j$ and $V_{j,\lambda}$, respectively, and is ordered so that, for each $j\in \mcJ$, the elements of $\mcI_{j,\lambda}$ precede those in $\mcI_{j,\gamma}$ if $\gamma-\lambda\in \msQ_+$.   Theorem \ref{T:URg-main} then implies that, with respect to the basis, $\mrL^+$ and $\mrL^-$ are lower and upper triangular matrices, respectively, and $\mrL_0^+$ and $\mrL_0^-$ are the diagonal factors in their Gaussian decompositions 
\begin{equation*}
\mathrm{L}^+=\mathrm{L}_0^+\left(I+ \hbar\mathsf{X}^+\right) \quad \text{ and }\quad \mathrm{L}^-=\left(I+ \hbar\mathsf{X}^-\right)\mathrm{L}_0^-,
\end{equation*}
where $\msX^+$ and $\msX^-$ are strictly lower and upper triangular, respectively. Then, by Corollary \ref{C:Z-compute}, the diagonal entries $\mathscr{l}_{ii}^\pm$ and $\mathscr{l}_{kk}^\pm$ of $\mrL^\pm$ (resp. $t_{ii}^\pm$ and $t_{kk}^\pm$ of $\mrT^\pm$) coincide for any $i,k\in\mcI_{j,\lambda}$, and are equal to $\mathscr{l}_{j,\lambda}^\pm$ (resp. $\mathrm{t}_{j,\lambda}^\pm$). In particular, one has 
\begin{gather*}
\prod_{i\in \mcI_{j,\lambda}} \mathscr{l}_{ii}^\pm=q^{\pm \dim V_j \cdot \msz_j^\pm}, \\
\mathsf{z}_j^\pm=\pm\frac{2}{\dim V_j} \sum_{k>0}\frac{(-\hbar)^{k-1}}{k}\mathsf{Tr}_{V_j}((\mrT^{\pm})^k),
\end{gather*}
for each $j\in \mcJ$. By Theorem \ref{T:URg->DUhb}, $D(\Uh{\mfb})$ is isomorphic to the unital associative $\C[\![\hbar]\!]$-algebra topologically generated by $\{t_{ik}^\pm\}_{i,k\in \mcI}$, subject to the relations \eqref{URg-T-tri}--\eqref{URg-RTT:2} and
\begin{equation*}
\sum_{k>0}\frac{(-\hbar)^{k-1}}{k}\mathsf{Tr}_{V_j}((\mrT^{\pm})^k)=0 \quad \forall \; j\in \mcJ.
\end{equation*}
Moreover, by Theorem \ref{T:URg->Uhg} and \eqref{l-j-lambda}, to recover $\Uh{\mfg}$ as a quotient of $\UR{\mfg}$ one imposes the additional family of relations 
\begin{equation*}
\hbar t_{ii}^+t_{ii}^-=-(t_{ii}^+ + t_{ii}^-)=\hbar t_{ii}^-t_{ii}^+ \quad \forall \; i\in \mcI. 
\end{equation*}
In particular, in this quotient one has the familiar identities 
\begin{equation*}
\prod_{i\in \mcI_{j,\lambda}} \mathscr{l}_{ii}^\pm=1\quad \text{ and }\quad 
\mathscr{l}_{kk}^+\mathscr{l}_{kk}^-=1=\mathscr{l}_{kk}^-\mathscr{l}_{kk}^+
\end{equation*}
for all $k\in \mcI$, and pairs $(j,\lambda)\in \mcJ\times \mfh^\ast$ for which $V_{j,\lambda}\neq 0$ (\textit{cf}. \cite{FRT}*{\S2.2}).
\end{remark}
\begin{remark}\label{R:zero-weight}
If $\lambda=0$ is a weight of $V_j$, then by \eqref{l-j-lambda}, we have 
\begin{equation*}
\mathscr{l}_{j,\lambda}^\pm=\Upsilon^{-1}\left(q^{\mp\sum_{i\in \mbI} \lambda(\omega_i^\vee)h_i^\pm}\right)q^{\pm \msz_j^\pm}=q^{\pm \msz_j^\pm}.
\end{equation*}
For example, if $V$ is the adjoint representation of $\mfg$, then we may write $\mathscr{l}_{\lambda}^\pm=\mathscr{l}_{j,\lambda}^\pm$ and $\msz^\pm=\msz_j^\pm$, and the diagonal part $\mrL_0^\pm$ of $\mrL^\pm$ decomposes as
\begin{equation*}
\mrL_0^\pm = q^{\pm \msz^\pm}\cdot \mathrm{id}_\mfh + \sum_{\alpha\in \Root^+} (\mathscr{l}_{\alpha}^\pm \cdot \mathrm{id}_{\mfg_\alpha} + \mathscr{l}_{-\alpha}^\pm \cdot \mathrm{id}_{\mfg_{-\alpha}}).
\end{equation*}
\end{remark}
%

\section{The vector representation of \texorpdfstring{$\mathfrak{sl}_n$}{sl\_n}}\label{sec:Ex}

In this final section, we narrow our focus to the special case where $\mcV$ is the vector representation of  $\Uh{\mfsl_n}$ (see Section \ref{ssec:V-sln}) with the intention of providing a detailed example of our main results and illustrating how they recover some well-known constructions. 
\subsection{The natural representations of $\mfsl_n$ and $\Uh{\mfsl_n}$}\label{ssec:V-sln} Let us now restrict our attention to the special linear Lie algebra $\mfg=\mfsl_n$, where $n\geq 2$. We take $\mbI=\{1,\ldots,n-1\}$, so that the Cartan matrix $(a_{ij})_{i,j\in \mbI}$ is given by 
\begin{equation*}
a_{ij}
=2\delta_{ij}-\delta_{i+1,j}-\delta_{i-1,j} \quad \forall \; i,j \in \mbI.
\end{equation*}
In addition, we henceforth fix $V=\C^n$ to be the natural representation of $\mfg=\mfsl_n$, with $\{v_i\}_{i\in \mcI}\subset \C^n$ taken to be its standard basis. That is, $\mcI=\mbI\cup\{n\}$ and $v_i=e_i$ for each $1\leq i\leq n$.   The associated algebra homomorphism $\pi:U(\mfsl_n)\to \End(V)$ outputs the standard realization of $\mfsl_n$, and is given explicitly by 
\begin{equation*}
\pi(h_i)=E_{ii}-E_{i+1,i+1}, \quad \pi(x_i^+)=E_{i,i+1}, \quad \pi(x_i^-)=E_{i+1,i} \quad \forall \; i\in \mbI,
\end{equation*}
where $h_i$ and $x_i^\pm$ are as in Sections \ref{ssec:g} and \ref{ssec:Db-Chev}, respectively. 

In this case, the action of $\mfg=\mfsl_n$ on the general linear Lie algebra $\mfgl(V)=\mfgl_n$ introduced in Section \ref{ssec:gl(V)} coincides with the standard adjoint action of $\mfsl_n$ on $\mfgl_n$. In particular, one has the $\mfsl_n$-module decomposition $\mfgl(V)=\mathrm{ad}(\mfsl_n)\oplus \C I$, with weight space decomposition 
\begin{equation}\label{sln-gl(V)}
\mfgl(V)=\bigoplus_{i,j}\mfgl(V)_{\eps_i-\eps_j},
\end{equation}
where $\mfgl(V)_0=\mfgl(V)^\mfh=\mfh\oplus \C I$  is the space of all diagonal matrices, and 
\begin{equation*}
\mfgl(V)_{\eps_i-\eps_j}=\Hom(V_{\eps_j},V_{\eps_i})=\C E_{ij} \quad \forall \; i\neq j.
\end{equation*}
 Here $\{\eps_i\}_{i\in \mcI}\subset \mfh^\ast$ are defined by $\eps_i(h_j)=\delta_{i,j}-\delta_{i,j+1}$ for all $j\in \mbI$, so that $\Root^+=\{\eps_i-\eps_j:i<j\}$ is the standard set of positive roots. 

Consider now $\Uh{\mfsl_n}$. By the results recalled in Section \ref{ssec:Uhg-fdreps}, there exists a (unique, up to isomorphism) $\Uh{\mfsl_n}$-module structure on $\mcV=\C^n[\![\hbar]\!]$ with the property that the associated algebra homomorphism 
\begin{equation*}
\pi_\hbar:\Uh{\mfsl_n}\to \End_{\C[\![\hbar]\!]}(\mcV)\cong \End(\C^n)[\![\hbar]\!]
\end{equation*}
has semiclassical limit $\bar{\pi}_\hbar=\pi$. In fact, the assignment
\begin{equation*}
\pi_\hbar(h_i)=E_{ii}-E_{i+1,i+1}, \quad \pi_\hbar(E_i)=E_{i,i+1},\quad \pi_\hbar(F_i)=E_{i+1,i} \quad \forall \;i \in \mbI
\end{equation*}
uniquely extends to an algebra homomorphism with the desired property, as is readily verified using Definition \ref{D:Uhg}; see \cite{CPBook}*{Ex.~8.3.17}, for instance. Note that $\pi_\hbar$ also satisfies the auxiliary conditions $\pi_\hbar|_{\mfh}=\pi|_\mfh$ and $\pi_\hbar(\Uh{\mfsl_n})\subset \pi(U(\mfsl_n))[\![\hbar]\!]$ imposed in Section \ref{ssec:URg-def}. Moreover, in this particular case, the evaluation $\mrR_\pi=(\pi_\hbar\otimes \pi_\hbar)(R)$ of the universal $R$-matrix of $\Uh{\mfsl_n}$ is not difficult to compute using the explicit factorizations established in \cite{KR90} and \cite{LeSo90}, and is given  by 
\begin{equation}\label{sln:R}
q^{1/n}\mrR_\pi=\sum_{i,j}q^{\delta_{ij}}E_{ii}\otimes E_{jj} +(q-q^{-1})\sum_{i<j} E_{ij}\otimes E_{ji}.
\end{equation}
We refer the reader to \cite{CPBook}*{\S8.3.G} or \cite{KS-book}*{\S8.4.2} for a detailed derivation of this formula, in addition to equation (3.7) of \cite{FRT}*{Thm.~18} and \cite{Jimbo86}.

\subsection{The $R$-matrix realization of $\Uh{\mfsl_n}$ associated to $\C^n$}
We now focus on illustrating some of the main results of this paper in the special case where $\mfg$, $V$, $\{v_i\}_{i\in \mcI}$ and $\mrR_\pi$ are as in the previous subsection. 

To begin, note that Definition \ref{D:URg} for $\UR{\mfsl_n}$ collapses to the following concrete definition: $\UR{\mfsl_n}$ is the unital associative $\C[\![\hbar]\!]$-algebra topologically  generated by $\{t_{ij}^\pm\}_{i\in \mcI}$, subject only to the relations 
\begin{gather}
t_{ij}^+=0=t_{ji}^- \quad \forall\quad i>j,\label{sln-triangle}\\
\begin{aligned}\label{sln:RLL}
q^{\delta_{ik}}t_{ij}^{\sigma_1}&t_{kl}^{\sigma_2}-q^{\delta_{j l}}t_{kl}^{\sigma_2}t_{ij}^{\sigma_1}\\
=\,&\delta_{kl}(q^{\delta_{jk}}-q^{\delta_{ik}})t_{ij}^{\sigma_1}+\delta_{ij}(q^{\delta_{il}}-q^{\delta_{ik}})t_{kl}^{\sigma_2}
\\
&+
\frac{(q-q^{-1})}{\hbar}\left( \delta_{l<j}\mathscr{l}_{kj}^{\sigma_2}t_{il}^{\sigma_1}-\delta_{i<k}t_{kj}^{\sigma_1}\mathscr{l}_{il}^{\sigma_2}+\delta_{i<j}(\delta_{il}t_{kj}^{\sigma_2}-\delta_{kj}t_{il}^{\sigma_2})\right),
\end{aligned}
\end{gather}
where $(\sigma_1,\sigma_2)$ takes value $(\pm,\pm)$ or $(+,-)$, $\mathscr{l}_{ij}^\pm=\delta_{ij}+\hbar t_{ij}^\pm$ for all $i,j\in\mcI$, and we recall that $q=e^{\hbar/2}$.
%

Indeed, from \eqref{sln-gl(V)} and the remarks that follows it, the projection $\mrT_\lambda^\pm$ of $\mrT^\pm$ onto $\mfgl(V)_\lambda\otimes_\C \UR{\mfsl_n}$ is zero unless $\lambda=\eps_i-\eps_j$ for some $1\leq i,j\leq n$, in which case 
\begin{equation*}
\mrT_0^\pm=\sum_{i=1}^n E_{ii}\otimes t_{ii}^\pm \quad \text{ and }\quad \mrT^\pm_{\eps_i-\eps_j}=E_{ij}\otimes t_{ij}^\pm \quad \forall \; i\neq j.
\end{equation*}
Hence, the set of relations \eqref{sln-triangle} is equivalent to the triangularity relations \eqref{URg-T-tri}.
\begin{remark}
In particular, $\mrT^+$ and $\mrT^-$ are upper and lower triangular, respectively. 
Note that this is opposite to the situation described in Remark \ref{R:weight-basis}; this is due to the fact that the natural ordering on the standard basis of $\C^n$ is dual to the partial ordering on $\mfsl_n$-weights: Indeed, if $i>j$ then $\eps_i-\eps_j\in \msQ_-$ and so $\eps_i<\eps_j$.
\end{remark}
Similarly, the relations \eqref{URg-RTT:1} and \eqref{URg-RTT:2} are equivalent to \eqref{sln:RLL} with $(\sigma_1,\sigma_2)=(\pm,\pm)$ and $(\sigma_1,\sigma_2)=(+,-)$, respectively. This is easily seen by inputting $\dot\mrR_\pi=\hbar^{-1}(\mrR_\pi-1)$ into \eqref{URg-RTT:1} and \eqref{URg-RTT:2}, with $\mrR_\pi$ as in \eqref{sln:R}, and then taking the coefficient of $E_{ij}\otimes E_{kl}$ in both relations and simplifying.   

Next, since $\C^n$ is an irreducible representation of $\mfsl_n$, the space of $\mfg$-invariants $\mfgl(V)^\mfg$ coincides with $\C I$. Theorem \ref{T:URg-main} therefore outputs an isomorphism of topological Hopf algebras 
\begin{equation*}
\UR{\mfsl_n}\cong D(\Uh{\mfb})\otimes\C[\msz^+,\msz^-][\![\hbar]\!],
\end{equation*}
where $\mfb\subset \mfsl_n$ is the standard Borel subalgebra of upper triangular matrices in $\mfsl_n$, and we have used that $\msS_\hbar(\mfz_V^\pm)\cong \msS(\mfz_\msZ^\pm)[\![\hbar]\!]\cong \C[\msz^\pm][\![\hbar]\!]$; see Section \ref{ssec:g-invar-comm}. Moreover, by Corollary \ref{C:Z-compute} and Remark \ref{R:weight-basis}, the primitive central elements $\msz^\pm$ are uniquely determined in $\UR{\mfsl_n}$ by the formulas 
\begin{equation*}
q^{\pm n\msz^\pm}=\mathscr{l}_{11}^\pm \mathscr{l}_{22}^\pm \cdots \mathscr{l}_{nn}^\pm=\prod_{i=1}^n \mathscr{l}_{ii}^\pm. 
\end{equation*}

In addition, by Theorem \ref{T:Quasi}, $\UR{\mfsl_n}$ is isomorphic to the quantum double of $\Uh{\mfb}\otimes \C[\msz^+]$, which itself is isomorphic to the subalgebra of $\UR{\mfsl_n}$ topologically generated by $\{t_{ij}^+\}_{i,j\in \mcI}$; see Corollary \ref{C:URb}.

Finally, as a consequence of the above conclusions, the formulas for the Hopf algebra structure maps given in Theorem \ref{T:URg-main}, the statement of Theorem \ref{T:URg->Uhg}, and Remark \ref{R:weight-basis}, we obtain the following characterization of the quantized enveloping algebra $\Uh{\mfsl_n}$.
\begin{corollary}\label{C:URg->Uhsln}
$\Uh{\mfsl}_n$ is isomorphic to the unital, associative $\C[\![\hbar]\!]$-algebra topologically generated by $\{t_{ij}^\pm\}_{i,j\in \mcI}$, subject to the relations 
\begin{equation}\label{Uhsln-def}
\begin{gathered}
t_{ij}^+=0=t_{ji}^- \quad \forall\quad i>j,\\
\begin{aligned}
q^{\delta_{ik}}t_{ij}^{\sigma_1}&t_{kl}^{\sigma_2}-q^{\delta_{j l}}t_{kl}^{\sigma_2}t_{ij}^{\sigma_1}\\
=\,&\delta_{kl}(q^{\delta_{jk}}-q^{\delta_{ik}})t_{ij}^{\sigma_1}+\delta_{ij}(q^{\delta_{il}}-q^{\delta_{ik}})t_{kl}^{\sigma_2}
\\
&+
\frac{(q-q^{-1})}{\hbar}\left( \delta_{l<j}\mathscr{l}_{kj}^{\sigma_2}t_{il}^{\sigma_1}-\delta_{i<k}t_{kj}^{\sigma_1}\mathscr{l}_{il}^{\sigma_2}+\delta_{i<j}(\delta_{il}t_{kj}^{\sigma_2}-\delta_{kj}t_{il}^{\sigma_2})\right),
\end{aligned}
\\[5pt]
\hbar t_{ii}^+t_{ii}^- = -(t_{ii}^+ + t_{ii}^-)=\hbar t_{ii}^-t_{ii}^+,\\
\sum_{i=1}^n \sum_{k>0}\frac{1}{k}(-\hbar)^{k-1}( t_{ii}^\pm)^k=0,
\end{gathered}
\end{equation}
where $(\sigma_1,\sigma_2)$ takes value $(\pm,\pm)$ or $(+,-)$, and $\mathscr{l}_{ij}^\pm=\delta_{ij}+\hbar t_{ij}^\pm$ for all $i,j\in\mcI$. Moreover, the coproduct $\Delta$, counit $\veps$ and antipode $S$ on $\Uh{\mfsl_n}$ are determined by
\begin{gather*}
\Delta(t_{ij}^\pm)=t_{ij}^\pm \otimes 1 + 1\otimes t_{ij}^\pm + \hbar \sum_{a=1}^n t_{ia}^\pm \otimes t_{aj}^\pm, \quad \veps(t_{ij}^\pm)=0, \\
S(t_{ij}^\pm)=-t_{ij}^\pm -\hspace{-1em}\sum_{\substack{b> 0\\ 1\leq a_1,\ldots,a_{b}\leq n}} \hspace{-1em}(-\hbar)^{b} t_{i,a_1}^\pm t_{a_1,a_2}^\pm\cdots t_{a_{b},j}^\pm.
\end{gather*}
\end{corollary}
%
%
\subsection{Remarks}
We conclude our analysis of the $R$-matrix algebra $\UR{\mfsl_n}$ associated to the natural representation $V=\C^n$ of $\mfsl_n$ with a sequence of remarks:

\begin{enumerate}\setlength{\itemsep}{3pt}
\item The topological generators $\{\mathscr{l}_{ij}^\pm\}_{i,j\in \mcI}$ of the quantum formal series Hopf algebra $(\Uh{\mfsl_n})^\prime\subset \Uh{\mfsl_n}$ (see Section \ref{ssec:Uhb} and Remark \ref{R:URg-main}) satisfy the algebraic relations 
\begin{gather*}
\mathscr{l}_{ij}^+=0=\mathscr{l}_{ji}^- \quad \forall\quad i>j,\\
q^{\delta_{ik}}\mathscr{l}_{ij}^{\sigma_1}\mathscr{l}_{kl}^{\sigma_2}
-q^{\delta_{j l}}\mathscr{l}_{kl}^{\sigma_2}\mathscr{l}_{ij}^{\sigma_1}
=(q-q^{-1})\left( \delta_{l<j}\mathscr{l}_{kj}^{\sigma_2}\mathscr{l}_{il}^{\sigma_1}-\delta_{i<k}\mathscr{l}_{kj}^{\sigma_1}\mathscr{l}_{il}^{\sigma_2}\right),
\\
\mathscr{l}_{ii}^+\mathscr{l}_{ii}^-=1=\mathscr{l}_{ii}^-\mathscr{l}_{ii}^+,\\
\mathscr{l}_{11}^\pm \mathscr{l}_{22}^\pm \cdots \mathscr{l}_{nn}^\pm=1,
\end{gather*}
where $(\sigma_1,\sigma_2)$ takes value $(\pm,\pm)$ or $(+,-)$. In addition, one has $S(\mrL^\pm)=(\mrL^\pm)^{-1}$, while 
\begin{equation*}
\Delta(\mathscr{l}_{ij}^\pm)=\sum_{a=1}^n \mathscr{l}_{ia}^\pm \otimes \mathscr{l}_{aj}^\pm \quad \text{ and }\quad \veps(\mathscr{l}_{ij}^\pm)=1
\quad \forall\; i,j\in \mcI.
\end{equation*}
It is these relations that have predominantly appeared in the literature, rather than those of Corollary \ref{C:URg->Uhsln}; see \S2.2 and Theorem 12 of \cite{FRT}, in addition to \cite{KS-book}*{\S8.5}, \cite{GoMo10}*{\S2.1}, \cite{MRS}*{\S2} and \cite{DF93}*{\S2}, for example. We emphasize that they \textit{are not} defining relations for the quantized enveloping algebra $\Uh{\mfsl_n}$. However, they are defining relations for its $\C(q)$ form $U_q(\mfsl_n)$, and this is the context in which they have primarily arisen in the literature.

\item Removing the last relation in \eqref{Uhsln-def} of Corollary \ref{C:URg->Uhsln} yields a presentation of the quantized enveloping algebra $\Uh{\mfgl_n}$ of  the general linear Lie algebra $\mfgl_n$; see \cite{Jimbo86}*{\S2}  and \cite{DF93}*{Thm.~2.1}. It may also be realized as the subalgebra of $\UR{\mfsl_n}$ consisting of those elements fixed by all automorphisms of the form $\mathring{\upchi}_C\circ \upgamma_h$, for $h\in \mfh$ and $C\in \mathrm{GL}_I(V)^\mfg \cong 1+\hbar\C[\![\hbar]\!]$. Here $\upgamma_h$ is as in \eqref{upgamma-h} and we have set $\mathring{\upchi}_C=\upchi_\mathscr{C}$ with $\mathscr{C}=(C,C^{-1})$ (see Proposition \ref{P:aut}). Explicitly,   $\mathring{\upchi}_C\circ \upgamma_h$ is uniquely determined by 
\begin{equation*}
\mrL^\pm \mapsto q^{-\pi(h)/2} \mrL^\pm q^{-\pi(h)/2} \cdot C^{\pm 1}.
\end{equation*}
\item Finally, we note that it is possible to repeat the analysis carried out in this section in the special case where $\mfg$ is of symplectic or orthogonal type and $V$ is its vector representation. In this setting, all the underlying data is available; see (1.9) and (3.7) in \cite{FRT}, \cite{CPBook}*{\S8.3.G}, \cite{KS-book}*{\S8.5} and the articles \cites{JLM20-c,JLM20-b,GRWqLoop}. It is worth emphasizing that in this setting the underlying $V$ is again irreducible and all the observations from Remark \ref{R:weight-basis} apply (\textit{cf.} \cite{FRT}*{Rem.~21} and \cite{KS-book}*{Prop.~8.28}). 
\end{enumerate}


\bibliography{Yangians}

\end{document}